\documentclass[11pt]{amsart}
\usepackage[utf8x]{inputenc}
\usepackage{microtype}
\usepackage[T1]{fontenc}
\usepackage{ae,aecompl}
\usepackage{times}
\usepackage{mathrsfs}  
\usepackage[margin=3.8cm]{geometry}
\usepackage{verbatim,amsmath,amsthm,amsfonts,amssymb,latexsym,graphicx,mathtools,extpfeil,color} 
\usepackage{epstopdf,pinlabel}
\usepackage[all]{xy}
\usepackage{graphicx}
\usepackage{caption}
\usepackage{subcaption}
\usepackage{tikz-cd}
\DeclareMathAlphabet{\mathpzc}{OT1}{pzc}{m}{it}
\theoremstyle{plain}
\newtheorem{Th}{Theorem}[section]
\newtheorem{Lemma}[Th]{Lemma}
\newtheorem{Cor}[Th]{Corollary}
\newtheorem{Prop}[Th]{Proposition}

 \theoremstyle{definition}
\newtheorem{Def}[Th]{Definition}

\newtheorem{Rem}[Th]{Remark}
\newtheorem{?}[Th]{Question}

\newcommand{\im}{\operatorname{im}}

\newcommand{\la}{\langle}
\newcommand{\ra}{\rangle}

\newcommand\norm[1]{\left\lVert#1\right\rVert}

\newcommand\ind{{\rm ind}}
\newcommand\R{\mathbf{R}}
\newcommand\Z{\mathbf{Z}}
\newcommand\Q{\mathbf{Q}}

\usepackage[colorlinks,pagebackref,hypertexnames=false]{hyperref} \usepackage[alphabetic,backrefs,msc-links]{amsrefs}

\usepackage{aliascnt}
\numberwithin{equation}{section}

\makeatletter
\def\l@subsection{\@tocline{2}{0pt}{2.5pc}{5pc}{}}
\def\l@subsubsection{\@tocline{2}{0pt}{2.5pc}{5pc}{}}
\makeatother

\title{\large \bf Spectral invariants and equivariant monopole Floer homology for rational homology three-spheres}
\author{Minh Lam Nguyen}

\newcommand{\Addresses}{{
\bigskip
\footnotesize
Minh Lam Nguyen , \textsc{ Department of Mathematics,
  Washington University in St. Louis,
  St. Louis, MO 63130}\par\nopagebreak
  \textit{E-mail address}: \texttt{minhn@wustl.edu}
  }}


\usepackage{amsmath} 

\begin{document}
\maketitle
	
	\begin{abstract}
		In this paper, we study a model for $S^1$-equivariant monopole Floer homology for rational homology three-spheres via a homological device called $\mathcal{S}$-complex. Using the Chern-Simons-Dirac functional, we define an $\mathbf{R}$-filtration on the (equivariant) complex of monopole Floer homology $HM$. This $\mathbf{R}$-filtration fits $HM$ into a persistent homology theory, from which one can define a numerical quantity called the spectral invariant $\rho$. The spectral invariant $\rho$ is tied with the geometry of the underlying manifold. The main result of the papers shows that $\rho$ provides an obstruction to the existence of positive scalar curvature metric on a ribbon homology cobordism.
				
	\end{abstract}


	\tableofcontents
	
    \section{Introduction}

    Persistent homology has become an important tool in the field of topological data analysis and machine learning \cite{MR2572029, MR3774757}. One of the basic motivations for the technique is that often a large collection of data, e.g., signals and images, comes with a certain shape where intrinsic properties are typically captured most efficiently by notions from algebraic topology. Besides the applied settings, persistent homology and its various numerical-type invariants like barcodes or spectral invariants have recently played an increasingly notable role in understanding some of the purely geometrical and topological problems that arise in the area of symplectic topology/geometry, e.g., see \cite{MR3590354, MR4368349, MR4413744}.

    On the other hand, various Floer homology theories have provided an innovative approach to studying low-dimensional topology \cite{MR956166, MR2113019, MR2388043}. The basic idea behind the construction of any Floer homology theory is that it is an infinite-dimensional Morse theory applied to a functional $\mathcal{F}: \mathcal{C}(Y) \to \mathbf{R}$ defined on some infinite-dimensional configuration space $\mathcal{C}$ that takes in topological (and sometimes geometrical) data of the underlying closed smooth oriented $3$-manifold $Y$. As a Morse theory, over some field, the differential complex $C(Y)$ is generated by the critical points of $\mathcal{F}$. It can be shown that the functional $\mathcal{F}$ induces an $\mathbf{R}$-filtration on $C(Y)$, $\{C^t(Y)\}_{t\in \mathbf{R}}$, where there is a map $f_{ts}: C^t(Y) \to C^s(Y)$ whenever $t \leq s$, $f_{tt}:=\text{id}_{C^t(Y)}$, and $f_{ts} =  f_{rs}f_{tr}$. In other words, $C(Y)$ is a persistent module. From this, we can define the spectral invariant $\rho_\ast(C(Y), \mathcal{F})$, which is the infimum of all possible $t \in \mathbf{R}$ such that the induced map $ f_{t\ast}: H_\ast(C^t(Y)) \to H_\ast(C(Y))$ is non-trivial (cf. Definition \ref{Def1.2}). 
    
    Depending on the context, $\rho_\ast(C(Y), \mathcal{F})$ (and many other similarly defined numerical quantities) is an actual (diffeomorphism and homology cobordism) invariant of $Y$. For example, in instanton Floer homology theory, if $Y$ is an integral homology three-spheres, $\mathcal{C}(Y)$ is taken to be the configuration space of gauge equivalence classes of $SU(2)$-connections on $Y$ and $\mathcal{F}$ is the Chern-Simons functional $CS$, then the instanton Floer homology $I(Y)$ is a persistent module, and the various associated numerical quantities $\rho_\ast(I(Y), CS)$, $r_s(Y)$ (defined in \cite{Nozaki2023}), and $\Gamma_Y$ (defined in \cite{MR4158669}) end up being homology cobordism invariants of $Y$.

    Naturally, one should try to extend this theme to monopole Floer homology, where  $\mathcal{F}$ is a Chern-Simons-Dirac functional $\mathcal{L}$ defined on the configuration space of gauge equivalence classes of spinors and $U(1)$-connections on a rational homology three-sphere $Y$. This is, in fact, one of the main tasks we take up in the current paper. Note that, unlike $CS$, the definition of $\mathcal{L}$ depends on an auxiliary choice given by a Riemannian metric $g$ on $Y$. As a result, the associated spectral invariant $\rho_\ast(Y, g)$ is not expected to satisfy many formal properties of its counterparts on the instanton side. It turns out that $\rho_\ast(Y,g)$ is only an invariant of $(Y,g)$. However, this is a feature and not a bug. Being an invariant of $(Y,g)$, $\rho_\ast(Y,g)$ is a suitable tool to study geometry on a cobordism.

    \subsection{Main results}

    For simplicity and convenience, we work with the field $\mathbf{F} = \mathbf{F}_2$ of characteristic $2$. Let $Y$ be a closed, oriented, connected, smooth rational homology three-spheres, i.e., $b_1(Y) = 0$. There is a classical invariant of $Y$, denoted by $\mathfrak{m}(Y) \in \mathbf{Q}/\mathbf{Z}$ (cf. Subsection \ref{Sub4.2}) such that for each $spin^c$ structure $\mathfrak{s}$ over $Y$, one can define a representative element $m_{\mathfrak{s}}$ of a "chamber" in $\mathfrak{m}(Y)$. Associated to $(Y, m_{\mathfrak{s}})$, we construct a differential complex $(\widetilde{CM}(Y, m_\mathfrak{s}), \widetilde{d})$ that is a module over $H(BS^1; \mathbf{F}) = \mathbf{F}[x]/(x^2)$. This module structure will not be important for our paper. However, we mention it here to allude to the fact that $(\widetilde{CM}(Y, m_\mathfrak{s}), \widetilde{d})$ can be thought of as an $S^1$-equivariant chain complex defined on a based configuration space of gauge equivalence classes of spinors and $U(1)$-connections on $\det\,\mathfrak{s}$. In general, this structure arises whenever one works with an $S^1$-manifold (see Section \ref{Sec3} for more details).

    The complex $(\widetilde{CM}(Y, m_\mathfrak{s}), \widetilde{d})$ decomposes as
    $$\widetilde{CM}(Y, m_\mathfrak{s}) = CM(Y, m_\mathfrak{s}) \oplus CM(Y, m_\mathfrak{s})[1]\oplus \mathbf{F}$$
    Here, $CM(Y, m_\mathfrak{s})$ is a $\mathbf{Z}$-graded differential complex generated by the gauge equivalence classes of irreducible solutions of a $3$-dimensional (perturbed) Seiberg-Witten equations defined on $Y$. The last summand $\mathbf{F}$ is generated by the reducible solutions, which in this case there is only one up to gauge transformations since $b_1(Y) = 0$. One can formalize this algebraic structure and view $(\widetilde{CM}(Y, m_\mathfrak{s}), \widetilde{d})$ as an object of the category of $\mathcal{S}$-complexes (cf. Subsection \ref{Sub3.4}). We refer to $\widetilde{CM}$ as an $\mathcal{S}$-complex of monopoles.

    The construction of the $\mathcal{S}$-complex of monopoles depends on auxiliary data such as Riemannian metric $g$ on $Y$ and a generic perturbation $\eta \in i\Omega^2(Y)$. However, for an allowable range of homological degree $q \in \mathcal{I}$, the $\mathcal{S}$-chain homotopy type of $\widetilde{CM}_q(Y, m_\mathfrak{s})$ (see Subsection \ref{Sub3.4} for more details) is an invariant of $Y$.

    \begin{Th}[cf. Theorem \ref{Th4.10}, Corollary \ref{Cor4.11}]\label{mainTh1}
        For an allowable range of homological degree $q\in \mathcal{I}$, $\widetilde{HM}_q(Y, m_\mathfrak{s}; \mathbf{F}) := H_q(\widetilde{CM}(Y, m_\mathfrak{s}))$ is an invariant of $Y$. Furthermore, $HM_q(Y, m_\mathfrak{s}; \mathbf{F}) := H_q(CM(Y, m_\mathfrak{s}))$ is also an invariant of $Y$ for $q \in \mathcal{I}$.
    \end{Th}

    \begin{Rem}
        The fact that $HM_q$ is an invariant of $Y$ for $q$ in the allowable range has been proved by Fr\o yshov in \cite{MR2738582}. In fact, Fr\o yshov proved something even stronger, as we let the chamber $m_\mathfrak{s} \to \infty$, the resulting homology group is an invariant of $Y$ regardless of the homological degree. Conjecturally, this group is expected to be isomorphic to one of the flavors of monopole Floer homology defined by Kronheimer-Mrowka in \cite{MR2388043}. 
    \end{Rem}

    \begin{Rem}
        The terminology "$\mathcal{S}$-complex" is not original in this paper and has appeared elsewhere. Roughly, the category of $\mathcal{S}$-complexes is a formalism that describes many $S^1$-equivariant (co)homology theories that arise naturally in differential topology. A full definition can be given as we replace $S^1$ with some other group $G$. For example, when $G = SO(3)$, we have the formalism of $\mathcal{SO}$-complexes that is investigated in \cite{MR1910040, MR1883043, daemi2022instantonsrationalhomologyspheres}, when $G = \mathbf{Z}/2$, an analogous formalism appears in  \cite{MR2739000}. For $G = S^1$, the particular definition of $\mathcal{S}$-complexes we use here appeared in, e.g., \cite{MR4742808} for singular instanton homology.
    \end{Rem}

    It turns out that the perturbed Chern-Simons-Dirac functional $\mathcal{L}_{\eta}$ induces an $\mathbf{R}$-filtration on $CM, \widetilde{CM}$ (cf. Proposition \ref{Prop5.1}). Thus, $CM, \widetilde{CM}$ are persistent modules. As a result, we can define the following numerical quantities.

    \begin{Def}[cf. Definition \ref{Def5.2}]\label{mainDef1}
        Let $Y$ be a rational homology three-sphere that is smooth, closed, connected, and oriented. Suppose $g$ is a fixed Riemannian metric of $Y$, and $\mathfrak{s}$ is $spin^c$ structure on $Y$. Let $m_\mathfrak{s}$ be a representative of a chamber and $\eta \in i\Omega^2(Y)$ that is closed and generic. For each $q$, we define
        $$\rho_q(Y, m_{\mathfrak{s}}, g) = \lim_{\norm{\eta} \to 0} \inf\{t : \iota^t_\ast: \widetilde{HM}^t_q(Y, m_\mathfrak{s}, g, \eta; \mathbf{F}) \to \widetilde{HM}_q(Y, m_\mathfrak{s}, g, \eta; \mathbf{F}) \text{ non-trivial}\}.$$
        Similarly, we define $\lambda_q(Y, m_\mathfrak{s}, g)$ associated with $HM_q$.
    \end{Def}
    
    As mentioned previously, $\rho_q, \lambda_q$ cannot be an invariant of $Y$ because of the dependence on $g$ in the definition of the Chern-Simons-Dirac functional. However, what remains true is that

    \begin{Th}[cf. Theorem \ref{Th5.4}]\label{mainTh2}
        With the same set up as Theorem \ref{mainTh1}, for allowable range of $q$, and $\mathfrak{s}$ is a torsion $spin^c$ structure, then $\rho_q(Y, m_\mathfrak{s}, g)$ and $\lambda_q(Y, m_\mathfrak{s}, g)$ are invariants of $(Y,g)$.
    \end{Th}

    Note that $\rho_q, \lambda_q$ take values in $\mathbf{R} \cup \{\infty\}$. One of the immediate observations we see from Definition \ref{mainDef1} is that if $\lambda_q$ or $\rho_q$ is finite, then $HM$ or $\widetilde{HM}$ is non-trivial, respectively. Thus, it is natural to investigate the relationship between $\lambda_q$ and $\rho_q$. Ideally, from a relation between the two spectral invariants, one gets some information about $\widetilde{HM}$ from $HM$. The differential $\widetilde{d}$ of $\widetilde{CM}$ decomposes as follows
    $$\widetilde{d}= \begin{bmatrix}
        d & 0 & 0 \\ u & \delta_2 & d \\ \delta_1 & 0 & 0
    \end{bmatrix},$$
    where $u: CM \to CM[2]$, $\delta_1 : CM \to \mathbf{F}$, and $\delta_2 : \mathbf{F} \to CM$. The fact that $\widetilde{d}^2 = 0$ yields for us a certain algebraic relations between $d, \delta_1, \delta_2$, and $u$ (cf. Proposition \ref{Prop4.8}). This leads to a finite filtration of $\widetilde{CM}$ that is given as follows
    $$\mathscr{F}^0 \widetilde{CM}_\ast = 0 \subset \mathscr{F}^1 \widetilde{CM}_\ast = CM_{\ast-1} \subset \mathscr{F}^2\widetilde{CM}_\ast= \mathbf{F} \oplus CM_{\ast -1} \subset \mathscr{F}^3\widetilde{C}_\ast = \widetilde{C}_\ast.$$
    The differential for each sub-complex is induced by $\widetilde{d}$, which respectively is
    $$\mathscr{F}^0 \widetilde{d} = 0, \quad  \mathscr{F}^1 \widetilde{d} = d, \quad \mathscr{F}^2 \widetilde{d} = \begin{bmatrix} 0 & 0 \\ \delta_2 & d \end{bmatrix}, \quad \mathscr{F}^3 \widetilde{d} = \widetilde{d}.$$
    By exploiting the above filtration, we show that

    \begin{Th}[cf. Theorem \ref{Th5.19}]\label{mainTh3}
        Let $Y$ be a rational homology three-sphere equipped with a $spin^c$ structure $\mathfrak{s}$. Suppose $m_{\mathfrak{s}} \in \mathfrak{m}(Y)$ is a chamber of $Y$. Fix a metric $g$ on $Y$ associated with $m_{\mathfrak{s}}$. For each $q$ in an allowable range, 
    \begin{enumerate}
        \item If $\delta_{2\ast} = 0$ and $u_\ast : ker(\delta_{1\ast} : HM_{q+1}(Y, m_\mathfrak{s}; \mathbf{F}) \to \mathbf{F}) \to HM_{q-1}(Y, m_\mathfrak{s}; \mathbf{F})$ is trivial, then $\lambda_{q-1}(Y, m_{\mathfrak{s}}, g) \geq \rho_q (Y, m_{\mathfrak{s}}, g)$.
        \item If $\delta_{1\ast} = 0$ and $u_{\ast} : HM_{q-3}(Y, m_{\mathfrak{s}}; \mathbf{F}) \to HM_{q-5}(Y, m_{\mathfrak{s}}; \mathbf{F})/im(\delta_{2\ast})$ is trivial, then $\lambda_{q-3}(Y, m_\mathfrak{s}, g) \geq \rho_{q}(Y, m_\mathfrak{s},g)$.
    \end{enumerate}
    \end{Th}

    \subsection{An application}

    Given a smooth manifold equipped with a Riemannian metric, there are different associated notions of curvature, e.g., sectional curvature, Ricci curvature, and scalar curvature. The topology of manifolds with controlled sectional and Ricci curvature tends to be rigid. However, manifolds with bounded from below scalar curvature tend to be more flexible topologically \cite{gromov2022perspectives} (see also, e.g., \cite{MR3911569, MR4637974, chodosh2024completeriemannian4manifoldsuniformly, DemetreKazaras24}). In particular, understanding (topological) obstruction to the existence of positive scalar curvature (psc) on a manifold is an active area of research in geometric analysis. Traditionally, there are two approaches to studying manifolds with psc: theories involved with the Dirac operators and minimal surfaces theory. An application of the main results of the current paper falls within the former framework. 

    Let $W$ be any smooth compact $spin^c$ $4$-manifold with boundary $\partial W = -Y_1 \cup Y_2$, where $Y_i$ is a rational homology sphere. Then we can view $W: Y_1 \to Y_2$ as a cobordism from $Y_1$ to $Y_2$. Let $\mathfrak{s}_W$ be a $spin^c$ structure on $W$ and its restriction to $Y_i$ is denoted by $\mathfrak{s}_i$. For each $\mathfrak{s}_i$, we choose a representative $m_{\mathfrak{s}_i}$ of a chamber in $\mathfrak{m}(Y_i)$. Let $d(W) = (c_1(\det\, \mathfrak{s}_W)^2 - \sigma(W))/4 + b_1(W) - b^+(W)$ and $k = m_{\mathfrak{s}_1} - m_{\mathfrak{s}_2} - d(W)/2$. We define

    \begin{Def}[cf. Definition \ref{Def5.8}]\label{mainDef2}
        Let $W: Y_1 \to Y_2$ be a cobordism between smooth, closed, connected oriented rational homology spheres. For $b^+(W) \leq 1$ and $k \leq -1$, we say that $W$ is an injective cobordism if the induced map $CM(W)_\ast : HM_q(Y_1, m_{\mathfrak{s}_1}) \to HM_{q-d(W)}(Y_2, m_{\mathfrak{s}_2})$ is injective for allowable range of the homological degree. 
    \end{Def}

    An example of an injective cobordism is a ribbon rational homology cobordism (cf. Section \ref{Sec6}). One of the nice features of ribbon cobordism $W$ is that its "double" $D(W) = -W \cup_{Y_2} W$ can be described by a surgery on $(Y_1 \times I) \# m (S^1 \times S^3)$, where $m$ is the number of $1$-handles in $W$ along certain closed curves (cf. Proposition 5.1 in \cite{MR4467148}). This is a crucial observation in our proof of the following theorem.

    \begin{Th}[cf. Theorem \ref{Th6.4}]\label{mainTh4}
    Suppose $W: Y_1 \to Y_2$ is a ribbon rational homology cobordism, where $Y_i$ is a smooth, closed, oriented rational homology three-sphere. Let $\mathfrak{s}_W$ be a $spin^c$ structure on $W$ and its restriction to $Y_i$ is denoted by $\mathfrak{s}_i$. Let $m_{\mathfrak{s}_i} \in \mathfrak{m}(Y_i)$ be a chamber. Then, the induced cobordism map $CM(D(W))_{\ast} : HM(Y_1, m_{\mathfrak{s}_1}; \mathbf{F}) \to HM(Y_1, m_{\mathfrak{s}_2}; \mathbf{F})$ satisfies
    $$CM(D(W))_{\ast} = |H_1(W, Y_1)| CM(Y_1 \times I)_{\ast}.$$
    The map on the right-hand side is some non-zero scalar multiple of the identity map on $HM(Y_1, m_{\mathfrak{s}_1}; \mathbf{F})$. Consequently, $CM(W)_\ast$ is injective.
\end{Th}

\begin{Rem}
    Analogous statements to Theorem \ref{mainTh4} have been proved for Heegaard Floer homology, instanton Floer homology, and sutured instanton Floer homology in \cite{MR4467148}. Furthermore, once Theorem \ref{mainTh4} is proved, it is just a matter of algebra to extend the result to the $\mathcal{S}$-complex of monopoles (cf. Theorem \ref{Th6.5}).
\end{Rem}

We introduce a notation: let $\rho^{\circ}_q(Y,\mathfrak{s}, g)$ refer to $\rho_q(Y,\mathfrak{s}, g)$ or $\lambda_q(Y, \mathfrak{s}, g)$, depending on whether $\circ = \emptyset$ or $\sim$. Taking advantage of the fact that a ribbon homology cobordism induces an injective homomorphism at the level of homology and that the induced map changes the $\mathbf{R}$-filtration by a quantity involved the scalar curvature of a metric on a cobordism, we will prove that

\begin{Th}[cf. Theorem \ref{Th6.6}]\label{mainTh5}
    Let $W : Y_1 \to Y_2$ be a ribbon rational cobordism between smooth, closed, connected, oriented rational homology three-sphere. Suppose $W$ is spin, and let $\mathfrak{s}_W$ be a spin structure on $W$ that, when restricted to $Y_i$, gives a spin structure $\mathfrak{s}_i$ on $Y_i$. Let $g$ be a metric on $W$ and denote $s$ by its scalar curvature. Then for allowable range of degree $q$ (cf. Corollary \ref{Cor4.11}) we have 
    $$\dfrac{1}{32}\int_{W} s^2 \geq \rho^{\circ}_{q - d(W)}(Y_2; \mathfrak{s}_2, g_2) - \rho^{\circ}_{q}(Y_1, \mathfrak{s}_1, g_1) - C.$$
    Here $g_i$ is the induced metric on the boundary components of $Y$, $d(W)$ is a topological constant of $W$, $C$ is a constant that only depends on the canonical (trivial) $spin^c$ connection $B_0$ on $W$ and the spin structure $\mathfrak{s}_W$.

    Furthermore, if we have $\rho^{\circ}_{q-d(W)}(Y_2,\mathfrak{s}_2, g_2) > \rho^{\circ}_{q}(Y_1, \mathfrak{s}_1, g_1) + C$, then $W$ can never have a metric with positive scalar curvature.
\end{Th}

If $Y_1, Y_2$ are hyperbolic rational homology three-spheres equipped with their respective canonical hyperbolic metric $g_1, g_2$, then by Mostow's rigidity theorem, $\rho_\ast(Y_i, \mathfrak{s}_i, g_i)$ is a topological invariant of $Y_i$. The upshot of Theorem \ref{mainTh5} is that $\rho^\circ$ provides a topological lower bound for the scalar curvature functional of a $4$-dimensional ribbon rational homology cobordism associated with any metric $g$ whose restriction to each of the boundary components gives us $g_i$. Furthermore, a relationship between spectral invariants of the two hyperbolic ends provides an obstruction to the existence of psc on the whole cobordism. Naturally, as an independent interest, it is worth to investigate a relationship between $\rho^\circ$ with other invariants defined from the scalar curvature, e.g., the Yamabe invariant. We reserve this question for future work.

\begin{Rem}
    Let us mention an elementary observation of $\rho^\circ$ and the spectral geometric data of $Y$. Since we are working with rational homology three-sphere $Y$, by the Hodge decomposition theorem, $\Omega^1(Y) = d\Omega^0(Y) \oplus d^* \Omega^2(Y)$. Furthermore, the Laplacian $\Delta = (d+ d^*)^2$ respects the decomposition. The spectrum of $\Delta$ on coexact $1$-forms is discrete and always positive. Denote by $\lambda^*_1$ the first eigenvalue of $\Delta$. A classical problem in spectral geometry is to obtain a bound for $\lambda^*_1$. Using a refinement of various well-known estimates of the solutions of the Seiberg-Witten equations, in \cite{MR4165316, MR4322393}, it was shown that there is an upper bound for $\lambda^*_1$ in terms of a lower bound of the Ricci curvature as long as $Y$ is not an $L$-space. Recall that $Y$ is an $L$-space if and only if $HM(Y) = 0$; or equivalently, there are no irreducible solutions of the Seiberg-Witten equations. The condition that $Y$ is not a $L$-space is purely topological and there are many examples (see \cite{MR2450204, MR2168576, MR2299739}). Additionally, if $Y$ is hyperbolic, then the upper-bound of $\lambda^*_1$ by $2$ is obtained when there is an irreducible solution of the Seiberg-Witten equations (cf. Theorem 3 of \cite{MR4322393}). Thus, if $\lambda^*_1 >2$, then the topological invariant $\rho^\circ_\ast(Y, \mathfrak{s},g)$ associated with the canonical hyperbolic metric $g$ has to be $+\infty$.  This begs further investigation of the role of $\rho^\circ$ in the context of hyperbolic geometry.
\end{Rem}



    \subsection{Organization}

    In Section \ref{Sec2}, we discuss the persistent module from the Morse theory perspective, define a spectral invariant in this setting, and derive some basic relation between the spectral invariants in the pull-up-push-down construction. In Section \ref{Sec3}, we consider the finite-dimensional prototype model of an $S^1$-equivariant complex defined for an $S^1$-manifold satisfies certain conditions that can be generalized to an infinite-dimensional setting. The precise formalism of the category of $\mathcal{S}$-complexes is given in Subsection \ref{Sub3.4}. Section \ref{Sec4} is a review of monopole Floer homology from Fr\o yshov's perspective, and a proof of Theorem \ref{mainTh1} is given in this section. We define the monopole spectral invariants in Section \ref{Sec5} and prove Theorem \ref{mainTh2} and Theorem \ref{mainTh3}. The majority of Section \ref{Sec6} is devoted to the review of ribbon homology cobordism and the proof of Theorem \ref{mainTh4}. 

    \begin{proof}[Acknowledgement]\renewcommand{\qedsymbol}{}
    Parts of this work were inspired by the discussion of the paper \cite{MR3590354} in the "Journal Club" virtual reading seminar that was organized by Daemi, Iida, and Scaduto in the previous academic year of 2023-2024. We would like to thank Ali Daemi for suggesting this project to us and taking the time to explain as well as provide many valuable insights about certain aspects of the category of $\mathcal{S}$-complexes. We are also grateful for some conversations with Mike Miller Eismeier, Tye Lidman about things related to this project.
\end{proof}

	\section{Finite dimensional prototype}\label{Sec2}

    \subsection{Morse homology}
    Let $M$ be a closed smooth finite-dimensional manifold. Suppose $f: M \to \mathbf{R}$ is a Morse-Smale function. As usual, over $\mathbf{F}=\mathbf{F}_2$, we have a Morse complex of $M$ associated with $f$
    $$C(M,f,d):= \bigoplus C_\ast (M, f, d),$$
    where $C_\ast (M,f,d)$ is a freely finitely generated abelian group over $\mathbf{F}$ with generators are the critical points of $f$ of index $\ast$, and the differential $d$ counts the numbers of unparametrized (broken) flow lines connecting critical points.

    The Morse complex $C(M,f,d)$ is also a \textit{non-Archemedian} vector space. Specifically, there is a semi-norm  $\ell: C(M,f,d) \to \R \cup \{-\infty\}$ defined as following: For each $x \in C(M,f)$ with $\displaystyle x = \sum a_i p_i$ where $p_i \in Crit\;f$ and $a_i \in \mathbf{F}$,
    $$\ell(x) := \max \{f(p_i): a_i \neq 0\}.$$
    If $x =0$, we set $\ell(x) = -\infty$ by convention. By semi-norm, we mean that $\ell(x+y) \leq \max \{\ell(x), \ell(y)\}$. 

    Using $\ell : C(M,f,d) \to \R \cup \{-\infty\}$, we define an $\R$-filtration on the Morse complex defined above. In particular, for each $t \in \R$, we define
    $$C^t(M,f,d):=\{ c \in C(M,f,d): \ell(c) \leq t\}.$$
    Naturally, $d$ respects this $\R$-filtration and thus, for each $t\in \R$, $C^t(M,f,d)$ is a chain complex in its own right. Denote the homology of $C^t(M,f,d)$ by $H^t$. Furthermore, naturally, if $t\leq s$, then there is a chain map $i_{ts}: C^t(M,f,d) \to C^s(M,f,d)$. We also define
    $$i_t : C^t(M,f,d) \to C(M,f,d).$$

    Having set up the above terminologies, we are now ready to define the spectral invariants of $M$. There are two ways to go about this, depending on whether one has a certain "distinguished" element of $H(C(M,f,d))$.

    \begin{Def}\label{Def1.1}
        For each $a \in H(C(M,f,d))$. We define $\rho_M(a)$ by
        \begin{align*}
            \rho_M(a) &:= \inf \{ t \in \R : a \in \text{Im}\,i^*_t\}\\
            & = \inf \{ \ell(\alpha) : \alpha \in C(M,f,d),\; [\alpha]=a\}.
        \end{align*}
    \end{Def}

    \begin{Def}\label{Def1.2}
        In the case when there is no "distinguished" element in the homology of the Morse complex, we can still define
        \begin{align*}
            \rho_M &:= \inf \{ t \in \R: \text{Im}\; i^*_t \neq \{0\}\}\\
            &= \inf \{\ell(\alpha) : \alpha \in C(M,f,d), [\alpha] \neq 0\}.
        \end{align*}
    \end{Def}
	
	\begin{Rem}
	    Both $\rho_M(a)$ and $\rho_M$ defined in Definition \ref{Def1.1} and Definition \ref{Def1.2} are \textit{not} invariants of $M$, i.e., they do depend on the choices of Morse-Smale function (or Riemannian metrics) on the underlying manifold. This is because as one continuously changes the auxiliary data from $f_1$ to $f_2$, the continuation map $F_{12}: C(M,f_1,d_1) \to C(M, f_2, d_2)$ (which is a quasi-isomorphism) might not preserve the $\mathbf{R}$-filtration without some special assumption on $f_1$ and $f_2$.
	\end{Rem}

    Next, we investigate how these spectral invariants are affected via functoriality. Functoriality in a finite dimension situation can be phrased as follows: Let $(W, M_1, M_2)$ be a triple of manifolds such that there are maps $r_j: W \to M_j$, $j=1,2$. We choose a Morse-Smale function $f_j$ on each $M_j$. We assume the following transversality condition
    
    \begin{quote}
       \textbf{Assumption A:} $r= (r_1, r_2): W \to M_1 \times M_2$ is always transverse to $U_a \times S_b$ for all critical points $a,b$ of $f_1, f_2$, respectively.
    \end{quote}

    With this assumption, using the \textit{pull-up-push-down construction}, we construct a chain map $MC(W): C(M_1,f_1,d_1) \to C(M_2,f_2,d_2)$ between Morse complexes. By the transversality assumption,
        $$\mathcal{M}(a,b) := \{z \in W : r(z) \in U_a \times S_b \}$$
    is a finite-dimensional manifold of dimension
        $$dim\;\mathcal{M}(a,b) = dim\;W-dim\;Y_1 + \ind\;a - \ind\;b.$$
    Thus, when $\ind\,a - \ind\, b = -dim\, W + dim\, Y_1$, $\mathcal{M}(a,b)$ is zero dimensional. We then can do a $\mathbf{F}$-count of points in $\mathcal{M}(a,b)$, denoted by $\#_{\mathbf{F}}\mathcal{M}(a,b)$. As a result, $MC(W)$ is defined as follows
        \[MC(W)(a)  = \sum_{b \in Crit(f_2)} \la MC(W)(a), b \ra\; b,\]
    where $\la MC(W)a, b\ra = \#_{\mathbf{F}}\mathcal{M}(a,b)$ and the summation is taken over critical points $b$ of $f_2$ that satisfies $\ind\; b = dim\; W- dim\; Y_1 + \ind\; a$. Note that as a chain map, $deg\;MC(W) = dim\;W - dim\;Y_1$.

    As we have seen above, each Morse complex of $M_j$ comes with an $\R$-filtration induced by the Morse-Smale function $f_j$. How the chain map $MC(W)$ interacts with these filtration depends on some assumption on $f_j$. To this end, for the sake of simplicity, it is most convenient that we assume

    \begin{quote}
        \textbf{Assumption B:} $f_j: M_j \to \R$ are Morse-Smale functions such that $\sup\;f_2 \leq \inf\;f_1$.
    \end{quote}

    \begin{Lemma}\label{Lem1.4}
        With the Assumptions A-B as above, the chain map $MC(W)$ respects the $\R$-filtration on the Morse complexes $C(M_j,f_j,d_j)$. That is,
        $$MC(W): C^t(M_1,f_1,d_1) \to C^t(M_2,f_2,d_2).$$
    \end{Lemma}

    \begin{proof}
        Let $x = \displaystyle \sum a_k p_k \in C(M_1,f_1,d_1)$. Then
            $$MC(W)x = \sum_{b\in Crit(f_2)} \left(\sum a_k \la MC(W)(a_k),b\ra\right)b.$$
        Then by the orthogonality property of the semi-norm,
            $$\ell_2\left(MC(W)x\right) = \max \{f_2(b) : \sum a_k \la MC(W)(a_k),b\ra\neq 0\}.$$
        Since $\sup\;f_2 \leq \inf\;f_1$, $f_2(b) \leq f_1(a_k) \leq \ell_1(x)$. As a result, $\ell_2(MC(W)x) \leq \ell_1(x)$. Therefore, $MC(W)$ preserves the $\R$-filtration as claimed.
    \end{proof}

    For each $t\in \R$, Lemma \ref{Lem1.4} allows us to consider the following commutative diagram
    \begin{equation}\label{eq:1.1}
        \begin{tikzcd}
            H^t(C(M_1,f_1,d_1)) \arrow{r}{MC(W)^*} \arrow[swap]{d}{i^{*}_{1t}} & H^t(C(M_2, f_2,d_2)) \arrow{d}{i^*_{2t}} \\%
            H(C(M_1,f_1,d_1)) \arrow{r}{MC(W)^*}& H(C(M,f_2,d_2))
        \end{tikzcd}
    \end{equation}
    Let $t\in \R$ such that $\text{Im}\; i^*_{1t}$ is non-trivial. By the commutativity of \eqref{eq:1.1}, we have 
    $$MC(W)^* i^*_{1t}(a) = i^*_{2t}MC(W)^*(a).$$ 
    As a result, $i^*_{2t}MC(W)^*(a)$ is non-trivial if and only if $i^*_{1t}(a) \notin \text{Ker}\;MC(W)^*$. If that is the case, then $i^*_{2t}$ has a non-trivial image. Therefore, $\rho_{M_2} \leq t$, which implies that $\rho_{M_2} \leq \rho_{M_1}$. 

    For the above argument to work, recall that we need $i^*_{1t}(a)$ to be a non-trivial element of $H(C(M_1,f_1,d_1))$ that does not belong to $\text{Ker}\; MC(W)^*$. This condition is easily achievable when, say, $MC(W)^*$ has a trivial kernel. Note that the mapping cone $\text{Cone}(MC(W))=C(M_1,f_1,d_1)\oplus C(M_2,f_2,d_2)$ is another complex where the differential is given by
    $$d_{\text{cone}} = \begin{bmatrix} d_1 & 0 \\ MC(W) & d_2 \end{bmatrix}.$$
    Then, $(C(M_1,f_1,d_1), C(M_2, f_2, d_2), \text{Cone}(MC(W)))$ forms an exact triangle. This means that we have the following long exact sequence
    \begin{align*}
        \cdots &\to H_{j-1}(\text{Cone}(MC(W))) \to H_j(C(M_1,f_1,d_1)) \to H_j(C(M_2,f_2,d_2)) \to\\ 
        &\to H_j(\text{Cone}(MC(W))) \to \cdots
    \end{align*}
    If the mapping cone complex $\text{Cone}(MC(W))$ is acyclic, then the above long exact sequence tells us that $MC(W)^*$ is an isomorphism. In particular, $\text{Ker}\; MC(W)^* = \{0\}$. To this end, we now state another assumption.

    \begin{quote}
        \textbf{Assumption C:} The mapping cone complex $\text{Cone}(MC(W))$ is acyclic.
    \end{quote}

    The above discussion can be summarized in the form of the following theorem.

    \begin{Th}\label{Th1.5}
        Suppose $(W,M_1,M_2)$ is a triple of manifolds that satisfies Assumptions A-B-C as above. Then $\rho_{M_2} \leq \rho_{M_1}$. 
    \end{Th}

	In the case that the homology of $M_1$ has a "preferred" element $a$ such that its image in the homology of $M_2$ is also "preferred" in a certain sense. A similar argument outlined in the discussion above gives the following theorem.

    \begin{Th}\label{Th1.6}
        Let $a\in H(M_1;\mathbf{F})$ be a "distinguished" element such that $MC(W)^*(a)$ is also "distinguished" in $H(M_2;\mathbf{F})$. With the Assumptions A-B-C as above, we have $\rho_{M_2}(MC(W)^*(a))\leq \rho_{M_1}(a)$. 
    \end{Th}
	
	\begin{Rem}\label{Rem1.7}
	    The meaning of "distinguished" or "preferred" depends on different contexts and certain topological information that the spectral invariants wish to identify. 
	\end{Rem}

    \begin{Rem}
        In general, Assumption A seems reasonable to achieve (after some perturbation). The assumption B-C is quite rigid and possibly can be relaxed to arrive at analogous quantitative statements in Theorem \ref{Th1.5} and Theorem \ref{Th1.6}, depending on the situation.
    \end{Rem}

    \subsection{Morse-Novikov homology}
    Even though we will not encounter a Morse-Novikov situation for monopole Floer homology of rational homology sphere, we include a finite-dimensional prototype here that may be of independent interest. In many situations, especially in infinite-dimensional settings, $M$ comes equipped readily with a Morse-Smale function $f$ that is not $\R$-valued. In particular, a setting that is relevant for us to generalize to infinite dimension is as follows: we have a Morse-Smale function $f: M \to S^1$. Associated with such auxiliary data, we build a Morse-Novikov complex.

    The idea is simple, we do Morse homology on a cover of $M$ instead. Now, there are many choices of cover. However, for the sake of simplicity of the exposition, we consider the \textit{infinite cyclic cover} $\widetilde{M}$ of $M$, $\widetilde{M}:=f^*(\R)$, where $\R \to S^1= \R/\mathbf{Z}$ is the universal cover of the unit circle. Recall that this means the group of deck transformations of $\widetilde{M}$ is $\Z$, and the action $\Z \curvearrowright \widetilde{M}$ is given by $1 \mapsto T \in Aut(\widetilde{M})$, where $T(x,\lambda) = (x, \lambda - 1)$ with $(x,\lambda) \in \widetilde{M} \subset M \times \R$.

    We define $\widetilde{f} : \widetilde{M} \to \R$ by $\widetilde{f}(x,\lambda) = \lambda$. Note that $\widetilde{f}$ is an $\Z$-equivariant function and the following diagram is commutative.
        \[
            \begin{tikzcd}
                \widetilde{M} \arrow{r}{\widetilde{f}} \arrow[swap]{d}{\widetilde{\pi}} & \R \arrow{d}{\pi} \\%
                 M \arrow{r}{f}& S^1
            \end{tikzcd}
        \]
    Let $d\theta \in H^1(S^1;\R)$ be the generator. Denote $grad\;f\in\Gamma(M,TM)$ by vector field on $M$ such that $\la grad\;f, v\ra = f^*d\theta(v)$ for all $v \in \Gamma(M,TM)$. Typically, $grad\;f$ is an example of \textit{gradient-like vector field} of $f$.

    \begin{Lemma}\label{Lem1.9}
        The usual gradient of $\widetilde{f}$ is given by $\widetilde{\pi}^*(f^*d\theta)$. 
    \end{Lemma}

    \begin{proof}
        Firstly note that we will work with the pulled-back metric on $\widetilde{M}$. It suffices for us to show that $\widetilde{\pi}^*(f^*d\theta) = d\widetilde{f}$. Suppose $\gamma : (-\epsilon, \epsilon) \to \widetilde{M}$ such that $\gamma(s) = (x(s), \lambda(s))$ and $\dot{\gamma}(0)$ is a vector in $T_{\gamma(0)}\widetilde{M}$. It is easy to see that
        $$d\widetilde{f}(\dot{\gamma}(0)) = \dfrac{d}{ds} \widetilde{f}(\gamma(s))\bigg|_{s=0} = \dfrac{d}{ds} \lambda(s)\bigg|_{s=0} = \dot{\lambda}(0).$$
        On the other hand, we also have
        \begin{align*}
            \widetilde{\pi}^*(f^*d\theta)(\dot{\gamma}(0)) & = d\theta ((f\circ \widetilde{\pi})_* (\dot{\gamma}(0))\\
            &= d\theta \left(\dfrac{d}{ds} (f\circ \widetilde{\pi})(\gamma(s))\bigg|_{s=0}\right) = d\theta \left( \dfrac{d}{ds}f(x(s))\bigg|_{s=0}\right).
        \end{align*}
        Since $f(x(s)) = e^{2\pi i \lambda(s)}$, differentiate at $s=0$, we obtain $\dfrac{d}{ds}f(x(s))\bigg|_{s=0} = e^{2\pi i \lambda(0)}\cdot \dot{\lambda}(0)$. Thus, we can re-write the right-hand side of the above expression as
        $$\widetilde{\pi}^*(f^*d\theta)(\dot{\gamma}(0)) = d\theta (e^{2\pi i \lambda(0)}\cdot \dot{\lambda}(0)) = \dfrac{d}{ds}\theta(e^{2\pi i \lambda(s)})\bigg|_{s=0} = \dot{\lambda}(0).$$
        As a result, we see that $\widetilde{\pi}^*(f^*d\theta) = d\widetilde{f}$ as claimed.
    \end{proof}
	
	The above Lemma \ref{Lem1.9} implies that $\widetilde{a}$ is a critical point of $\widetilde{f}$ if and only if $\widetilde{\pi}(\widetilde{p})$ is a critical point of $f$. However, note that there could be multiple critical points of $\widetilde{f}$ that descend to the same critical point of $f$. Thus, even though there is only a finite number of critical points of $f$, there might be an infinite number of critical points of $\widetilde{f}$. Because of this, naively constructing an $\mathbf{F}$-module chain complex using $Crit(\widetilde{f})$ would not work. To compensate for this infinite-ness, we introduce an algebraic book-keeping device via the means of a \textit{Novikov field}. The Novikov field $\Lambda_{\mathbf{F}}$ for our setup is defined as follows: Treating $T \in Aut(\widetilde{M})$ defined above as an indeterminate,
        \begin{equation*}
            \Lambda_{\mathbf{F}}:=\left\{\displaystyle \sum_{j=-\infty}^\infty a_j T^j : a_j \in \mathbf{F},\; \#\{j \leq 0: a_j \neq 0\} < \infty\right\}.
        \end{equation*}
    Note that the action of $\Z$ on $\widetilde{M}$ preserves the set of critical points of $\widetilde{f}$. As a result, there is a natural action of $\Lambda_{\mathbf{F}}$ on $Crit(\widetilde{f})$. We denote the \textit{Morse-Novikov complex} by
        \begin{equation*}
            CN_\ast (M,f,d) := \Lambda_{\mathbf{F}}\cdot Crit_{\ast}(\widetilde{f}), \quad CN(M,f,d) = \bigoplus CN_{\ast}(M,f,d),
        \end{equation*}
    where $Crit_\ast(\widetilde{f})$ is the set of all critial points of $\widetilde{f}$ of index $\ast$, and the differential $d$ is defined by
        $$d\widetilde{p}=\displaystyle \sum_{j=-\infty}^\infty \sum_{\widetilde{q}\in Crit(\widetilde{f})} \#_{\mathbf{F}} \mathcal{M}(\widetilde{p},T^j\widetilde{q})T^j\widetilde{q}.$$
    The coefficients of $d$ is the $\mathbf{F}$-count of the number of unparametrized (broken) negative gradient flow lines of $\widetilde{f}$ connecting between critical points just as in the usual Morse theory. $\Lambda_{\mathbf{F}}$ is an $\mathbf{F}[T,T^{-1}]$-module, and $CN(M,f,d)$ is both an $\mathbf{F}[T,T^{-1}]$-module and $\Lambda_\mathbf{F}$-module. As a module over $\Lambda_\mathbf{F}$, $CN(M,f,d)$ is \textit{freely finitely generated}, the finite bases are given by a choice of lifts of the critical points of $f$. Hence, the differential $d$ of the Morse-Novikov complex defined above is a $\Lambda_\mathbf{F}$-linear map. Note the resulting homology of $CN(M,f,d)$, \textit{Morse-Novikov homology}, depends only on the homotopy classes of $f$ in $[M, S^1]$. Since there is a bijection between $[M, S^1]$ and $H^1(M;\Z)$, $H(CN(M,f,d))$ depends only on $[f]\in H^1(M;\Z)$.

    Just as in the case of $\R$-valued Morse theory, we can define a notion of spectral invariants for the Morse-Novikov complex just as in Definitions \ref{Def1.1}  and \ref{Def1.2}. The $\Lambda_\mathbf{F}$-module chain complex $CN(M,f,d)$ has a semi-norm induced from $\widetilde{f}$. $CN(M,f,d)$ has an equivalent description 
        $$CN(M,f,d)=\left\{\sum_{j=-\infty}^\infty a_j \widetilde{p_j}: a_j \in \mathbf{F},\; \#\{j\leq 0: a_j \neq 0\}<\infty\right\}.$$
    Hence, the semi-norm $\ell: CN(M,f,d) \to \R\cup\{-\infty\}$ is given by
        $$\ell\left(\sum_{j=-\infty}^\infty a_j\widetilde{p_j}\right):=\max\{\widetilde{f}(\widetilde{p_j}): a_j \neq 0\}.$$
    For each $t\in \R$, an $\R$-filtration on the Morse-Novikov complex is defined by
        $$CN^t(M,f,d):=\{c \in CN(M,f,d): \ell(c) \leq t\}.$$
	As a result, $d$ respects this filtration. For each $t\leq s$, there is a chain map $i_{ts}: CN^t \to CN^s$. There is also a chain map $i_t : CN^t \to CN$. Thus, following Definitions \ref{Def1.1} and \ref{Def1.2}, we define

    \begin{Def}\label{Def1.10}
        Let $f:M\to S^1$ be a Morse-Smale function and $\widetilde{M}:=f^*(\R)$ be the infinite cyclic cover of $M$ as above.
        \begin{enumerate}
            \item For each $a \in H(CN(M,f,d))$ viewed as an element of $H(C^{cell}(\widetilde{M})\otimes_{\mathbf{F}[\Z]} \Lambda_{\mathbf{F}})$ via the canonical isomorphism, we define
                $$\rho^{[f]}_M(a) = \inf\{t\in\R: a \in \text{Im}\;i^*_t\} = \inf\{\ell(\alpha): \alpha \in CN(M,f,d),\; [\alpha]= a\}.$$
            \item If there is no "distinguished" $a$ in the Morse-Novikov homology, we define
                $$\rho^{[f]}_M:= \inf\{t\in\R: \text{Im}\;i^*_t \neq \{0\}\} = \inf\{\ell(\alpha): \alpha \in CN(M,f,d), \, [\alpha] \neq 0\}.$$
        \end{enumerate}
    \end{Def}

    \begin{Rem}\label{Rem1.11}
        The Morse-Novikov complex $CN(M,f,d)$ with the semi-norm $\ell$ defined above forms a non-Archemidean vector space over $\Lambda_\mathbf{F}$.
    \end{Rem}

    \begin{Rem}
        Note that, like the Morse case, $\rho^[f]_M(a)$ and $\rho^[f]_M$ are not expected to be an invariant of $M$.
    \end{Rem}

    For $j=1,2$, let $M_j$ be a closed smooth manifold equipped with a Morse-Smale function $f_j : M_j \to S^1$. Consider the infinite cyclic cover $\widetilde{M_j}:=f_j^*(\R)$ with an $\R$-valued Morse-Smale function $\widetilde{f_j}$ as before. Suppose that there is another manifold $\widetilde{W}$ equipped with a map $\widetilde{r} = (\widetilde{r_1}, \widetilde{r_2}): \widetilde{W} \to \widetilde{M_1}\times \widetilde{M_2}$. We assume the following assumptions on $\widetilde{f_j}, \widetilde{r},$ and $\widetilde{W}$ analogous to Assumptions A-B-C in the previous subsection:
    
        \begin{quote}
            \textbf{Assumption D:} $\widetilde{r}$ is always transverse to $U_{\widetilde{p}} \times S_{\widetilde{q}}$.
        \end{quote}

        \begin{quote}
            \textbf{Assumption E:} $\sup \widetilde{f_2} \leq \inf \widetilde{f_1}$.
        \end{quote}

        \begin{quote}
            \textbf{Assumption F:} From Assumption D, using the pull-up-push-down construction, we have an $\Lambda_\mathbf{F}$-chain map $NC(\widetilde{W}): CN(M_1,f_1,d_1) \to CN(M_2,f_2,d_2)$. We assume that the resulting mapping cone complex $\text{Cone}(NC(\widetilde{W}))$ is acyclic.
        \end{quote}

    With the same argument as in the proof of Theorem \ref{Th1.5}, we arrive at the following relation between the spectral invariants of the Morse-Novikov complex of $M_j$ corresponded to each other as described above

    \begin{Th}\label{Th1.13}
        Let $(\widetilde{W}, M_1, M_2)$ be a triple of manifolds satisfying Assumptions D-E-F as above. If there is a distinguished element $a\in H(CN(M_1,f_1,d_1))$ such that $NC(\widetilde{W})^*(a)$ is also distinguished in $H(CN(M_2,f_2,d_2))$ in similar sense, then 
            $$\rho^{[f_2]}_{M_2}(NC(\widetilde{W})^*(a)) \leq \rho^{[f_1]}_{M_1}(a).$$ 
        If there is no such notion of "distinguished", then we still have $\rho^{[f_2]}_{M_2} \leq \rho^{[f_1]}_{M_1}$. 
    \end{Th}

    \section{Equivariant Morse-Bott-Novikov complex}\label{Sec3}
    \subsection{Construction of the $S^1$-equivariant complex}\label{Sub3.1} Let $M$ be a compact $n$-dimensional smooth manifold with $G$-action. $G$ is assumed to have a CW complex structure. Suppose that $M$ has an $G$-invariant metric and $f: M \to S^1$ is an $G$-invariant Morse-Smale function. Furthermore, we also assume the following
    \begin{itemize}
        \item Given a point $x \in M$, its $Stab^G_x$ is either the $\{1\}$ or the entire $G$. 
        \item Among all the critical points of $f$, there is only one that is fixed under the $G$ action. Denoted by $\theta$ such a critical point.
    \end{itemize}
	Consider $\R \to S^1$ to be the infinite cyclic (universal) cover of $S^1$. Let $\widetilde{M}:=f^*(\R)=\{(x,\lambda)\in M \times \R: f(x) = e^{2i\pi \lambda}\}$ be the pull-back of $\R$ over $M$. Let $\widetilde{f} : \widetilde{M} \to \R$ be given by $\widetilde{f}(x,\lambda) = x$. Note that $\widetilde{f}$ is also $G$-invariant, where the action of $G$ on $\widetilde{M}$ is induced by $G$ acting on $M$ as assumed and $G$ acting on $\R$ trivially. Finally, we suppose that all critical points of $\widetilde{f}$ are isolated. Note that $(x,\lambda)$ is a critical point of $\widetilde{f}$ if and only if $x$ is a critical point of $f$. The stabilizer group of $(x,\lambda)$ is nothing but the stabilizer group $x$ in $G$.
 
 \begin{Lemma}\label{Lem3.1}
     $(x,\lambda)$ is a critical point of $\widetilde{f}$ if and only if $g\cdot(x,\lambda) = (g\cdot x, \lambda)$ is also a critical point for all $g \in G$.
 \end{Lemma}

 \begin{proof}
     $(\Leftarrow)$ is obvious. We will show $(\Rightarrow)$. Since $(x,\lambda)$ is a critical point of $\widetilde{f}$, $d_{(x,\lambda)} \widetilde{f} = 0$. Let $(x_t, \lambda_t)$ be a path in $\widetilde{M}$ for small real number $t$ such that $(x_0,\lambda_0) = (x,\lambda)$. Then for any $g \in G$, $(g \cdot x_t, \lambda_t)$ is small path in $\widetilde{M}$ where its starting point is $(g\cdot x, \lambda)$. By direct calculations, we have
     \begin{equation*}
         d_{(g\cdot x, \lambda)}\widetilde{f} = \dfrac{d}{dt}\widetilde{f}(g\cdot x_t, \lambda_t)\bigg|_{t=0} = \dfrac{d}{dt} \widetilde{f}(x_t,\lambda_t)\bigg|_{t=0} = 0,
     \end{equation*}
     which is what we want.
 \end{proof}

We say $\alpha$ is a critical orbit of $\widetilde{f}$ if and only if it is the $G$-orbit of a critical point $(x,\lambda)$ of $\widetilde{f}$. Note that by Lemma \ref{Lem3.1}, all points inside $\alpha$ are critical points of $\widetilde{f}$. From here onward, we shall assume that all critical orbits of $\widetilde{f}$ are isolated. If $G$ acts on $\widetilde{M}$ freely, then one should think of $\alpha$ as one of the non-degenerate critical points of the induced function $\widetilde{M}/G \to \R$ by $\widetilde{f}$. However, in our setup, there are fixed points of the $G$-action. We will deal with these two types of critical points separately. Since $\widetilde{M}$ is an infinite cyclic cover of $M$, there is a $\Z$-action on $\widetilde{M}$, where the generator $1$ acts on $M$ by the map $T\in Aut(\widetilde{M})$ given by $T(x,\lambda) = (x, \lambda-1)$.Suppose $(x,\lambda)$ is a critical point of $\widetilde{f}$. For any $k\in \Z$, we claim that $T^k(x,\lambda) = (x, \lambda -k )$ is another critical point. Indeed, let $(x_t, \lambda_t - k)$ be a path in $\widetilde{M}$ for small $t$ such that its starting point is $(x,\lambda- k)$. Then
\begin{equation*}
    d_{(x,\lambda-k)}\widetilde{f} = \dfrac{d}{dt}\widetilde{f}(x_t, \lambda_t - k)\bigg|_{t=0}= \dfrac{d}{dt}(\lambda_t - k)\bigg|_{t=0} = \dfrac{d}{dt} \lambda_t \bigg|_{t=0} = d_{(x,\lambda)}\widetilde{f} = 0.
\end{equation*}
From this observation, we see that any two lifts of a critical point of $f$ to universal cover $\widetilde{M}$ of $M$ are again critical points of $\widetilde{f}$, and they are only different from each other by a transformation $T^k$ for some $k\in \Z$. In particular, for the critical point $\theta$, denote one of its lifts by some $\Theta = (\theta, \lambda)$. Over a coefficient field $\mathbf{F}$, the abelian group generated by all of lifts of $\theta$ over the Novikov field $\Lambda_{\mathbf{F}}$ can be identified with $\Lambda_{\mathbf{F}}\{\Theta\}$. Furthermore, since $\Theta$ is a fixed point of $G$, the $G$-critical orbit of $\Theta$ is just a singleton set containing $\Theta$ itself. As a result, the abelian group generated by the critical orbits of all lifts of $\theta$ over the Novikov field $\Lambda_\mathbf{F}$ is also $\Lambda_{\mathbf{F}}\{\Theta\}$. Recall that $\Lambda_{\mathbf{F}}$ is defined as: Treating $T \in Aut(\widetilde{M})$ defined above as an indeterminate,
        \begin{equation*}
            \Lambda_{\mathbf{F}}:=\left\{\displaystyle \sum_{j=-\infty}^\infty a_j T^j : a_j \in \mathbf{F},\; \#\{j \leq 0: a_j \neq 0\} < \infty\right\}.
        \end{equation*}
	
Let $(x,\lambda)$ and $(x',\lambda')$ be two critical points of $\widetilde{f}$. Denote $\Breve{\mathcal{M}}((x,\lambda),(x',\lambda'))$ by the unparametrized moduli space of negative gradient flow lines from $(x,\lambda)$ to $(x',\lambda')$. For any $\gamma \in \Breve{\mathcal{M}}((x,\lambda),(x',\lambda'))$, this means that 
$$\dfrac{d\gamma}{dt}= -\nabla \widetilde{f}(\gamma(t)), \quad \displaystyle \lim_{t\to -\infty} \gamma(t) = (x,\lambda),\quad \lim_{t\to +\infty} \gamma(t) = (x',\lambda').$$
Denote $L = \dfrac{d}{dt} + \nabla\widetilde{f}$ by the first order differential operator from $Maps(\R, \widetilde{M}) \to T\widetilde{M}$. The action $G$ on the domain and codomain of $L$ is induced naturally by the action of $G$ on $M$. We claim the following

\begin{Lemma}\label{Lem3.2}
    The operator $L$ is $G$-equivariant.
\end{Lemma}
	
	\begin{proof}
	    Let $L_g: \widetilde{M} \to \widetilde{M}$ be an isometry of $\widetilde{M}$ that is given by the action of $g$ on $\widetilde{M}$ for each $g\in G$. $G$ acts on $T\widetilde{M}$ by the push-forward map $dL_g$. Let $\gamma_t = (x_t, \lambda_t)$ be a path in $\widetilde{M}$. By definition, $\dfrac{d}{dt} L_g (\gamma_t) = dL_g \left(\dfrac{d\gamma}{dt}\right)$. To show that equivariant-ness of $L$, we just have to show that $dL_g (\nabla \widetilde{f}(\gamma_t)) = \nabla \widetilde{f}(L_g(\gamma_t))$. It suffices for us to prove that for each point $p \in \widetilde{M}$, we have
        \begin{equation}\label{eq:3.1}
            dL_g(\nabla \widetilde{f}(p)) = \nabla \widetilde{f}(L_g(p)).
        \end{equation}
        Since $L_g$ is an isometry, we have
        \begin{align*}
            \la dL_g(\nabla \widetilde{f}(p)), dL_g(v_p)\ra = \la \nabla \widetilde{f}(p), v_p \ra = d_p \widetilde{f}(v_p).
        \end{align*}
        On the other hand, it is not difficult to see that $d_{L_g(p)}\widetilde{f} (dL_g(v_p)) = d_p \widetilde{f}(v_p)$ from the $G$-invariant-ness of $\widetilde{f}$. As a result, the left-hand side of the above identity is also equal to $\la \nabla \widetilde{f}(L_g(p)), dL_g(v_p)\ra$. Thus, \eqref{eq:3.1} is true as claimed.
	\end{proof}

 Lemma \ref{Lem3.2} then implies that if $\gamma$ is an unparametrized negative gradient flow lines connecting the critical points $(x,\lambda)$ and $(x',\lambda')$, then $g \cdot \gamma$ is another unparametrized negative gradient flow lines connecting $(g \cdot x, \lambda)$ and $(g \cdot x', \lambda')$. In other words, $G$ shuffles all negative gradient flow lines connecting points in critical orbit $\alpha = [(x,\lambda)]$ and points critical orbit $\beta = [(x,\lambda)]$. With this, we define $\Breve{\mathcal{M}}(\alpha, \beta)$ to be the moduli space of unparametrized flow lines that connect $\alpha$ and $\beta$, i.e., if $\gamma \in \Breve{\mathcal{M}}(\alpha, \beta)$, then
 $$L(\gamma) = 0, \quad \displaystyle \lim_{t\to -\infty} \gamma(t) \in \alpha, \quad \lim_{t \to +\infty} \gamma(t) \in \beta.$$
 From Lemma \ref{Lem3.2}, we know that $\Breve{\mathcal{M}}(\alpha,\beta)$ is a $G$-space. This fact also holds for the compactification $\Breve{\mathcal{M}}^+(\alpha,\beta)$ given by adding in the broken flow lines of $\Breve{\mathcal{M}}(\alpha,\beta)$. If we assume that the unstable manifold of the critical orbit $\alpha$ intersects transversally with the stable manifold of the critical orbit $\beta$, then $\Breve{\mathcal{M}}^+(\alpha, \beta)$ is a compact finite-dimensional manifold, where its dimension is given by
 $$\dim \Breve{\mathcal{M}}^+(\alpha,\beta) = \dim G/ Stab^G_x + \mu(\alpha) - \mu(\beta) - 1,$$
 where $\mu$ is the Morse index function. By a free critical orbit, we mean that the stabilizer group of one (any) of its representatives is a trivial subgroup of $G$.

 \begin{Lemma}\label{Lem3.3}
     If either $\alpha$ or $\beta$ is a free critical orbit of $\widetilde{f}$, then $G$ acts on $\Breve{\mathcal{M}}^+(\alpha, \beta)$ freely.
 \end{Lemma}

 \begin{proof}
     We argue in the case $\alpha$ is free. A similar argument applies in the case $\beta$ is free. Suppose $\gamma \in \Breve{\mathcal{M}}^+(\alpha, \beta)$ and let $g$ be an element in the stabilizer group of $\gamma$. Pick a representative $(x,\lambda)$ for $\alpha$ such that $\displaystyle \lim_{t\to -\infty}\gamma(t) = (x,\lambda) \in \alpha$. Since $g \cdot \gamma(t) = \gamma(t)$, by Lemma \ref{Lem3.2}, we must have $(g\cdot x, \lambda) = (x,\lambda)$. This implies that $g$ is in the stabilizer subgroup of $x$. But $\alpha$ is a free critical orbit; thus, $g$ can only be the identity.
 \end{proof}

 We have two maps $e^+_\alpha, e^-_\beta : \Breve{\mathcal{M}}^+(\alpha,\beta) \to \alpha, \beta$  defined by sending each $\gamma \in \Breve{\mathcal{M}}^+(\alpha,\beta)$ to its limits at the ends. Note that $e^+_\alpha, e^-_\beta$ are $G$-equivariant maps by Lemma \ref{Lem3.2}. From these maps, we define another $G$-equivariant map
 $$h_{\alpha\beta}: \Breve{\mathcal{M}}^+(\alpha,\beta) \to \alpha \times \beta, \quad h_{\alpha\beta} = (e^+_\alpha, e^-_\beta).$$
 
 If $\alpha$ and $\beta$ are both distinct free critical orbits, by fixing representative elements in $\alpha$ and $\beta$ respectively, then we can identify $\alpha \times \beta$ with $G\times G$. Since $h_{\alpha\beta}$ is $G$-equivariant, it descends to a map
 $$\widetilde{h}_{\alpha\beta}: \Breve{\mathcal{M}}^+(\alpha,\beta)/G \to (G\times G)/G \cong G.$$
 We would like to emphasize that at this point the definition of $\widetilde{h}_{\alpha\beta}$ depends on the choices of representatives of $\alpha, \beta$. On the other hand, if either $\alpha$ or $\beta$ is $\Theta$ and the remaining critical orbit is free, then we have two constant maps
 $$\widetilde{h}_{\alpha \Theta}: \Breve{\mathcal{M}}^+(\alpha, \Theta)/G \to \Theta, \quad \widetilde{h}_{\Theta \beta}: \Breve{\mathcal{M}}^+(\Theta, \beta)/G \to \Theta.$$
 Note that in each of the cases above, the dimensions of the moduli spaces are given by
 \begin{itemize}
     \item $\dim \Breve{\mathcal{M}}^+(\alpha, \beta)/G = \mu(\alpha) - \mu(\beta) - 1$,
     \item $\dim \Breve{\mathcal{M}}^+(\alpha, \Theta)/G = \mu(\alpha) - \mu(\Theta) - 1$,
     \item $\dim \Breve{\mathcal{M}}^+(\Theta, \beta)/G = \mu(\Theta) - \mu(\beta) - 1- \dim G$.
 \end{itemize}
For the last two cases, we are only interested in the situation when the dimension of the moduli spaces is equal to zero. We will come back to this later. 

So far, we set things up for some general CW-group $G$ and some general ground field $\mathbf{F}$. Now, we will look at the particular case when $G= S^1$ and $\mathbf{F} = \mathbf{F}_2$. Note that $H^*(S^1;\mathbf{F}_2) = \mathbf{F}^{(0)}_2\oplus \mathbf{F}^{(1)}_2$. In this case, when $\alpha, \beta$ are both distinct free critical orbits and $\mu(\alpha) = \mu(\beta) - 1$, equivalently, $\dim \Breve{\mathcal{M}}^+(\alpha,\beta) = 0$, we obtain the differential 
$$d: C_\ast = \Lambda_\mathbf{F}\{\alpha: \mu(\alpha) = \ast\} \to \Lambda_{\mathbf{F}}\{\beta: \mu(\beta) = \ast-1\}= C_{\ast -1},$$
$$d(\alpha) = \sum_{j=-\infty}^{\infty} \sum_{\beta} \#_{\mathbf{F}_2} \Breve{\mathcal{M}}^+(\alpha, T^j\beta)/G \cdot T^j \beta.$$
When $\alpha, \beta$ are free critical orbits and $\mu(\alpha) = \mu(\beta) + 2$, equivalently, $\dim \Breve{\mathcal{M}}^+(\alpha,\beta)/G = 1$, we obtain another map $u: C_\ast \to C_{\ast -2}$ defined by
$$\la u(\alpha), T^j \beta\ra = \la  h_{\alpha T^j\beta}^*(1), [\Breve{\mathcal{M}}^+(\alpha, T^j\beta)/G]\ra\, \mod 2.$$
Here $1$ denotes a generator of $H^1(S^1;\mathbf{F}_2)$. Next, we deal with the case when either of the critical orbits is $\Theta$. We have two maps
$$\delta_1 : C_\ast \to \Lambda_{\mathbf{F}}\{\Theta\}_{(\ast -1)}, \quad \quad \delta_2: \Lambda_{\mathbf{F}}\{\Theta\}_{(\ast)} \to C_{\ast - 2}$$
defined by $\la\delta_1(\alpha),\Theta\ra = \#_{\mathbf{F}_2}\Breve{\mathcal{M}}^+(\alpha,\Theta)/G$ and $\la \delta_2(\Theta),\beta\ra = \#_{\mathbf{F}_2} \Breve{\mathcal{M}}^+(\Theta, \beta)/G$. Having set up these maps, we define the following $\Lambda_{\mathbf{F}}$-finitely generated abelian groups
\begin{align}\label{eq:3.2}
    \widetilde{C}_{\ast}(M,f):= C_\ast \oplus C_{\ast -1} \oplus \Lambda_\mathbf{F}\{\Theta\}_{(\ast)}. 
\end{align}
The maps $d, u, \delta_1, \delta_2$ defined above can be gathered into one map that is given by the following matrix form
\begin{align}\label{eq:3.3}
    \widetilde{d}=\begin{bmatrix} d & 0 & 0 \\ u & d & \delta_2 \\ \delta_1 & 0 & 0 \end{bmatrix}.
\end{align}
Note that $\widetilde{d}: \widetilde{C}_\ast \to \widetilde{C}_{\ast-1}$ and it is a $\Lambda_{\mathbf{F}}$-linear map and $\deg \widetilde{d} = -1$ by construction.

\begin{Th}\label{Th3.4}
    $(\widetilde{C}(M,f), \widetilde{d})$ defined as in \eqref{eq:3.2} and \eqref{eq:3.3} is a differential complex.
\end{Th}

For $(\widetilde{C}(M), \widetilde{d})$ to be a differential complex, we need to show that $\widetilde{d}^2 = 0$. Equivalently, we must show the following
$$d^2 = 0,  \quad ud+du+\delta_2\delta_1 = 0, \quad d\delta_2 = 0, \quad  \delta_1 d = 0.$$
The fact $d^2=0$ follows from the standard argument that when the unparametrized moduli space of (broken) flow lines has dimension 1, it is a compact $1$-manifold with an even number of boundary points and $d^2$ counts the number of these exact boundary points $\mod 2$. Thus, the proof of Theorem \ref{Th3.4} follows from the following successive propositions that verify the remaining three identities above.

\begin{Prop}\label{Prop3.5}
    $d\delta_2 = 0$, and $\delta_1 d = 0$.
 \end{Prop}
	
\begin{proof}
    Let $\alpha$ be a free critical orbit of $\widetilde{f}$ such that $\mu(\Theta, \alpha)=3$. We need to show that $\la d\delta_2(\Theta), \alpha\ra = 0$. Note that
    \begin{equation}\label{eq:3.4}
        \la d \delta_2(\Theta), \alpha\ra = \displaystyle \sum_{\substack{\beta \text{ free critical orbit}\\\mu(\Theta,\beta) = 2}} \#_{\mathbf{F}_2} \Breve{\mathcal{M}}^+(\Theta, \beta)_1/S^1 \cdot \#_{\mathbf{F}_2} \Breve{\mathcal{M}}(\beta, \alpha)_1/S^1.
    \end{equation}
Furthermore, $\Breve{\mathcal{M}}^+(\Theta, \alpha)_2/S^1$ is a $1$-dimensional compact manifold with boundary. The proof is complete if we can show that the $\mod 2$ count of the boundary of $\Breve{\mathcal{M}}^+(\Theta, \alpha)_2/S^1$ is exactly the right-hand side of the above identity. The boundary of $\Breve{\mathcal{M}}^+(\Theta, \alpha)$ can be described by the co-dimension $1$ faces
$$\bigcup_{\beta} (\Breve{\mathcal{M}}^+(\Theta,\beta) \times_{\beta} \Breve{\mathcal{M}}^+(\beta,\alpha))/S^1,$$
where $\beta$ is another free critical orbit distinct from $\alpha$. There are two situations for $\beta$, either $\mu(\Theta, \beta) = 1$ or $\mu(\Theta, \beta)=2$. The case where $\mu(\Theta,\beta) = 1$ cannot happen because $\dim \Breve{\mathcal{M}}^+(\Theta, \beta) = -1 <0$. Thus, the only remaining possibility is that $\mu(\Theta, \beta) = 2$, and the count of all boundary points with such $\beta$ is exactly the right-hand side of \eqref{eq:3.7} as claimed. A similar argument also shows that $\delta_1 d = 0$. 
\end{proof}

\begin{Prop}\label{Prop3.6}
    $ud+du+\delta_2\delta_1 = 0$.
\end{Prop}

 \begin{proof}
     Let $\alpha, \beta$ be free critical orbits of $\widetilde{f}$ such that $\mu(\alpha, \beta):= \mu(\alpha) - \mu(\beta) = 3$, $\mu(\alpha, \Theta) = 1$, and $\mu(\Theta, \beta) = 2$. We need to show that $\la (ud+ du + \delta_2 \delta_1) \alpha, \beta \ra = 0$. Firstly, note that since $\mu(\alpha, \beta) = 3$, the moduli space $\Breve{\mathcal{M}}^+(\alpha, \beta)$ has dimension $3$. To indicate the dimension of the moduli space relative to the indices of critical points, we write $\Breve{\mathcal{M}}^+(\alpha, \beta)_3$. As a result, $\dim \Breve{\mathcal{M}}^+(\alpha,\beta)_3 /S^1 = 2$. Consider the map $\widetilde{h}_{\alpha \beta} : \Breve{\mathcal{M}}^+(\alpha,\beta)_3 /S^1 \to S^1$ constructed earlier and $p$ is any of its regular value. Then, $\widetilde{h}^{-1}_{\alpha\beta}(p)$ is a compact submanifold of dimension $1$ with boundary. Thus, the $\mod 2$ count of boundary points of $\widetilde{h}^{-1}_{\alpha \beta}(p)$ is always equal to zero. The proof is complete as soon as we can prove that
     $$\# \partial \widetilde{h}^{-1}_{\alpha \beta}(p) = \la (ud + du + \delta_2 \delta_1) \alpha, \beta \ra.$$
     
     By the definitions of the maps $\delta_1, \delta_2, u, d$, it is not difficult to see that
     \begin{align}
         \la \delta_2 \delta_1 (\alpha), \beta \ra & = \#_{\mathbf{F}_2}\Breve{\mathcal{M}}(\alpha,\Theta)_1/S^1 \cdot \#_{\mathbf{F}_2}\Breve{\mathcal{M}}^+(\Theta, \beta)_1/S^1, \label{eq:3.5}\\
         \la du(\alpha),\beta\ra &= \displaystyle \sum_{\substack{\gamma \text{ free critical orbit, }\\\mu(\alpha, \gamma) = 2, \mu(\gamma,\beta) = 1}} \la \widetilde{h}^*_{\alpha \gamma}(1), [\Breve{\mathcal{M}}^+(\alpha, \gamma)_2/S^1]\ra_{\mathbf{F}_2}\cdot \#_{\mathbf{F}_2} \Breve{\mathcal{M}}(\gamma,\beta)_1/S^1, \label{eq:3.6}\\
         \la ud(\alpha),\beta\ra &= \displaystyle \sum_{\substack{\gamma' \text{ free critical orbit,}\\ \mu(\alpha, \gamma')=1, \mu(\gamma', \beta) = 2}} \#_{\mathbf{F}_2}  \Breve{\mathcal{M}}(\alpha, \gamma')_1/S^1 \cdot \la \widetilde{h}^*_{\gamma' \beta}(1), [\Breve{\mathcal{M}}^+(\gamma,\beta)_2/S^1]\ra_{\mathbf{F}_2}. \label{eq:3.7}
     \end{align}
    Recall that if $f: M^n \to N^n$ is a smooth map between $n$-dimensional manifolds, then $\la f^*(1), [M]\ra_{\mathbf{F}_2} = \la 1, f_\ast [M]\ra_{\mathbf{F}_2} = \la 1, \deg_2(f)[N]\ra_{\mathbf{F}_2} = \deg_2(f) = \#_{\mathbf{F}_2} f^{-1}(y)$, where $1$ is the generator of the top cohomology of $N$ and $y$ is any regular value of $f$. Hence, we can think of $\la \widetilde{h}^*_{\alpha \gamma}(1), [\Breve{\mathcal{M}}^+(\alpha,\gamma)_2/G]\ra_{\mathbf{F}_2} = \deg_2(\widetilde{h}_{\alpha \gamma})$, and $\la \widetilde{h}^*_{\gamma' \beta}(1), [\Breve{\mathcal{M}}^+(\gamma,\beta)_2/G]\ra_{\mathbf{F}_2} = \deg_2(\widetilde{h}_{\gamma \beta})$. Next, the points in the boundary of $\widetilde{h}^{-1}_{\alpha \beta}(p)$ are exactly points in the boundary of $\Breve{\mathcal{M}}^+(\alpha,\beta)_3/S^1$ whose $\widetilde{h}_{\alpha \beta}$-values are $p\in S^1$. Note that a boundary of $\Breve{\mathcal{M}}^+(\alpha,\beta)/S^1$ belongs to one of the following three co-dimension $1$-faces
    \begin{itemize}
        \item Type I: $\bigcup_{\gamma \text{ free critical orbit, } \mu(\alpha, \gamma)=2, \mu(\gamma, \beta) =1} (\Breve{\mathcal{M}}^+(\alpha, \gamma)_2 \times_{\gamma} \Breve{\mathcal{M}}(\gamma,\beta)_1)/S^1$.
        \item Type II: $\bigcup_{\gamma' \text{ free critical orbit, } \mu(\alpha, \gamma')=1, \mu(\gamma', \beta) =2} (\Breve{\mathcal{M}}(\alpha, \gamma')_1 \times_{\gamma'} \Breve{\mathcal{M}}^+(\gamma',\beta)_2)/S^1$.
        \item Type III: $\Breve{\mathcal{M}}(\alpha,\Theta)_1/S^1 \times S^1 \times \Breve{\mathcal{M}}(\Theta, \beta)_1/S^1$.
    \end{itemize}
    The number of boundary points of type III corresponds exactly to \eqref{eq:3.5}. Suppose $(\mathfrak{a}_1, \mathfrak{a}_2) \in \Breve{\mathcal{M}}^+(\alpha, \gamma)_2/S^1 \times \Breve{\mathcal{M}}(\gamma,\beta)_1/S^1$ is a boundary point of type I such that its $\widetilde{h}_{\alpha \beta}$-value is $p$. Without loss of generality, assume $\widetilde{h}_{\gamma \beta} \equiv 1$ for any $\mu(\gamma,\beta) = 1$. Then the number of boundary points of this type whose values are $p$ is exactly the $\#_{\mathbf{F}_2} \widetilde{h}^{-1}_{\alpha \gamma}(p) \cdot \#_{\mathbf{F}_2} \Breve{\mathcal{M}}(\gamma, \beta)/S^1 = \deg_2(\widetilde{h}_{\alpha \gamma}) \cdot \#_{\mathbf{F}_2} \Breve{\mathcal{M}}(\gamma,\beta)/S^1$. Summing over all $\gamma$ corresponds to \eqref{eq:3.6}. We argue similarly for type II. Therefore, we achieve the identity as claimed.
 \end{proof}

 \begin{Rem}
     There is some gluing theory that makes various compactifications described above rigorous. We will ignore this technicality for now and will come back to it when we adapt this prototype model in a gauge theory setting.
 \end{Rem}

\subsection{Pull-up-push-down construction}\label{Sub3.2} The pull-up-push-down construction in finite dimension will be a prototype for the various induced cobordism maps in the infinite-dimensional setting that we will discuss later. Let $(M_i, f_i)_{i=0,1}$ be $G=S^1$-manifolds with the same assumptions as in the previous subsection. Suppose we have another $G$-manifold $W$ equipped with $S^1$-equivariant maps $r_i: W \to \widetilde{M_i}$. Define $r = (r_0, r_1) : W \to \widetilde{M_0} \times \widetilde{M_1}$. For any pair of critical points $(x_+,\lambda_+)$ and $(x_-, \lambda_-)$ of $\widetilde{f}_0$ and $\widetilde{f}_1$, respectively, we assume that $r$ is transversal to $U_{(x_+,\lambda_+)}\times S_{(x_-,\lambda_-)}$. If $z \in W$ such that $r(z) \in U_{(x_+,\lambda_+)}\times S_{(x_-,\lambda_-)}$, then by $S^1$-equivariant-ness of $r$, a similar argument as in Lemma \ref{Lem3.2} implies that for any $g \in S^1$, $r(g\cdot z) \in U_{(g\cdot x_+, \lambda_+)} \times S_{(g \cdot x_-, \lambda_-)}$. Thus, for a pair of critical orbit $\alpha_+, \alpha_-$ of $\widetilde{f}_0, \widetilde{f}_1$, respectively, we define $\mathscr{M}(W;\alpha_+, \alpha_-)$ to be all elements in $z \in W$ such that $r(z) \in U_{\alpha_+} \times S_{\alpha_-}$. Note that $\mathscr{M}(W; \alpha_+, \alpha_-)$ is an $S^1$-manifold. With a similar argument as in Lemma \ref{Lem3.3}, we see that if either $\alpha_{\pm}$ is a free critical orbit, then $S^1 \curvearrowright \mathscr{M}(W; \alpha_+, \alpha_-)$ freely. Furthermore, 
$$\dim \mathscr{M}(W; \alpha_+, \alpha_-) = \dim W - \dim M_0 + \dim S^1/ Stab^{S^1}_{\alpha_+} + \mu(\alpha_+) - \mu(\alpha_-).$$
Denote $\Theta_i$ to be the lift of the only $S^1$-invariant critical point $\theta_i$ of $f_i$, $i=0,1$. As in the previous subsection, we single out the dimension of the $S^1$-quotient of the moduli space in the following cases
\begin{itemize}
    \item $\alpha_{\pm}$ free critical orbits: $ \dim \mathscr{M}(W; \alpha_+, \alpha_-)/S^1 = \dim W - \dim M_0 + \mu(\alpha_+) - \mu(\alpha_-)$.
    \item $\alpha_+$ free critical orbit and $\Theta_1 = \alpha_-$: $\dim \mathscr{M}(W;\alpha_+ , \Theta_1)/S^1 = \dim W - \dim M_0 + \mu(\alpha) - \mu(\Theta_1)$.
    \item $\alpha_+ = \Theta_0$ and $\alpha_-$ free critical orbit: $\dim \mathscr{M}(W; \Theta_1, \alpha_-)/S^1 = \dim W - \dim M_0 + \mu(\Theta_0) - \mu (\alpha_-) -1$.
\end{itemize}
	
Let $\alpha_{\pm}$ be a pair of free critical orbits. By picking a representative element in each of $\alpha_{\pm}$, we have the following $S^1$-equivariant maps
$$h_{W; \alpha_+ \alpha_-}: \mathscr{M}(W;\alpha_+, \alpha_-)/S^1 \to S^1,$$
$$h_{W; \Theta_0\alpha_-} : \mathscr{M}(W; \Theta_0, \alpha_-)/S^1 \to \Theta_0, \quad h_{W; \alpha_+ \Theta_1}: \mathscr{M}(W; \alpha_+ \Theta_1)/S^1 \to \Theta_1.$$
If $\mu(\alpha_+, \alpha_-) = \dim M_0 - \dim W$, then the $\mod 2$ count of $\mathscr{M}(W; \alpha_+ ,\alpha_-)$ gives us a map $C(W) : C(M_0) \to C(M_1)$ where $\deg C(W) = \dim W - \dim M_0$. However, if $\mu(\alpha_+, \alpha_-) = \dim M_0 - \dim W + 1$, using the map $h_{W; \alpha_+ \alpha_-}$, we define a map $u(W): C(M_0) \to C(M_1)$  of degree equal to $\dim W - \dim M_0 - 1$ by
$$\la u(W)\alpha_+, \alpha_-\ra = \la h^*_{W; \alpha_+ \alpha_-}(1), [\mathscr{M}(W;\alpha_+,\alpha_-)/S^1]\ra \mod 2.$$
Here, $1$ is a the generator of $H^1(S^1;\mathbf{F}_2)$. In the case $\mu(\alpha_+, \Theta_1) = \dim M_0 - \dim W$ and $\mu(\Theta_0, \alpha_-) = \dim M_0 -\dim W + 1$, by counting ($\mod 2$) of the respective moduli spaces, we obtain
$$\delta_1(W): C(M_0) \to \Lambda_{\mathbf{F}_2}, \quad \delta_2(W): \Lambda_{\mathbf{F}_2} \to C(M_1).$$
Note that $\deg \delta_1(W) = \dim W - \dim M_0$ and $\deg \delta_2(W) = \dim W -\dim M_0 -1$. Gathering all of these maps in a matrix form, we obtain
\begin{equation}\label{eq:3.8}
    \widetilde{C}(W) = \begin{bmatrix}
        C(W) & 0 & 0 \\
        u(W) & C(W) & \delta_2(W) \\
        \delta_1(W) & 0 & 1
    \end{bmatrix},
\end{equation}
where $\widetilde{C}(W): \widetilde{C}(M_0,f_0) \to \widetilde{C}(M_1,f_1)$ and $\deg \widetilde{C}(W) = \dim W - \dim M_0 $.
	
\begin{Th}\label{Th3.8}
    $\widetilde{C}(W)$ defined in \eqref{eq:3.8} is a chain map.
\end{Th}

Denote $(d^i, u^i, \delta^i_1, \delta^i_2)_{i=0,1}$ be the maps associated with the $S^1$-equivariant complex of $(M_i, f_i)$ constructed in the previous subsection. For $\widetilde{C}(W)$ to be a chain map, we must show the following identities
\begin{align}
    d^1 C(W) - C(W) d^0 & = 0. \label{eq:3.9}\\
    d^1 u(W) - u(W) d^0 + u^1 C(W) - C(W) u^0 + \delta^1_2 \delta_1(W) - \delta_2(W) \delta^0_1 & = 0. \label{eq:3.10}\\
    d^1 \delta_2(W) + \delta^1_2 - C(W) \delta^0_2 & = 0. \label{eq:3.11}\\
    \delta^1_1 C(W) - \delta_1(W) d^0 - \delta^0_1 &= 0. \label{eq:3.12}
\end{align}

We denote $k = \dim M_0 - \dim W$. For any pair of critical orbits $\alpha_{\pm}$ of $\widetilde{f_0}, \widetilde{f_1}$, respectively, the moduli space $\mathscr{M}(W; \alpha_+, \alpha_-)/S^1$ is a compact manifold with boundary. If $\mu(\alpha_+, \alpha_-) = k + i$ and $i$ is some non-negative integer, then as we have seen before, $\dim \mathscr{M}(W; \alpha_+, \alpha_-)_{i+1}/S^1 = i$. A boundary point of $\mathscr{M}(W; \alpha_+, \alpha_-)_{i+1}/S^1$ belongs to the following co-dimension $1$ stratum of one of the below types
$$ \displaystyle \bigcup_{d \in \Z} \bigcup_{\gamma \text{ critical orbit of } \widetilde{f_0}} (\Breve{\mathcal{M}}^+(\alpha_+, \gamma)_{d+1} \times_\gamma \mathscr{M}(W; \gamma, \alpha_-)_{i-d})/S^1,$$
$$\displaystyle \bigcup_{d \in \Z} \bigcup_{\gamma \text{ critical orbit of } \widetilde{f_0}} (\Breve{\mathcal{M}}^+(\alpha_+, \gamma)_{i-d} \times_\gamma \mathscr{M}(W; \gamma, \alpha_-)_{d+1})/S^1,$$
$$\displaystyle \bigcup_{d\in \Z} \bigcup_{\gamma' \text{ critical orbit of }\widetilde{f_1}} (\mathscr{M}(W; \alpha_+, \gamma')_{d+1} \times_{\gamma'} \Breve{\mathcal{M}}^+(\gamma', \alpha_-)_{i-d})/S^1,$$
$$\displaystyle \bigcup_{d\in \Z} \bigcup_{\gamma' \text{ critical orbit of }\widetilde{f_1}} (\mathscr{M}(W; \alpha_+, \gamma')_{i-d} \times_{\gamma'} \Breve{\mathcal{M}}^+(\gamma', \alpha_-)_{d+1})/S^1,$$
$$\displaystyle \bigcup_{d_1 + d_2 = i} \Breve{\mathcal{M}}^+(\alpha_+, \Theta_1)_{d_1}/S^1 \times S^1 \times \mathscr{M}(W;\Theta_1, \alpha_-)_{d_2}/S^1,$$
$$\displaystyle \bigcup_{d_1 +d_2 = i} \mathscr{M}(W; \alpha_+, \Theta_2)_{d_1}/S^1 \times S^1 \times \Breve{\mathcal{M}}^{+}(\Theta_2, \alpha_-)_{d_2}/S^1.$$
The last two strata are derived from standard gluing theory. For example, if $i = 1$ and $\alpha_{\pm}$ is a free critical orbit, then for dimension reason, $d = 0$. Note that there is not enough dimension to pass through either of the invariant critical orbits. Furthermore, when $\gamma$ is a critical orbit of $\widetilde{f_0}$ such that $ \Breve{\mathcal{M}}^+(\alpha_+,\gamma)_1/S^1$ and $\mathscr{M}(W;\gamma,\alpha_-)/S^1$ have dimension zero, we must have
$$\mu(\alpha_+,\gamma) = 1, \quad \mu(\gamma, \alpha_-) = k + 1 - \dim S^1/Stab^{S^1}_\gamma.$$
As a result, $k + 1 = \mu(\alpha_+, \alpha_-) = k+2 - \dim S^1/Stab^{S^1}_{\gamma}$. This relation can only hold when $\gamma$ is a free critical orbit of $\widetilde{f_0}$. Similarly, we can argue that $\gamma'$ must also be a free critical orbit of $\widetilde{f_1}$ if $\dim \mathscr{M}(W; \alpha_+,\gamma')/S^1 = \dim \Breve{\mathcal{M}}^+(\gamma',\alpha_-)/S^1 = 0$. Therefore, the $\mod 2$ count of the boundary points of $\mathscr{M}(W; \alpha_+, \alpha_-)_2/S^1$ is given by
$$\displaystyle \underbrace{\sum_{\gamma} \#_{\mathbf{F}_2} \Breve{\mathcal{M}}(\alpha_+,\gamma)/S^1\cdot \#_{\mathbf{F}_2}\mathscr{M}(W; \gamma, \alpha_-)/S^1}_{\text{\normalfont $\la C(W)d^0(\alpha_+),\alpha_-\ra$}} + \underbrace{\sum_{\gamma'}  \#_{\mathbf{F}_2}\mathscr{M}(W; \alpha_+, \gamma')/S^1 \cdot \Breve{\mathcal{M}}(\gamma',\alpha_-)/S^1}_{\text{\normalfont $\la d^1C(W)(\alpha_+),\alpha_-\ra$}}.$$
Since there is always an even number of boundary points of $1$-dimensional compact manifold, the above proves exactly \eqref{eq:3.9}. We sum up this discussion in the form of the following proposition.

\begin{Prop}\label{Prop3.9}
   $d^1C(W) - C(W)d^0 = 0$.
\end{Prop}
	
A similar argument as the above and in the proof of Proposition \ref{Prop3.5} will also show \eqref{eq:3.11} and \eqref{eq:3.12}.

\begin{Prop}\label{Prop3.10}
    $d^1 \delta_2(W) + \delta^1_2 - C(W) \delta^0_2  = 0$,  $\delta^1_1 C(W) - \delta_1(W) d^0 - \delta^0_1 = 0$.
\end{Prop}

The proof of Theorem \ref{Th3.8} will be complete if we can show the following proposition.

\begin{Prop}\label{Prop3.11}
    $d^1 u(W) - u(W) d^0 + u^1 C(W) - C(W) u^0 + \delta^1_2 \delta_1(W) - \delta_2(W) \delta^0_1 =0$.
\end{Prop}

\begin{proof}
    Set $i = 2$ and consider the moduli space $\mathscr{M}(W; \alpha_+, \alpha_-)_{3}/S^1$, where $\alpha_\pm$ is any pair of free critical orbits of $\widetilde{f_0}, \widetilde{f_1}$, respectively. Let $p$ be a regular value of $S^1$. Then $h^{-1}_{W; \alpha_+, \alpha_-}(p)$ is a compact $1$-dimensional submanifold of $\mathscr{M}(W; \alpha_+, \alpha_-)_3/S^1$ with boundary. A different $\mod 2$ count of the boundary points of $h^{-1}_{W;\alpha_+,\alpha_-}(p)$ will give us exactly the desired relation. In other words, we have to count the boundary points of $\mathscr{M}(W; \alpha_+,\alpha_-)_3/S^1$ whose $h_{W;\alpha_+,\alpha_-}$-values are $p$. Note that for a dimensional reason, $d$ can only be either $0$ or $1$.

    Firstly, suppose that $d=0$. Suppose $\gamma$ is a critical orbit of $\widetilde{f_0}$ such that $\Breve{\mathcal{M}}^+(\alpha_+, \gamma)_1/S^1$ has dimension $0$ and $\mathscr{M}(W; \gamma, \alpha_-)_2/S^1$ has dimension $1$. Then
    $$1= \mu(\alpha_+,\gamma), \quad 2 = -k + \dim S^1/Stab^{S^1}_{\gamma} + \mu(\gamma, \alpha_-).$$
    The two equalities above imply that $3 = \dim S^1/Stab^{S^1}_\gamma + 2$, which can only be true when $\gamma$ is a free orbit. Let $(\mathfrak{a}_1, \mathfrak{a}_2)$ be a boundary point belonging to $(\Breve{\mathcal{M}}^+(\alpha_+, \gamma)_1 \times \mathscr{M}(W;\gamma, \alpha_-)_2)/S^1$ such that $h_{W;\alpha_+,\alpha_-}(\mathfrak{a}_1,\mathfrak{a}_2) = p$. Since $\widetilde{h}^{1}_{\alpha_+ \gamma}\equiv 1$, $h_{W; \gamma, \alpha_-}(\mathfrak{a}_2) = p$.  The $\mod 2$ count of all such boundary points summing over all free critical orbit $\gamma$ such that $\mu(\alpha_+,\gamma)=1$ and $\mu(\gamma,\alpha_-) = k+1$ is given by
    $$\displaystyle \underbrace{\sum_{\substack{\gamma \text{ free critical orbit of } \widetilde{f_0}\\ \mu(\alpha_+,\gamma)=1, \mu(\gamma, \alpha_-) = k+1}} \#_{\mathbf{F}_2} \Breve{\mathcal{M}}(\alpha_+,\gamma)_1/S^1 \cdot \deg_2(h_{W;\gamma,\alpha_-})}_{\text{\normalfont $\la u(W)d^0(\alpha_+),\alpha_-\ra$}}.$$
    If $\gamma$ is a critical orbit of $\widetilde{f_0}$ such that $\Breve{\mathcal{M}}^+(\alpha_+,\gamma)_2/S^1$ has dimension 1 and $\mathscr{M}(W;\gamma,\alpha_-)_1/S^1$ has dimension 0, then argue similarly, $\gamma$ must also be a free critical orbit. Furthermore, the $\mod 2$ count of all boundary points summing over all such $\gamma$ is given by
    $$\displaystyle \underbrace{\sum_{\substack{\gamma \text{ free critical orbit of } \widetilde{f_0}\\ \mu(\alpha_+,\gamma)=2, \mu(\gamma, \alpha_-) = k}}  \deg_2(h^0_{\alpha_+\gamma})\cdot \#_{\mathbf{F}_2}\mathscr{M}(W;\gamma,\alpha_-)}_{\text{\normalfont $\la C(W)u^0(\alpha_+),\alpha_-\ra$}}.$$
    With the same argument, we also obtain the $\mod 2$-count of pre-image of $p$ in each of the stratum $\bigcup_{\gamma'} (\mathscr{M}(W;\alpha_+, \gamma')_1 \times_{\gamma'} \Breve{\mathcal{M}}^+(\gamma',\alpha_-)_2)/S^1$ and $\bigcup_{\gamma'} (\mathscr{M}(W;\alpha_+,\gamma')_2 \times_{\gamma'} \Breve{\mathcal{M}}^+(\gamma',\alpha_-)_1)/S^1$, respectively, is given by
    \begin{align*}
        \la u^1C(W)\alpha_+,\alpha_-\ra & = \displaystyle \sum_{\substack{\gamma' \text{ free critical orbit of } \widetilde{f_1} \\ \mu(\alpha_+,\gamma')=k, \mu(\gamma',\alpha_-)=2}} \#_{\mathbf{F}_2} \mathscr{M}(W; \alpha_+, \gamma')_1/S^1 \cdot \deg_2(h^1_{\gamma' \alpha_-}),\\
        \la d^1 u(W) \alpha_+, \alpha_- \ra & = \displaystyle \sum_{\substack{\gamma' \text{ free critical orbit of }\widetilde{f_1}\\ \mu(\alpha_+,\gamma') = k+1, \mu(\gamma',\alpha_-) = 1}} \deg_2(h_{W;\alpha_+,\gamma'})\cdot \#_{\mathbf{F}_2}\Breve{\mathcal{M}}(\gamma',\alpha_-)_1/S^1.
    \end{align*}
    
    If $d=1$, we obtain the same type of boundary points as above and thus can be omitted. Next, we consider a boundary point of the type belonging to the stratum 
    $$\Breve{\mathcal{M}}^+(\alpha_+,\Theta_0)_{d_1}/S^1 \times S^1 \times \mathscr{M}(W;\Theta_0,\alpha_-)_{d_2}/S^1.$$
    Since $d_1 + d_2 = i = 2$ and $d_1, d_2 \geq 1$, $d_1 = d_2 = 1$. Without loss of generality, we assume that $h_{W; \alpha+,\alpha_-}$ restricted to the above stratum is just the projection onto the $S^1$-factor. As a result, the $\mod 2$-count of the pre-image of $p$ in this stratum is given by
    $$\la \delta_2(W)\delta^0_1(\alpha_+),\alpha_-\ra = \#_{\mathbf{F}_2} \Breve{\mathcal{M}}^+(\alpha_+, \Theta_1)_1/S^1 \cdot \#_{\mathbf{F}_2} \mathscr{M}(W;\Theta_1,\alpha_-)_1/S^1.$$
    Similarly, we can show that $\la \delta^1_2 \delta_1(W)(\alpha_+),\alpha_-\ra$ is the $\mod 2$-count of pre-image of $p$ in the stratum $\mathscr{M}(W;\alpha_+,\Theta_1)_1/S^1 \times S^1 \times \Breve{\mathcal{M}}^+(\Theta_1,\alpha_-)_1/S^1$. 
\end{proof}

\subsection{Well-definedness of the construction}\label{Sub3.3}
Recall that the construction of the $S^1$-equivariant complex $(\widetilde{C}(M, f), \widetilde{d})$ and the induced chain map $\widetilde{C}(W): \widetilde{C}(M_0, f_0) \to \widetilde{C}(M_1, f_1)$ in the previous subsections depends on the choice of representative of a critical orbit. Furthermore, in the proof Theorem \ref{Th3.4} and Theorem \ref{Th3.8}, we make the assumption (H) that $\widetilde{h}_{\alpha\beta} = h_{W; \alpha_+ \alpha_-} \equiv 1$ whenever $\mu(\alpha, \beta) = 1$ and $\mu(\alpha_+, \alpha_-) = k$. In this subsection, we will show that $\widetilde{C}$-complex is independent of the choices made, and (H) can be assumed for certain choices of representative elements of the corresponding critical orbits.

Let $\alpha, \beta$ be free critical orbits of $\widetilde{f}$. Note that by choosing  representatives $x, x'$ of $\alpha$ and $\beta$, respectively, we can identify $\alpha, \beta$ with $S^1$ via the $S^1$-equivariant maps $\phi_x(g\cdot x) = g$, $\phi_{x'}(g\cdot x') = g$. Then, there is a standard action of $S^1$ on $\alpha \times \beta$ induced by the multiplication of $S^1$ on each factor of $S^1 \times S^1$. Furthermore, $S^1 \curvearrowright \alpha \times \beta$ freely so that $(\alpha \times \beta)/S^1$ is without singularity. Consider the following map
$$\psi_{xx'}: \dfrac{\alpha \times \beta}{S^1} \cong \dfrac{S^1 \times S^1}{S^1} \to S^1, \quad \quad \psi_{xx'}[h,f] = h^{-1}f.$$
In terms of $\phi_x, \phi_{x'}$, we can write $\psi_{xx'}[z,z']  = (\phi_x(z))^{-1} \phi_{x'}(z')$.

\begin{Lemma}\label{Lem3.12}
    $\psi_{xx'}$ is well-defined. Furthermore, $\psi_{xx'}$ is a diffeomorphism.
\end{Lemma}

\begin{proof}
    Suppose $(h',f')$ is another representative of $[h,f]$. Then there is a $g \in S^1$ such that $h'=gh, f' = gf$. Observe that $h'^{-1} f' = h^{-1}g^{-1} g f = h^{-1} f$. This shows that the map $\psi_{xx'}$ is indeed well-defined.

    Note that for $[h',f'] = [h,f]$, there must be a $g \in S^1$ such that $gh=h', gf = f'$. Equivalently, this means that $h^{-1}h' = f^{-1}f'$. Using the fact that $S^1$ is an abelian group, we re-write the aforementioned condition as $h' f = h f'$. Now, suppose $\psi_{xx'}[h,f] = \psi_{xx'}[h',f']$, which means that $h^{-1}f = h'^{-1}f'$. We multiply both sides of the equation by $hh'$ and re-arrange to obtain $h'f = hf'$. By the previous observation, we must have $[h,f]=[h',f']$. As a result, $\psi_{xx'}$ is injective. On the other hand, for any $g \in S^1$, we always have $\psi_{xx'} [1,g] = g$. Hence, $\psi_{xx'}$ is also subjective. It's not hard to see that $\psi_{xx'}$ is smooth and has a smooth inverse. Therefore, $\psi_{xx'}$ is a diffeomorphism as claimed. 
\end{proof}

By Lemma \ref{Lem3.12}, for each representative element $x,x'$ of $\alpha, \beta$, respectively, we have a well-defined smooth map $\widetilde{h}^{xx'}_{\alpha\beta} : \Breve{\mathcal{M}}(\alpha, \beta)/S^1 \to S^1$ defined by 
$$\widetilde{h}^{xx'}_{\alpha \beta}[\gamma] = (\phi_x(e^-_{\alpha}(\gamma)))^{-1}\phi_{x'}(e^+_{\beta}(\gamma)).$$
Let $y$ be another representative of $\alpha$. There is a unique $h_{xy} \in S^1$ such that $x = h_{xy} y$. Denote $h_{xy}$ by the automorphism of $S^1$ defined by multiplication by $h_{xy}$. Similarly, let $y'$ be another representative of $\beta$ and we also obtain another automorphism $h_{x'y'}$ of $S^1$. Then we have the following commutative diagrams
\[
  \begin{tikzcd}
    \alpha \arrow{r}{\phi_x} \arrow[swap]{dr}{\phi_y} & S^1 \arrow{d}{h_{xy}} \\
     & S^1
  \end{tikzcd} \quad \begin{tikzcd}
    \beta \arrow{r}{\phi_{x'}} \arrow[swap]{dr}{\phi_{y'}} & S^1 \arrow{d}{h_{x'y'}} \\
     & S^1
  \end{tikzcd} \quad \begin{tikzcd}
    \Breve{\mathcal{M}}(\alpha,\beta)/S^1 \arrow{r}{\widetilde{h}^{xx'}_{\alpha \beta}} \arrow[swap]{dr}{\widetilde{h}^{yy'}_{\alpha\beta}} & S^1 \arrow{d}{h^{-1}_{xy}h_{x'y'}} \\
     & S^1
  \end{tikzcd}
\]
The above discussion also extends to the compactification $\Breve{\mathcal{M}}^+(\alpha,\beta)/S^1$.

\begin{Lemma}\label{Lem3.13}
    We always obtain the same $u$-map for $\widetilde{C}(M,f)$ regardless of the choice of representatives of any pair of free critical orbits $(\alpha, \beta)$ of $\widetilde{f}$, where $\mu(\alpha,\beta) = 2$. 
\end{Lemma}

\begin{proof}
    Let $x, y$ be two different representatives of $\alpha$. Similarly define $x',y'$ for $\beta$. It is sufficient for us to show that $\deg_2(\widetilde{h}^{xy}_{\alpha\beta}) = \deg_2(\widetilde{h}^{x'y'}_{\alpha\beta})$. By the commutative of the above diagram, we have $(h^{-1}_{xy}h_{x'y'})\widetilde{h}^{xx'}_{\alpha \beta} = \widetilde{h}^{yy'}_{\alpha\beta}$. Since $h^{-1}_{xy}h_{x'y'}$ is an automorphism of $S^1$, its $\deg_2$ is always equal to $1$. By the composition rule of $\deg_2$, we obtain the result as claimed.
\end{proof}

Now, we deal with the assumption (H) stated at the beginning of this subsection. This assumption was used extensively in the proof of Proposition \ref{Prop3.5} and Proposition \ref{Prop3.6}. Recall if we have $\mu(\alpha,\beta) = 3$, then $\dim \Breve{\mathcal{M}}^+(\alpha,\beta)_3/S^1 = 2$. Pick any representative $x,x''$ of $\alpha$ and $\beta$ so that we can write down the map $\widetilde{h}^{xx''}_{\alpha\beta} : \Breve{\mathcal{M}}^+(\alpha,\beta)_3/S^1 \to S^1$.  We pick $p$ to be any regular value of $\widetilde{h}^{xx''}_{\alpha\beta}$ and we wish to count the boundary $(\widetilde{h}^{xx''}_{\alpha \beta})^{-1}(p)$. This amounts to the same thing as counting the boundary points of $\Breve{\mathcal{M}}^+(\alpha,\beta)_3/S^1$ whose have $p$-value in $S^1$. The boundary of $\Breve{\mathcal{M}}^+(\alpha,\beta)_3/S^1$ is given by certain kinds of co-dimension $1$ stratum in the compactification of $\Breve{\mathcal{M}}(\alpha,\beta)_3/S^1$. One such stratum is given by, say,
$$(\Breve{\mathcal{M}}(\alpha,\gamma)_{1}\times_{\gamma} \Breve{\mathcal{M}}^+(\gamma,\beta)_{2})/S^1.$$
Here $\gamma$ is another free critical orbit. Pick $x'$ to be any representative of $\gamma$. Let $(\mathfrak{a}_1, \mathfrak{a}_2)$ be a boundary point of $\Breve{\mathcal{M}}^+(\alpha,\beta)_3/S^1$ belonging to the above stratum and $\{\mathfrak{c}_n\}$ be a sequence in the moduli space that approaches to $(\mathfrak{a}_1,\mathfrak{a}_2)$. Suppose $p=\displaystyle \lim_{n\to \infty} \widetilde{h}^{xx''}_{\alpha\beta} (\mathfrak{c}_n) = \widetilde{h}^{xx''}_{\alpha\beta}(\mathfrak{a}_1, \mathfrak{a}_2)$. This implies that
$(\phi_x(e^-_{\alpha}(\mathfrak{a}_1)))^{-1}\phi_{x''}(e^+_\beta(\mathfrak{a}_2)) = p$. Note that by gluing, we have $\phi_{x'}(e^+_\gamma (\mathfrak{a}_2)) = \phi_{x'}(e^-_{\gamma}(\mathfrak{a}_2))$. Thus,
$$p = (\phi_x(e^-_{\alpha}(\mathfrak{a}_1)))^{-1}\phi_{x''}(e^+_\beta(\mathfrak{a}_2)) = \widetilde{h}^{xx'}_{\alpha \gamma}(\mathfrak{a}_1) \cdot \widetilde{h}^{x'x''}_{\gamma \beta}(\mathfrak{a}_2).$$
Hence, for each $\mathfrak{a}_1 \in \Breve{\mathcal{M}}(\alpha,\gamma)_1/S^1$, for the above equation to be true, $\mathfrak{a}_2 \in \Breve{\mathcal{M}}^+(\gamma,\beta)_2/S^1$ has to satisfy $\widetilde{h}^{x'x''}_{\gamma\beta}(\mathfrak{a}_2) = (\widetilde{h}^{xx'}_{\alpha \gamma}(\mathfrak{a}_1))^{-1}(p)$. On the other hand, we have the following general facts.

\begin{Lemma}\label{Lem3.14}
    Suppose $f: M \to N$ is a smooth map between smooth manifolds of the same dimension. Let $g: N \to N$ be any diffeomorphism of $N$. Then for any $p$ that is a regular value of $M$, $g(p)$ is also another regular value.
\end{Lemma}

\begin{proof}
    Consider the pullback of $f, T$ over $N$ that is given by $M\times _{f, N, g} N = \{(x,y) \in M\times N: f(x) = g(y)\}$. Then we have the following commutative diagram
    \[\begin{tikzcd}
M\times_{f,N,g} N \arrow{r}{p_2} \arrow[swap]{d}{p_1} & N \arrow{d}{g} \\%
M \arrow{r}{f}& N
\end{tikzcd}\]
Here $p_1$ and $p_2$ denote the projection of the pullback onto the first and second factors, respectively. Let $x \in f^{-1}(g(p))$, where $p$ is a regular value of $f$. This means that $(x,p) \in M\times_{f,N,g} N$. By the commutative of the above diagram, $f\circ p_1 = g \circ p_2$. We differentiate this equation at $(x,p)$ to obtain $d_x f \circ d_{(x,p)}p_1 = d_p g \circ d_{(x,p)}p_2$. Since $p$ is a regular value, both $p_1$ and $p_2$ have full rank. As a result, $\text{rank}\, d_x f = \text{rank}\, d_p g$. But $g$ is a diffeomorphism, hence, $f$ also has full rank at $x$. This shows that $g(p)$ is a regular value of $f$ as claimed. 
\end{proof}

Since $(\widetilde{h}^{xx'}_{\alpha \gamma}(\mathfrak{a}_1))^{-1}$ can be viewed as a diffeomorphism of $S^1$ interpreted as a map by taking every point in $S^1$ and multiply it with $(\widetilde{h}^{xx'}_{\alpha \gamma}(\mathfrak{a}_1))^{-1}$, by Lemma \ref{Lem3.14}, $(\widetilde{h}^{xx'}_{\alpha \gamma}(\mathfrak{a}_1))^{-1}p$ is another regular value of $\widetilde{h}^{x'x''}_{\gamma\beta}$. Therefore, the number of points $\mod 2$ in the moduli space $\Breve{\mathcal{M}}(\gamma,\beta)_2/S^1$ whose values are $(\widetilde{h}^{xx'}_{\alpha \gamma}(\mathfrak{a}_1))^{-1}p$ is still exactly $\deg_2 (\widetilde{h}^{x'x''}_{\gamma \beta})$. Thus, it makes no difference in the construction of the $S^1$-equivariant complex if we assume (H). The same justification can be given for the well-definedness of the construction of the chain map $\widetilde{C}(W)$.

\subsection{The category of $S^1$-equivariant complexes}\label{Sub3.4}
In Daemi and Scaduto's work \cite{daemi2020equivariantaspectssingularinstanton, MR4742808}, the $S^1$-equivariant differential complex defined by \eqref{eq:3.2}, \eqref{eq:3.3} is called an \textit{$\mathcal{S}$-complex}. The collection of all $\mathcal{S}$-complexes form a category, which we will now describe.

\begin{Def}\label{Def3.15}
    Let $\mathscr{S}$ be the category of $\mathcal{S}$-complexes. An \textit{object} in $\mathscr{S}$ is a $\widetilde{C}$-complex that is $\Z$-graded, finitely generated over the Novikov field $\Lambda_{\mathbf{F}_2}$, and of the form
$$\widetilde{C}_\ast = C_\ast \oplus C_{\ast -1} \oplus \Lambda_{\mathbf{F}_2}.$$
Here $(C_\ast, d)$ is some $\Z$-graded complex over $\Lambda_{\mathbf{F}_2}$. The differential $\widetilde{d}$ of $\widetilde{C}$ is given by the following matrix
$$\widetilde{d} = \begin{bmatrix} d & 0 & 0 \\ u & d & \delta_2 \\ \delta_1 & 0 & 0 \end{bmatrix}.$$
For some $k\in \Z$, a \textit{morphism} $\widetilde{\lambda}: \widetilde{C}_\ast = C_\ast \oplus C_{\ast -1} \oplus \Lambda_{\mathbf{F}_2} \to \widetilde{C'} = C'_{\ast+k} \oplus C'_{\ast -1+k} \oplus \Lambda_{\mathbf{F}_2}$ is a chain map of degree $k$ that has the form
$$\widetilde{\lambda} = \begin{bmatrix} \lambda & 0 & 0 \\ \eta & \lambda & \Delta_2 \\ \Delta_1 & 0 & 1 \end{bmatrix}.$$
\end{Def}

\begin{Rem}\label{Rem3.16}
    $\widetilde{d}$ being a differential implies that we have the following identities
    \begin{align*}
        d^2 & = 0,\\
        ud + du + \delta_2 \delta_1 & = 0,\\
        d\delta_2 & = 0,\\
        \delta_1 d & = 0.
    \end{align*}
    While $\widetilde{\lambda}$ being a chain map means that we have
    \begin{align*}
        d'\lambda - \lambda d & = 0,\\
        d' \eta - \eta d + u' \lambda - \lambda u + \delta'_2 \Delta_1 - \Delta_2 \delta_1 & = 0,\\
        d'\Delta_2 + \delta'_2 - \lambda \delta_2 & = 0,\\
        \delta'_1 \lambda - \Delta_1 d - \delta_1 & = 0.
    \end{align*}
\end{Rem}

\begin{Def}\label{Def3.17}
    Given a pair of objects $(\widetilde{C}, \widetilde{C'})$ in $\mathscr{S}$ together with two morphisms $\widetilde{\lambda}_1, \widetilde{\lambda}_2: \widetilde{C} \to \widetilde{C'}$. An \textit{$\mathscr{S}$-chain homotopy} between $\widetilde{\lambda}_1$ and $\widetilde{\lambda}_2$ is a map $h : \widetilde{C} \to \widetilde{C'}$ (\text{not} a morphism) that satisfies
    $$\widetilde{\lambda}_1 - \widetilde{\lambda}_2 = \widetilde{d'} h + h \widetilde{d}.$$
    Given this, we say that two objects $\widetilde{C}, \widetilde{C'}$ are \textit{$\mathscr{S}$-chain homotopy equivalent} to each other if there are two morhisms $\widetilde{\lambda}_1 : \widetilde{C} \to \widetilde{C'}$ and $\widetilde{\lambda}_2: \widetilde{C'} \to \widetilde{C}$ such that $\widetilde{\lambda}_1\circ \widetilde{\lambda}_2$ and $\widetilde{\lambda}_2\circ \widetilde{\lambda}_1$ are $\mathscr{S}$-chain homotopy equivalence to the respective identity morphisms.
\end{Def}

\begin{Rem}\label{Rem3.18}
    If $h$ is an $\mathscr{S}$-chain homotopy, then with respect to the decompositions of $\mathcal{S}$-complexes $h$ has a matrix form $(h_{ij})_{1\leq i, j \leq 3}$. Based on the relation that $\widetilde{\lambda}_1 - \widetilde{\lambda}_2 = \widetilde{d'}h + h \widetilde{d}$, we must have $h_{12} = h_{13} = h_{32} = h_{33} = 0$. As a result, in matrix form, an $\mathscr{S}$-chain homotopy $h$ is written as
    $$ h = \begin{bmatrix} L & 0 & 0 \\ N & L & D_2 \\ D_1 & 0 & 0 \end{bmatrix}.$$
\end{Rem}

\begin{Def}\label{Def3.19}
We denote $K(\mathscr{S})$ by the \textit{homotopy category of $\mathcal{S}$-chain complexes} where its objects are still the objects of $\mathscr{S}$ and morphisms are given by morphisms of $\mathscr{S}$ modulo $\mathscr{S}$-chain homotopy equivalence. 
\end{Def}

As seen in the previous subsections, given an $S^1$-manifold $M$ equipped with an $S^1$-invariant Morse-Smale function $f: M \to S^1$, we construct an $\mathcal{S}$-complex $\widetilde{C}(M, f)$. We also assume that there is an $S^1$-invariant metric on $M$. Let $\mathscr{P}$ denote the parameter space containing the auxiliary data that goes into the construction of $\widetilde{C}(M,f)$. If $\mathscr{P}$ is connected, then by a continuity argument, the image of the assignment $M \mapsto [\widetilde{C}(M,f)] \in K(\mathscr{S})$ is an invariant of $M$. Similarly, the homotopy category $K(\mathscr{S})$ is where we are going to define an invariant of a $\Q HS^3$ using Seiberg-Witten gauge theory. This process will be carried out in the subsequent sections.

\section{Floer homology from Seiberg-Witten gauge theory}\label{Sec4}
\subsection{Set-up}\label{Sub4.1} Most of what is written here is well-known and can be found in greater detail in \cite{MR2388043}. We briefly recall the setting of Seiberg-Witten gauge theory in this subsection for the sake of self-containment. Let $(Y^3,g)$ be a closed smooth oriented $3$-manifold, and $g$ is a fixed Riemannian metric on $Y$. For now, we are not assuming any topology on $Y$. Since all $3$-manifold is $spin^c$, we pick a $spin^c$ structure $\mathfrak{s}$ on $Y$. Denote $P_{\mathfrak{s}} = \mathfrak{s} \times_{Spin^c(3)} \mathbf{C}^2$ by the associated spinor bundle over $Y$. It is a standard fact that $P_\mathfrak{s} \to Y$ is a Clifford module over $Y$, i.e., 
\begin{itemize}
    \item There is a Clifford multiplication $\rho : TY\otimes \mathbf{C} \to End(P_{\mathfrak{s}})$ that is compatible with the metric $g$ on $Y$.
    \item There is a covariant derivative $\nabla_A$ on $P_{\mathfrak{s}}$ determined by the Levi-Civita connection induced from $g$ and a $U(1)$-connection $A$ on the determinant line bundle $\det\,\mathfrak{s} = \mathfrak{s}/Spin(3)$ associated with the $spin^c$ structure $\mathfrak{s}$.
\end{itemize}
These two geometric data let us define the Dirac operator $D_A := \rho \circ \nabla_A : \Gamma(P_\mathfrak{s}) \to \Gamma(P_\mathfrak{s})$. Note that if we fix $A_0$ to be a referenced $U(1)$-connection, then any other connection $A$ can be written as $A_0 + a$, where $a \in i\Omega^1(Y)$. We consider the following Chern-Simons-Dirac functional.

\begin{Def}\label{Def4.1}
    Let $\mathcal{C}(Y, \mathfrak{s}) = \mathcal{A}(det\,\mathfrak{s}) \times \Gamma(P_\mathfrak{s})$. Define $\mathcal{L} : \mathcal{C}(Y,\mathfrak{s}) \to \R$ to be
    $$\mathcal{L}(A,\psi) = -\dfrac{1}{8}\int_Y (A -A_0) \wedge (F_A + F_{A_0}) + \dfrac{1}{2}\int_Y \la D_A \psi, \psi\ra.$$
\end{Def}

Recall that we have a gauge group action $\mathcal{G}(Y) := Maps(Y, U(1)) \curvearrowright \mathcal{C}(Y,\mathfrak{s})$ given by $(\sigma, (A, \psi)) \mapsto (A - \sigma^{-1}d\sigma, \sigma \psi)$. In general, it is not true that $\mathcal{L}$ is $\mathcal{G}(Y)$-invariant. However, if $c_1(\mathfrak{s})$ is torsion or $Y$ is $\Q HS^3$, i.e., $b_1(Y) = 0$, then $\mathcal{L}$ is indeed $\mathcal{G}(Y)$-invariant. The same statement also holds for the based gauge group $\mathcal{G}_0(Y) := \{ \sigma \in \mathcal{G}(Y) : \sigma(x_0) = 1\}$, for some fixed $x_0 \in Y$. Thus, from now on, we shall work with rational homology sphere $Y$ (unless otherwise stated). As a result, $\mathcal{L}$ descends to a functional
$$\mathcal{L} : \mathcal{B}_0(Y,\mathfrak{s}) := \mathcal{C}(Y,\mathfrak{s})/\mathcal{G}_0(Y) \to \R.$$
Upon $L^2_k$-completion of $\mathcal{C}(Y,\mathfrak{s})$ and $L^2_{k+1}$-completion of $\mathcal{G}_0(Y)$, since $\mathcal{G}_0 \curvearrowright \mathcal{C}_0$ freely,  $\mathcal{B}_0(Y,\mathfrak{s})$ is an infinite dimensional Hilbert manifold. We use $\mathcal{L} : \mathcal{B}_0 \to \R$ to define an $S^1$-equivariant monopole Floer homology. Note that there is a natural $S^1$-action on $\mathcal{B}_0$ that is left over from the full gauge group action $\mathcal{G}(Y) \curvearrowright \mathcal{C}(Y,\mathfrak{s})$. Furthermore, as in \cite{MR2388043}, the $L^2$-gradient of $\mathcal{L}$ is given by 
$$grad\, \mathcal{L} = \left(\dfrac{1}{2} \star_3 F_A + \rho^{-1}(\psi\psi^*)_0, D_A \psi\right).$$
Hence, the critical points of $\mathcal{L}$ are solutions of the three-dimensional Seiberg-Witten equations
\begin{equation*}
\begin{cases}
    D_A \psi = 0,\\
    \dfrac{1}{2} \rho(F_A) = (\psi \psi^*)_0.
\end{cases}
\end{equation*}
It is an elementary calculation to show that the critical points of $\mathcal{L}$ is $\mathcal{G}_0(Y)$-invariant. Elements of the set of critical points of $\mathcal{L}$ are divided into two types: irreducible if $\psi$ is not identically zero, and reducible if $\psi \equiv 0$. If we consider the based gauge symmetry, the based gauge equivalence classes of the irreducible solutions are free under the $S^1$-action, whereas the gauge equivalence classes of the reducible solutions are invariant under the $S^1$-action. In fact, since $b_1(Y) = 0$, there is a unique reducible solution $(A,0)$ up to based gauge transformation. We denote the based gauge equivalence class of $(A,0)$ by $\Theta$.

To do Morse theory for $\mathcal{L} : \mathcal{B}_0(Y,\mathfrak{s}) \to \R$, we need all of its critical points non-degenerate, i.e., the kernels of the corresponding Dirac operators are trivial. For this to happen, we have to perturb the Seiberg-Witten equations by a closed 2-form; and in turn, this means that we have to perturb the original Chern-Simons-Dirac functions $\mathcal{L}$. To this end, let $\eta \in i\Omega^2(Y)$ that is closed. We define the perturbed $\mathcal{L}$ to be
\begin{equation}\label{eq:4.1}
    \mathcal{L}_{\eta}(A,\psi) = -\dfrac{1}{8}\int_Y (A-A_0)\wedge (F_A + F_{A_0} - \eta) + \dfrac{1}{2}\int_Y \la D_A \psi, \psi\ra.
\end{equation}
Hence, the critical points of $\mathcal{L}_{\eta}$ solves the following perturbed Seiberg-Witten equations
\begin{equation}\label{eq:4.2}
    \begin{cases}
        D_A \psi = 0,\\
        \dfrac{1}{2} \rho(F_A - \eta) = (\psi \psi^*)_0.
    \end{cases}
\end{equation}

\begin{Prop}[\cite{MR2465077}, Proposition 8.1.1]\label{Prop4.2}
    For each metric $g$ on $Y$, there is generic $\eta \in i\Omega^2(Y)$ that is closed such that all critical points of $\mathcal{L}_{\eta}$ are non-degenerate.
\end{Prop}

We define $\mathcal{R}_0(Y, \mathfrak{s}, g, \eta)$ to be the set of based gauge equivalence classes of solutions of \eqref{eq:4.2}. For generic $\eta$, by Proposition \ref{Prop4.2}, $\mathcal{R}_0$ is a discrete finite set containing only non-degenerate critical points. Recall that $\mathcal{R}_0$ inherits the $S^1$-action from $\mathcal{B}_0$. All elements of $\mathcal{R}_0$ are free under this $S^1$-action, except for $\Theta$.

Before we state the next lemma, which is a result of an \textit{a priori} bound for solutions of \eqref{eq:4.2}, we set up the following convention. We denote by $[\cdot,\cdot]_0$ the based gauge equivalence classes of solutions of \eqref{eq:4.2} and $[\cdot, \cdot]$ the (full) gauge equivalence classes. Correspondingly, $\mathcal{R}(Y, \mathfrak{s}, g, \eta)$ is the set of gauge equivalence classes of solutions of the perturbed Seiberg-Witten equations. Similar to $\mathcal{R}_0$, $\mathcal{R}$ is a discrete finite set containing $S^1$-critical orbits of $\mathcal{L}_{\eta}$.

\begin{Lemma}\label{Lem4.3}
    Let $[A,\psi] \in \mathcal{R}(Y,\mathfrak{s},g,\eta)$ arbitrarily. Then we must have
    \begin{enumerate}
        \item $\norm{\psi}^2_{L^\infty} \leq -\inf_{Y}\dfrac{s}{2} + \norm{\eta}_{L^\infty}$, where $s$ denotes the scalar curvature of the metric $g$ on $Y$.
        \item There is a positive constant $c(g, \eta, A_0)$ such that $\norm{A- A_0}_{L^2} \leq c$.
    \end{enumerate}
\end{Lemma}

\begin{proof}
    Since $Y$ is a compact manifold, $|\psi|$ achieves a maximum at some point $y_0$. By the maximum principle, at $y_0$, we have 
    $$0\leq - \Delta |\psi|^2 = \la \nabla^*_A \nabla_A \psi, \psi \ra - |\nabla_A \psi|^2 \leq \la \nabla^*_A \nabla_A \psi, \psi\ra.$$
    On the other hand, since $D_A \psi = 0$, by the Weitzenb\"ock identity, we also have 
    $$0 = D^2_A \psi = \nabla^*_A \nabla_A \psi + \dfrac{s}{4}\psi + \dfrac{1}{2} F_A \cdot \psi.$$
    Thus, we obtain
    \begin{align*}
        0 &\leq -\frac{s}{4}|\psi|^2 - \left\la \dfrac{1}{2}F_A \cdot \psi, \psi \right\ra = -\dfrac{s}{4}|\psi|^2 - |(\psi\psi^*)_0|^2 - \dfrac{1}{2}\la \eta \cdot \psi ,\psi\ra\\
        & \leq -\dfrac{s}{4}|\psi|^2- \dfrac{1}{4}|\psi|^4 - \dfrac{1}{2}\norm{\eta}_{L^\infty} |\psi|^2 = -|\psi|^2\left( \dfrac{s}{4}+ \dfrac{1}{2}|\psi|^2 + \dfrac{1}{2}\norm{\eta}_{L^\infty}\right).
    \end{align*}
    From the above estimate, unless $\psi \equiv 0$, it must be the case that
    $$|\psi(y_0)|^2 \leq -\inf_{Y} \dfrac{s}{2} + \norm{\eta}_{L^{\infty}},$$
    which implies the bound in the first part as claimed.

    For the second part, firstly, note that the uniform $L^\infty$-bound on $\psi$ will give us a uniform $L^8$-bound for $\psi$. This bound depends on the metric $g$ and the perturbation $\eta$. Since $F_A - \eta = 2 (\psi \psi^*)_0$, we have
    $$\norm{F_A}^2_{L^2} \lesssim \norm{\psi}^8_{L^8} + \norm{\eta}^2_{L^\infty}.$$
    As a result, $F_A$ also has a uniform $L^2$-bound in terms of $g$ and $\eta$. Furthermore, by Uhlenbeck's gauge fixing result \cite{MR648355}, up to a gauge transformation, we may assume $A-A_0$ is co-closed. Thus, the elliptic estimate tells us that
    $$\norm{A-A_0}^2_{L^2_1} \lesssim \norm{d(A-A_0)}^2_{L^2} = \norm{F_A-F_{A_0}}^2_{L^2} \lesssim \norm{F_A}^2_{L^2} + \norm{F_{A_0}}^2_{L^2}.$$
    Because the $ L^2_1$-norm controls the $L^2$-norm, we obtain part two as desired.
\end{proof}

One of the crucial ingredients to construct a Floer theory we have to consider is the negative gradient flow line equations. The equations ask for the unknowns $(A(t), \psi(t))$ which can be thought of as a family of configuration parametrized by time $t \in \R$ that satisfy the following system
\begin{equation}\label{eq:4.3}
    \begin{cases}
        dA/dt = -\dfrac{1}{2}\star_3 (F_A -\eta)- \rho^{-1}(\psi\psi^*)_0\\
        d\psi/dt=-D_A \psi.
    \end{cases}
\end{equation}
It can be shown that \eqref{eq:4.3} can be re-written as a four-dimensional Seiberg-Witten equation on the cylinder $Y \times \R$ \cite{MR2388043}:
\begin{equation}\label{eq:4.4}
    \begin{cases}
        \dfrac{1}{2}\rho_{Y\times \R}(F_B - \pi^*_{Y}\eta)^+ = (\Phi \Phi^*)_0,\\
        D^+_B \Phi = 0.
    \end{cases}
\end{equation}
Here $\pi_Y$ denotes the projection of $Y\times \R$ onto $Y$, the $spin^c$ structure on $Y\times \R$ is induced by the $spin^c$ structure $\mathfrak{s}$ on $Y$, and the Clifford multiplication $\rho_{Y\times \R}$ is given by
$$\rho_{Y\times \R}\left(\dfrac{d}{dt}\right) = \begin{bmatrix} 0 & -1 \\ 1 & 0 \end{bmatrix}, \quad \quad \rho_{Y\times \R}(v) = \begin{bmatrix} 0 & -\rho(v)^* \\ \rho(v) & 0 \end{bmatrix}, \text{ where } v\in TY.$$
The Clifford module associated with the $spin^c$ structure decomposes into positive and negative parts $P^+_\mathfrak{s} \oplus P^-_\mathfrak{s}$. Let $\mathcal{C}(Y\times \R,\mathfrak{s}) = \mathcal{A}(det\,\mathfrak{s})\times \Gamma(W, P^+_{\mathfrak{s}})$, where $\mathcal{A}(det\, \mathfrak{s})$ is the space of $U(1)$-connections on the the determinant line bundle of the $spin^c$ structure on $Y\times \R$. $D^+_B$ is the restriction of the full Dirac operator to $P^+_\mathfrak{s}$.

In general, we consider $W$ a compact $4$-dimensional manifold with boundary $Y$ ($Y$ is not necessarily connected). Let $W^*$ be the manifold obtained from $W$ by attaching $Y\times [0,\infty)$ along the boundary $Y$. If $\mathfrak{s}$ is a $spin^c$ structure on $Y$, we consider a $spin^c$ structure $\mathfrak{s}_W$ on $W$ such that its restriction to the boundary is exactly $\mathfrak{s}$. Pick $\mu$ to be a purely imaginary $2$-form on $W$ such that $\mu|_{Y}$ equals to the pulled-back of $\eta \in i\Omega^2(Y)$ via the projection map $Y \times [0,\infty) \to Y$. We extend $\mathfrak{s}_W, \mu$ to $W^*$. Then the Seiberg-Witten equations on $W^*$ read as following
\begin{equation}\label{eq:4.5}
    \begin{cases}
        \dfrac{1}{2}\rho_{W^*}(F_B - \mu)^+ = (\Phi \Phi^*)_0,\\
        D^+_B \Phi = 0.
    \end{cases}
\end{equation}
The solutions of \eqref{eq:4.5} are gauge invariant. For each $t \in [0,\infty)$, $(B|_{Y \times t}, \Phi|_{Y \times t})$ is a configuration in $\mathcal{C}(Y, \mathfrak{s})$. We define $\mathscr{M}(W;\alpha)$ to be the moduli space of gauge equivalence classes of solutions $(B, \Phi)$ of \eqref{eq:4.5} such that $\lim_{t\to \infty} (B|_{Y\times t}, \Phi|_{Y\times t}) = (A, \psi)$, where $\alpha = [A, \psi]$ is a critical point of $\mathcal{L}_{\eta}$. Pick $B_0$ to be a referenced connection on $W$ (so, it's also a referenced connection on $W^*$) and denote $B_0|_{Y} := A_0$. Just as in Lemma \ref{Lem4.3}, we have an \textit{a priori} estimate for elements of $\mathscr{M}(W; \alpha)$.

\begin{Lemma}[\cite{MR2465077}, Proposition 3.4.1]\label{Lem4.4}
    Let $[B, \Phi] \in \mathscr{M}(W; \alpha)$ arbitrarily. Then we must have that either $\Phi \equiv 0$ or
    $$\norm{\Phi}^2_{L^\infty} \leq -\inf_{W}\dfrac{s}{2} + \norm{\mu}_{L^\infty}.$$
    Consequently, $F_B$ has a uniform $L^2$-bound in terms of the geometry of $W$ and $\mu$.
\end{Lemma}

We end this set-up subsection by discussing the notions of topological energy and analytical energy that will be relevant to us later. For that, we restrict our attention to compact $4$-manifold $W$ with boundary as above. Let $(B, \Phi)$ be a configuration on $W$ and $\mu$ be defined as before. By Lemma 4.5.1 in \cite{MR2388043} and the Weitzenb\"ock formula, we have 
\begin{align}\label{eq:4.6}
    \int_{W}|D^+_B \Phi|^2 = &\int_{W} |\nabla_B \Phi|^2 + \dfrac{1}{2}\int_{W} \la \Phi, \rho_W(F^+_A)\Phi\ra + \dfrac{1}{4}\int_W s|\Phi|^2\\
    &+ \int_{Y}\la \Phi|_Y, D_{B|_Y}\Phi|_Y\ra - \int_{Y} \dfrac{H}{2}\, |\Phi|_Y|^2,\nonumber
\end{align}
where $H$ is the mean curvature of the boundary of $W$. On the other hand,
\begin{align}\label{eq:4.7}
    &\int_W \left| \dfrac{1}{2}\rho_W(F^+_B - \mu^+) - (\Phi \Phi^*)_0\right|^2  = \\
    & = \dfrac{1}{4}\int_W |\rho_W (F^+_B - \mu^+)|^2 + \int_W |(\Phi\Phi^*)_0|^2 - \dfrac{1}{2}\int_W tr(\rho_W(F^+_B - \mu^+)(\Phi\Phi^*)_0) \nonumber \\
    & = \dfrac{1}{2}\int_{W} |F^+_B - \mu^+|^2 + \dfrac{1}{4}\int_W |\Phi|^4- \dfrac{1}{2} \int_W \la \Phi, \rho_W(F^+_B - \mu^+)\Phi\ra \nonumber\\
    & = \dfrac{1}{4}\int_W |F_B - \mu|^2 - \dfrac{1}{4}\int_W (F_B - \mu) \wedge (F_B - \mu) + \dfrac{1}{4}\int_W |\Phi|^4 \nonumber\\
    &- \dfrac{1}{2}\int_W \la \Phi, \rho_W(F^+_B)\Phi\ra + \dfrac{1}{2} \int_W \la \Phi, \rho_W(\mu^+)\Phi\ra \nonumber.
\end{align}
We define $\mathscr{F}_{W,\mu}(B,\Phi)$ to the Seiberg-Witten functional on $W$ so that
$$\norm{\mathscr{F}_{W,\mu}(B,\Phi)}^2 = \norm{D^+_B\Phi}^2_{L^2} + \norm{\dfrac{1}{2}\rho_W(F^+_B - \mu^+) - (\Phi\Phi^*)_0}^2_{L^2}.$$
Combine \eqref{eq:4.6} and \eqref{eq:4.7} with completion of square, we can rewrite
\begin{align}\label{eq:4.8}
    &\norm{\mathscr{F}_{W,\mu}(B,\Phi)}^2 = \underbrace{\dfrac{1}{4}\norm{F_B - \mu}^2_{L^2}+\norm{\nabla_B \Phi}^2_{L^2}+\dfrac{1}{4}\norm{|\Phi|^2+s/2}^2_{L^2}-\norm{s/4}^2_{L^2}}_{\text{$:= \mathscr{E}^{an}_{W,\mu}(B,\Phi)$}}  \\
    &+ \underbrace{\left(-\dfrac{1}{4}\int_W (F_B - \mu)\wedge (F_B - \mu) + \dfrac{1}{2}\int_W \la \Phi, \rho_W (\mu^+)\Phi\ra + \int_{Y}\la \Phi|_Y, D_{B|_Y}\Phi|_Y\ra - \int_{Y} \dfrac{H}{2}|\Phi|_Y|^2  \right)}_{\text{$:= -\mathscr{E}^{top}_{W,\mu}(B, \Phi)$}}.\nonumber
\end{align}

\begin{Def}\label{Def4.5}
    For a configuration $(B, \Phi)$, $\mathscr{E}^{an}_{W,\mu}(B,\Phi)$ and $\mathscr{E}^{top}_{W,\mu}(B,\Phi)$ defined above, respectively, are the \textit{analytical energy} and \textit{topological energy}.
\end{Def}

\begin{Rem}\label{Rem4.6}
Immediately from the definition, we see that $\mathscr{F}_{W, \mu}(B, \Phi) =0$ if and only if the topological energy equals the analytical energy. This observation will used heavily when we compare the values of the Chern-Simons-Dirac functional evaluated at the critical points on the two ends of a cobordism between two $3$-manifolds. 
\end{Rem}

\subsection{$S^1$-equivariant monopole Floer homology}\label{Sub4.2} In this subsection, we construct the $\mathcal{S}$-complex of monopole Floer homology via $\mathcal{L}_{\eta}: \mathcal{B}_0(Y,\mathfrak{s}) \to \R$ by following the finite-dimensional model in Section \ref{Sec3}. Most of what is presented here are based on the framework laid out by Fr\o yshov in \cite{MR2738582}. The complex is parametrized by elements of $\mathfrak{m}(Y) \in \Q/\mathbf{Z}$, which are called "chambers". $\mathfrak{m}(Y)$ is a classical invariant that is defined as $\mathfrak{m}(Y) \equiv (\sigma(W) - c_1(det\,\mathfrak{s}_W)^2)/8 \mod \mathbf{Z}$, where $\sigma(W)$ denotes the signature of a four-manifold $W$ with boundary $Y$ and $\mathfrak{s}_W$ is a $spin^c$ structure. Alternatively, we can view $\mathfrak{m}$ as a set of rational numbers of the form $m + \mathbf{Z}$ with $m \in \Q$. 

Let $\mathfrak{s}$ be a $spin^c$ structure on $Y$ and $(g,\eta)$ as in Proposition \ref{Prop4.2}, we consider $m_\mathfrak{s}$ which is a representative of a chamber in $\mathfrak{m}(Y)$. Suppose $W^*$ is a 4-manifold with an end that looks like $Y \times [0,\infty)$ and $B$ is a $spin^c$ connection on $W^*$ whose restriction to $Y$ is a $spin^c$ connection $A - \eta$. Define an associated index $I(g,\eta) \in \mathfrak{m}(Y)$ to be
$$I(g,\eta) = \ind_{\mathbf{C}}(D^+_B) - \dfrac{1}{8}(c_1(det\,\mathfrak{s}_{W^*})^2- \sigma(W^*)).$$
Theorem 9 of Section 3 in \cite{MR2738582} tells us that there always exists a pair $(g,\eta)$ as in Proposition \ref{Prop4.2} such that $I(g,\eta) = m_{\mathfrak{s}}$. With this auxiliary data $(g,\eta)$, we build an $\mathcal{S}$-complex of monopoles $\widetilde{CM}(Y, m_\mathfrak{s}; \mathbf{F}):= \widetilde{CM}(Y, \mathfrak{s}; \mathbf{F})$, where recall that $\mathbf{F}$ denotes the field of characteristic $2$.

Consider \eqref{eq:4.4} the equation of negative gradient flow lines of $\mathcal{L}_{g,\eta} : \mathcal{C}(Y, \mathfrak{s})_{L^2_k} \to \R$, which is the Seiberg-Witten equations of $Y \times \R$. Let $\mathcal{F}$ be the Seiberg-Witten map associated with \eqref{eq:4.4},
$$\mathcal{F} (B, \Phi) = \left( D^+_B \Phi, \dfrac{1}{2}\rho_{Y\times \R}(F_B - \pi^*_Y \eta)^+ - (\Phi\Phi^*)_0\right).$$
Note that $\mathcal{F}$ is $\mathcal{G}_0(Y\times \R)$-equivariant, where $\mathcal{G}_0(Y\times \R)_{L^2_{k+1}}$ is a based gauge group of the full gauge group $\mathcal{G}(Y\times \R)_{L^2_{k+1}} = L^2_{k+1}-maps(Y\times \R \to U(1))$. Let $(A_\pm, \psi_\pm)$ be solutions of the perturbed three-dimensional Seiberg-Witten equations on $Y$. Consider $\mathcal{M}([A_-,\psi_-]_0, [A_+, \psi_+]_0)$ to be
\[
\left\{(B, \Phi) \in \mathcal{C}(Y\times \R, \pi^*_Y \mathfrak{s})_{L^2_k} : \begin{array}{cc}
    \mathcal{F}(B,\Phi) = 0 &  \\
    \lim_{t\to -\infty}(B,\Phi)|_{Y\times t} = (A_-,\psi_-) & \\
     \lim_{t\to \infty} (B,\Phi)|_{Y\times t}= (A_+, \psi_+) & 
\end{array}\right\}/\mathcal{G}_0(Y\times \R)_{L^2_{k+1}}.
\]
Since $\mathcal{F}$ is $\mathcal{G}_0(Y\times \R)$-equivariant, and
$$\lim_{t\to -\infty} [B,\Phi]_0|_{Y\times t} = [A_-, \psi_-]_0, \quad \lim_{t\to \infty}[B,\Phi]_0|_{Y\times t} = [A_+, \psi_+]_0,$$
$\mathcal{M}([A_-,\psi_-]_0, [A_+, \psi_+]_0)$ is exactly the based moduli space of negative gradient flow lines of $\mathcal{L}_{g,\eta} : \mathcal{B}_0(Y,\mathfrak{s}) \to \R$ that are asymptotic to $[A_\pm, \psi_\pm]_0 \in \mathcal{R}_0(Y,\mathfrak{s},g,\eta)$ at the ends.

Let $[B,\Phi]_0 \in \mathcal{M}([A_-,\psi_-]_0, [A_+, \psi_+]_0)$. For any $e^{i\theta} \in S^1$, $e^{i\theta}\cdot [B,\Phi]_0$ still satisfies $\mathcal{F} = 0$. Furthermore, it is not hard to see that $e^{i\theta}\cdot [B,\Phi]_0|_{Y\times t}$ is aymptotic to $e^{i\theta}\cdot [A_\pm, \psi_\pm]_0$ as $t \to \pm \infty$. This prompts us to consider the moduli space $\mathcal{M}([A_-,\psi_-],[A_+,\psi_+])$ of based gauge equivalence classes of solutions of $\mathcal{F} = 0$ that are asymptotic to the $S^1$-orbit of $[A_\pm, \psi_\pm]$. Note that, unlike the finite-dimensional prototype, even in the most ideal conditions, $\mathcal{M}([A_-,\psi_-],[A_+,\psi_+])$ is not a finite-dimensional manifold. However, what still holds is that if one of $[A_\pm, \psi_\pm]_0$ is an irreducible monopole, then there is a free $S^1$-action on $\mathcal{M}([A_-,\psi_-],[A_+,\psi_+])$.

\begin{Lemma}\label{Lem4.7}
    If either $[A_\pm, \psi_\pm]_0 \in \mathcal{R}_0(Y,\mathfrak{s},g,\eta)$ is irreducible, then 
    $$\mathscr{M}([A_-,\psi_-],[A_+,\psi_+]) := \mathcal{M}([A_-,\psi_-],[A_+,\psi_+])/S^1$$
    is a finite-dimensional manifold.
\end{Lemma}

\begin{proof}
    Since $\mathcal{F} : \mathcal{B}_0(Y\times \R) = \mathcal{C}(Y\times \R)_{L^2_k}/\mathcal{G}_0(Y\times \R)_{L^2_{k+1}} \to \R$ is $S^1$-equivariant, this implies that $\mathcal{M}([A_-,\psi_-],[A_+,\psi_+])/S^1$ is
    $$\left\{ (B,\Phi) : \begin{array}{cc}
        \mathcal{F}(B,\Phi)= 0 &  \\
        \lim_{t\to \pm \infty} (B,\Phi)|_{Y\times t} = (A_\pm, \psi_\pm) & 
    \end{array}\right\}/\mathcal{G}(Y\times \R),$$
    which is exactly the moduli space of gauge equivalence classes of negative gradient flow lines between critical points of $\mathcal{L}_{g,\eta}: \mathcal{C}(Y,\mathfrak{s}) \to \R$. It is a finite-dimensional manifold by the transversality and regularity results in Chapter 8 of \cite{MR2465077}. 
\end{proof}

Let $\alpha_\pm := [A_\pm, \psi_\pm]$, which we recall that it can be viewed as the $S^1$-orbit of $[A_\pm, \psi_\pm]_0$. Then there are natural $S^1$-equivariant maps $e_{\alpha_\pm}: \mathcal{M}(\alpha_-, \alpha_+) \to \alpha_{\pm}$ by following the based gauge equivalence classes of negative gradient flow lines to its asymptotic limit belonging the $S^1$-orbit $\alpha_\pm$. As a result, we get an $S^1$-equivariant map $h_{\alpha_+ \alpha_-}: \mathcal{M}(\alpha_-, \alpha_+) \to \alpha_+ \times \alpha_-$. After picking a representative of each critical $S^1$-orbit and quotient out everything by $S^1$, assuming that either $\alpha_\pm$ is a free critical $S^1$-orbit, by Lemma \ref{Lem4.7}, we obtain a smooth map (cf. the construction in Subsection \ref{Sub3.1})
$$\widetilde{h}_{\alpha_- \alpha_+} : \mathscr{M}(\alpha_-, \alpha_+) \to S^1.$$
We use $\widetilde{h}_{\alpha_- \alpha_+}$ to build the $\mathcal{S}$-complex of monopoles. Before we proceed to the construction, we need to discuss the dimension of the moduli space of gauge equivalence classes of flow lines defined above. The dimension gives us a notion of an index, just as in the finite-dimensional prototype model.

We follow the discussion in Section 4 of \cite{MR2738582} (cf. \cite{MR1895135} for $b_1(Y) >0$ case). Let $X$ be a $4$-manifold with tubular ends $Y_j \times [0,\infty)$, for $j=1,\cdots,n$, where $Y_j$ is a rational homology sphere. For each $j$, we choose a chamber $m_j$ belonging to $\mathfrak{m}(Y_j)$ associated with a $spin^c$ structure $\mathfrak{s}_j$ on $Y_j$ and let $(g_j, \eta_j)$ be an auxiliary data  such that $I(g_j, \eta_j) = m_j$ and all critical points of $\mathcal{L}_{\eta_j}$ are non-degenerate. For each $\alpha_j \in \mathcal{R}(Y_j, \mathfrak{s}_j, g_j, \eta_j)$, we define $\mathscr{M}(X; \alpha_1, \cdots, \alpha_n)$ to be the moduli space of gauge equivalence classes of solutions of the Seiberg-Witten equations on $X$ that are asymptotic to $\alpha_j$ at the ends. The expected dimension $\mathscr{M}(X; \alpha_1, \cdots, \alpha_n)$ is the index of a certain Fredholm operator (cf. Section 3.4 in \cite{MR2465077}). In particular, if we all have reducible limits $\{\Theta_j\}_{j=1}^n$ at the ends, then
\begin{equation}\label{eq:4.9}
\dim \mathscr{M}(X; \Theta_1, \cdots, \Theta_j) = -b_0(X) + b_1(X) - b^+(X) + 2\,\ind_{\mathbf{C}}(D^+_B),
\end{equation}
where $B$ is a $spin^c$ connection on $X$ defined as before. Consequently, by \eqref{eq:4.9}, in the case where $X= Y\times \R$, we obtain
\begin{equation}\label{eq:4.10}
    \dim \mathscr{M}(\Theta, \Theta) = -1 + 2\,\ind_{\mathbf{C}}(D^+_B).
\end{equation}
Also in this case $X = Y\times \R$, if $\alpha, \gamma, \beta \in \mathcal{R}(Y, \mathfrak{s}, g, \eta)$, then we have the following additional rule (Corollary C.0.1 in \cite{MR2465077})
\begin{equation}\label{eq:4.11}
\dim \mathscr{M}(\alpha, \beta) = \dim \mathscr{M}(\alpha, \gamma) + \dim \mathscr{M}(\gamma, \beta) + n_\gamma,
\end{equation}
where $n_\gamma = 1$ if $\gamma$ is irreducible, and $n_\gamma = 0$ if $\gamma = \Theta$. For $\alpha \in \mathcal{R}(Y, \mathfrak{s}, g, \eta)$, we are now ready to define $\ind(\alpha) \in 2\mathfrak{m}(Y) + \mathbf{Z}$. Suppose $X$ is any connected $4$-manifold with a tubular end $Y \times [0,\infty)$. Then $\ind(\alpha)$ is given by
\begin{equation}\label{eq:4.12}
    \dim \mathscr{M}(X; \alpha) = \ind(\alpha) + \dfrac{1}{4}(c_1(det\,\mathfrak{s}_{X})^2-\sigma(X))-n_{\alpha}+b_1(X) - b^+(X).
\end{equation}
\eqref{eq:4.10} and \eqref{eq:4.12} combined to give us the index of $\Theta$ in the case $X = Y\times \R$
\begin{equation}\label{eq:4.13}
    \ind(\Theta) = 2\left(\ind_{\mathbf{C}}(D^+_B)-\dfrac{1}{8}(c_1(det \mathfrak{s}_X)^2-\sigma(X))\right) = 2 I(g,\eta) = 2m_\mathfrak{s}.
\end{equation}
Furthermore, from \eqref{eq:4.12} and \eqref{eq:4.13}, note that (still assuming $X = Y\times \R$), for $\alpha, \beta \in \mathcal{R}(Y,\mathfrak{s},g,\eta)$ that are irreducible monopoles
\begin{equation}\label{eq:4.14}
\dim \mathscr{M}(\alpha, \Theta) = \ind(\alpha) -2m_\mathfrak{s}, \quad \dim \mathscr{M}(\Theta, \beta) = 2m_{\mathfrak{s}}-1 - \ind(\beta).
\end{equation}
Consequently, by \eqref{eq:4.11}, we must have
\begin{equation}\label{eq:4.15}
\dim \mathscr{M}(\alpha , \beta) = \ind(\alpha) - \ind(\beta).
\end{equation}

Let's get back to the map $\widetilde{h}_{\alpha_- \alpha_+} : \mathscr{M}(\alpha_-,\alpha_+) \to S^1$ constructed previously. There is a free $\R$-action on the moduli space of negative gradient flow lines between $\alpha_\pm$ given by translation. By considering the quotient space $\mathscr{M}(\alpha_-,\alpha_+)/\R := \Breve{\mathscr{M}}(\alpha_-, \alpha_+)$, we obtain the moduli space of unparametrized negative gradient flow lines interpolating between critical points. Even though $\Breve{\mathscr{M}}(\alpha_-, \alpha_+)$ is still a manifold, it is not compact. If we include the broken flow lines to obtain $\Breve{\mathscr{M}}^+(\alpha_-,\alpha_+)$, by gluing theory (cf. \cite{MR2465077}), $\Breve{\mathscr{M}}^+(\alpha_-, \alpha_+)$ is indeed compact. Furthermore, the map $\widetilde{h}_{\alpha_- \alpha_+}$ descends to a smooth map (still denoted by the same name)
\begin{equation}\label{eq:4.16}
\widetilde{h}_{\alpha_- \alpha_ +} : \Breve{\mathscr{M}}^+(\alpha_-, \alpha_+) \to S^1.
\end{equation}
Using the map \eqref{eq:4.16} and the construction in Subsection \ref{Sub3.1} combined with the relation between the dimension of various relevant moduli spaces and indices derived above (cf. \eqref{eq:4.14}, \eqref{eq:4.15}), we define the $\mathcal{S}$-complex of monopoles as follows. Let $CM(Y, m_{\mathfrak{s}}, g, \eta; \mathbf{F})$ be the finitely generated abelian group over $\mathbf{F}$ given by
$$CM(Y, m_{\mathfrak{s}}, g, \eta; \mathbf{F})_\ast = \mathbf{F}\cdot \{\alpha \in \mathcal{R}^*(Y, \mathfrak{s}, g, \eta) : \ind (\alpha) = \ast\},$$
where $\mathcal{R}^*(Y, \mathfrak{s}, g, \eta)$ denote the finite set containing only gauge equivalence classes of irreducible monopoles of the perturbed 3-dimensional Seiberg-Witten equations \eqref{eq:4.2}. For $\alpha_\pm \in \mathcal{R}^*(Y, \mathfrak{s}, g, \eta)$ such that $\ind(\alpha_-) = \ind(\alpha_+) + 1$, then by $\eqref{eq:4.15}$, $\dim \Breve{\mathscr{M}}^+(\alpha_- ,\alpha_+) = 0$. As a result, we obtain a map
\begin{equation}\label{eq:4.17}
    d: CM(Y, m_{\mathfrak{s}}, g, \eta; \mathbf{F})_\ast \to CM(Y, m_{\mathfrak{s}}, g, \eta; \mathbf{F})_{\ast - 1}, \quad \la d(\alpha_-), \alpha_+\ra = \#_{\mathbf{Z}_2}\Breve{\mathscr{M}}^+(\alpha_-,\alpha_+). 
\end{equation}
If $\alpha_\pm \in \mathcal{R}^*$ such that $\ind(\alpha_-) = \ind(\alpha_+) + 2$, by \eqref{eq:4.15}, $\dim \Breve{\mathscr{M}}^+(\alpha_-,\alpha_+) = 1$. In this case, using \eqref{eq:4.16}, we have a map
\begin{equation}\label{eq:4.18}
   u :  CM(Y, m_{\mathfrak{s}}, g, \eta; \mathbf{F})_\ast \to CM(Y, m_{\mathfrak{s}}, g, \eta; \mathbf{F})_{\ast - 2},
\end{equation}
$$\la u(\alpha_-), \alpha_+ \ra = \la \widetilde{h}^*_{\alpha_- \alpha_+}(1), [\Breve{\mathscr{M}}^+(\alpha_-,\alpha+)]\ra \mod 2.$$
Here $1$ denotes the generator of $H^1(S^1;\mathbf{F})$. Finally, we take into account the gauge equivalence class $\Theta$ of reducible monopoles. When $\ind(\alpha) = 2m_\mathfrak{s} + 1$, by \eqref{eq:4.14}, $\dim \Breve{\mathscr{M}}^+(\alpha, \Theta) = 0$, we obtain
\begin{equation}\label{eq:4.19}
    \delta_1 : CM(Y, m_{\mathfrak{s}}, g, \eta; \mathbf{F})_{2m_\mathfrak{s} + 1} \to \mathbf{F}, \quad \delta_1(\alpha) = \sum_{\alpha} \#_{\mathbf{Z}_2} \Breve{\mathscr{M}}^+(\alpha, \Theta) \mod 2.
\end{equation}
Similarly, when $\ind(\beta) = 2m_\mathfrak{s}-2$, we have
\begin{equation}\label{eq:4.20}
    \delta_2 : \mathbf{F} \to CM(Y, \mathfrak{s}, g, \eta; \mathbf{F})_{2m_{\mathfrak{s}}-2}, \quad \delta_2(1) = \sum_{\beta} \#_{\mathbf{Z}_2} \Breve{\mathscr{M}}^+(\Theta, \beta) \beta.
\end{equation}

So far, the maps defined by \eqref{eq:4.17}, \eqref{eq:4.18}, \eqref{eq:4.1}, \eqref{eq:4.20} rely on the definition of the map $\widetilde{h}_{\alpha_- \alpha_+}$ in \eqref{eq:4.16}, which is dependent of the choice of the representative element of the $S^1$-critical orbit of $\mathcal{L}_\eta: \mathcal{B}_0(Y,\mathfrak{s}) \to \R$. The discussion in Subsection \ref{Sub3.3} carries over directly in this setting to give us the well-definedness of the construction.  

\begin{Prop}[cf. Proposition 4 of Section 7, Proposition 8 of Section 8 in \cite{MR2738582}, Proposition \ref{Prop3.5}, Proposition \ref{Prop3.6}]\label{Prop4.8}
    Let $Y$ be a rational homology three-sphere equipped with a $spin^c$ structure $\mathfrak{s}$ with an associated $m_\mathfrak{s} \in \mathfrak{m}(Y)$. Pick $(g, \eta)$ to be a pair of auxiliary data on $Y$ such that $I(g,\eta) = m_\mathfrak{s}$. The maps defined by \eqref{eq:4.17}, \eqref{eq:4.18}, \eqref{eq:4.1}, \eqref{eq:4.20} satisfy the following relations
    $$d^2 = 0, \quad \delta_1 d = 0, \quad d \delta_2 = 0, \quad du + ud + \delta_2 \delta_1 = 0.$$
\end{Prop}

We define $\widetilde{CM}(Y, m_{\mathfrak{s}}, g, \eta; \mathbf{F})$ to be
 $$\widetilde{CM}(Y, m_{\mathfrak{s}}, g, \eta; \mathbf{F})_\ast = CM_\ast \oplus CM_{\ast-1} \oplus \mathbf{F},$$
and the map $\widetilde{d}: \widetilde{CM}_\ast \to \widetilde{CM}_{\ast-1}$ to be
$$\widetilde{d}=\begin{bmatrix} d & 0 & 0 \\ u & \delta_2 & d \\ \delta_1 & 0 & 0 \end{bmatrix}.$$
From Proposition \ref{Prop4.8}, we immediately see that 

\begin{Th}\label{Th4.9}
    Let $Y$ be a rational homology three-sphere with the same hypothesis as Proposition \ref{Prop4.8}. Then $(\widetilde{CM}(Y, m_\mathfrak{s}, g, \eta), \widetilde{d})$ is an $\mathcal{S}$-complex. 
\end{Th}

So far, for each each $spin^c$ structure $\mathfrak{s}$ on a rational homology sphere $Y$ and an associated $m_\mathfrak{s} \in \mathfrak{m}(Y)$, given an auxiliary data $(g, \eta)$, we constructed an $\mathcal{S}$-complex $\widetilde{CM}(Y, \mathfrak{s}, g, \eta)$. The following theorem tells us what happens as we vary the auxiliary data. Let $(g_1, \eta_1)$ and $(g_2, \eta_2)$ be two different pairs of metric and perturbation that give us two different chambers $m^1_\mathfrak{s}, m^2_\mathfrak{s}$. 

\begin{Th}[cf. Proposition 5 of Section 7 in \cite{MR2738582}]\label{Th4.10}
    For $q \leq 2m^1_\mathfrak{s} - 1$ or $q\geq 2m^2_\mathfrak{s}$, then the continuation map that gives rise to 
    $$\widetilde{CM}(Y, m^1_\mathfrak{s}, g_1, \eta_1; \mathbf{F})_q \to \widetilde{CM}(Y, m^2_\mathfrak{s}, g_2, \eta_2; \mathbf{F})_q$$
    is an $\mathcal{S}$-chain quasi-isomorphism in the sense of Definition \ref{Def3.17}. As a result, 
    $$\widetilde{HM}_q(Y, m^1_\mathfrak{s}; \mathbf{F}) := H(\widetilde{CM}(Y, m^1_\mathfrak{s}, g_1, \eta_1; \mathbf{F})_q) \cong H(\widetilde{CM}(Y, m^2_\mathfrak{s}, g_2, \eta_2; \mathbf{F})_q):= \widetilde{HM}_q(Y, m^2_\mathfrak{s}; \mathbf{F}).$$
\end{Th}

The proof of Theorem \ref{Th4.10} is essentially the same as the proof of Proposition 4 of Section 7 in \cite{MR2738582}. An immediate corollary of Theorem \ref{Th4.10} is as follows.

\begin{Cor}\label{Cor4.11}
    Suppose $m_\mathfrak{s} \geq 0$. For each $q \leq -1$ or $q \geq 2m_\mathfrak{s}$, we have $\widetilde{HM}_q(Y, m_\mathfrak{s}; \mathbf{F})$ is an invariant of $Y$. Thus, for such $q$, $HM_q(Y, m_\mathfrak{s}; \mathbf{F}):= H(CM(Y, m_\mathfrak{s}; \mathbf{F})_q)$ is also an invariant of $Y$. 
\end{Cor}

With Corollary \ref{Cor4.11} in mind, for allowable range of degree $q$, we now simply write
\begin{equation}\label{eq:4.21}
    HM_q(Y, \mathfrak{s}; \mathbf{F}) := HM_q(Y, m_\mathfrak{s}; \mathbf{F}), \quad \widetilde{HM}_q(Y, m_\mathfrak{s}; \mathbf{F}) := \widetilde{HM}_q(Y, \mathfrak{s}; \mathbf{F}).
\end{equation}
to highlight the fact that these invariants of the rational homology sphere $Y$ depend on an \textit{a priori} choice of the $spin^c$ structure. The next section investigates the spectral invariant associated with $HM_q, \widetilde{HM}_q$.

\subsection{Cobordism}\label{Sub4.3} Let $W$ be any smooth compact $spin^c$ $4$-manifold with boundary $\partial W = -Y_1 \cup Y_2$, where $Y_i$ is a rational homology sphere for $i=1,2$. We view $W: Y_1 \to Y_2$ as a cobordism from $Y_1$ to $Y_2$. Denote $W^*$ by the elongated version of $W$ where we attach a cylinder to its boundary, i.e.,
$$W^* =  Y_1 \times [0,-\infty) \cup_{Y_1} W \cup_{Y_2} Y_2 \times [0,\infty).$$
As in Subection \ref{Sub4.1}, we pick a $spin^c$ structure $\mathfrak{s}_W$, which can be extended to $W^*$. Let $\mathfrak{s}_i$ be the $spin^c$ structure on $Y_i$ given by $\mathfrak{s}_W|_{Y_i}$. Suppose $g$ is a metric on $W^*$ such that its restriction to one of the cylinder ends is $1 \times g_i$, where $g_i$ is a metric on $Y_i$. Let $\eta_i$ be a generic closed purely imaginary $2$-form on $Y_i$ and $\mu$ is a closed purely imaginary $2$-form on $W$ such that
\begin{itemize}
    \item The pulled-back of $\eta_i$ via $Y_i \times \R^{\pm}$ for $i=1,2$, respectively, agrees with the restriction of $\mu$ on each of the cylindrical ends.
    \item $\mathcal{L}_{\eta_i}$ has non-degenrate critical points and $I(g_i, \eta_i) = m_{\mathfrak{s}_i}$, where $m_{\mathfrak{s}_i}$ is a chamber of $\mathfrak{m}(Y_i)$.
\end{itemize}
Consider the Seiberg-Witten equations on $W^*$ as in \eqref{eq:4.5}. Let $\alpha_i \in \mathcal{R}(Y_i,\mathfrak{s}_i, g_i, \eta_i)$. Then the moduli space $\mathscr{M}(W;\alpha_1, \alpha_2)$ of gauge equivalence classes of solutions $(B, \Phi)$ of \eqref{eq:4.5} such that its limit at each end is $\alpha_1, \alpha_2$, respectively, has dimension given by
$$\dim \mathscr{M}(W; \alpha_1, \alpha_2) = \ind(\alpha_1) - \ind(\alpha_2) +n_{\alpha_1} - d(W),$$
where $d(W) = (c_1(det\,\mathfrak{s}_W)^2-\sigma(W))/4 + b_1(W) - b^+(W)$. Denote $k = m_{\mathfrak{s}_1} - m_{\mathfrak{s}_2} - d(W)/2 \in (1/2)\mathbf{Z}$. Let $\Theta_i$ be the gauge equivalence classes of reducible solutions of the corresponding perturbed Seiberg-Witten equations on $Y_i$. Then $\dim \mathscr{M}(W; \Theta_1, \Theta_2) = 2k +1$.

Recall that if $\alpha_i = [A_i, \psi_i]$, then $\alpha_i$ can be thought of as the $S^1$-orbit of $[A_i, \psi_i]_0$. Similar to the proof of Lemma \ref{Lem4.7}, we can show that if either $\{\alpha_i\}_{i=1,2}$ is irreducible, then $\mathcal{M}(W; \alpha_1, \alpha_2)/S^1 = \mathscr{M}(W; \alpha_1, \alpha_2)$, where $\mathcal{M}(W;\alpha_1,\alpha_2)$ is the moduli space of based gauge equivalence classes of solutions of \eqref{eq:4.5} whose asymptotic limits at the ends belongs to the $S^1$-orbits of $[A_i, \psi_i]_0$. Suppose $\alpha_i$ is irreducible. Following what had been done in Subsection \ref{Sub4.2} and Subsection \ref{Sub3.2}, similarly, we obtain the following maps
$$h_{W;\alpha_1\alpha_2} : \mathscr{M}(W; \alpha_1, \alpha_2) \to S^1,$$
$$h_{W; \Theta_1 \alpha_2} : \mathscr{M}(W; \Theta_1, \alpha_2) \to \Theta_1, \quad h_{W; \alpha_1 \Theta_2} : \mathscr{M}(W; \alpha_1, \Theta_2) \to \Theta_2.$$

Suppose $b^+(W) \leq 1$ and $k \leq -1$, then for $q \leq 2m_{\mathfrak{s}_1}-1$ or $q \geq 2m_{\mathfrak{s}_2}$, due to no factorization through reducible monopoles that may appear in certain chain limits as we glue the moduli spaces along certain critical points (cf. Section 8 in \cite{MR2738582}), we can define the following $\mathbf{F}$-linear maps
\begin{align}
    C(W) : CM(Y_1, \mathfrak{s}_1;\mathbf{F})_q \to CM(Y_2, \mathfrak{s}_2;\mathbf{F})_{q-d(W)}, \label{eq:4.22}
\end{align}
$$ \la C(W)(\alpha_1), \alpha_2\ra = \#_{\mathbf{Z}_2} \mathscr{M}(W; \alpha_1, \alpha_2).$$
\begin{align}
    u(W) : CM(Y_1, \mathfrak{s}_1;\mathbf{F})_q \to CM(Y_2, \mathfrak{s}_2; \mathbf{F})_{q-d(W)-1}, \label{eq:4.23}
\end{align}
$$\quad \la u(W)(\alpha_1),\alpha_2 \ra = \la h^{*}_{W;\alpha_1 \alpha_2}(1), [\mathscr{M}(W; \alpha_1,\alpha_2]\ra \mod 2,$$
\begin{align}
    \delta_1(W): CM(Y_1, \mathfrak{s}_1;\mathbf{F}) \to \mathbf{F}, \quad \delta_2(W) : \mathbf{F} \to CM(Y_2, \mathfrak{s}_2;\mathbf{F}),\label{eq:4.24}
\end{align}
where $\delta_1(W)$ is given by the modular $2$ count of the moduli space $\mathscr{M}(W; \alpha_1, \Theta_2)$, and $\la \delta_2(W)(1), \alpha_2 \ra$ is equal to the $\mod 2$ count of $\mathscr{M}(W; \Theta_1, \alpha_2)$, whenever the respective dimension of the moduli spaces is zero.

\begin{Rem}
    In general, when $b^+(W) >1$ or $k > -1$, the maps \eqref{eq:4.22}, \eqref{eq:4.23}, \eqref{eq:4.24} can be defined for any degree. However, we choose to focus entirely on the former restrictive case of degree in light of Corollary \ref{Cor4.11}, where for the allowable range of degree, the $\mathcal{S}$-chain homotopy type of the $\mathcal{S}$-complex of monopoles is an invariant of the rational homology three-sphere $Y$.
\end{Rem}

\eqref{eq:4.22}, \eqref{eq:4.23}, \eqref{eq:4.24} can be repackaged into a single linear map between $\mathcal{S}$-complexes
$$\widetilde{CM}(W): \widetilde{CM}(Y_1,\mathfrak{s}_1;\mathbf{F})_{\ast} \to \widetilde{CM}(Y_2, \mathfrak{s}_2;\mathbf{F})_{\ast - d(W)},$$
$$\widetilde{CM}(W) = \begin{bmatrix}
    CM(W) & 0 & 0\\
    u(W) & CM(W) & \delta_2(W) \\
    \delta_1(W) & 0 & 1
\end{bmatrix}.$$

\begin{Th}\label{Th4.13}
    Suppose $W: Y_1 \to Y_2$ is a cobordism between two rational homology spheres. Suppose $b^+(W) \leq 1$. Let $\mathfrak{s}_W$ be a $spin^c$ structure on $W$ and $\mathfrak{s}_i$ be its restriction to each of the boundary components of $W$. Let $m_{\mathfrak{s}_i} \in \mathfrak{m}(Y_i)$ and denote $k = m_{\mathfrak{s}_1} - m_{\mathfrak{s}_2} - d(W)/2$, where  $d(W) = (c_1(det\,\mathfrak{s}_W)^2-\sigma(W))/4 + b_1(W) - b^+(W)$. If $k \leq -1$, then for $q \leq 2m_{\mathfrak{s}_1}-1$ or $q\geq 2m_{\mathfrak{s}_2}$, then 
    $$\widetilde{CM}(W): \widetilde{CM}(Y_1, m_{\mathfrak{s}_1};\mathbf{F})_q \to \widetilde{CM}(Y_2, m_{\mathfrak{s}_2}; \mathbf{F})_{q - d(W)}$$
    is an $\mathcal{S}$-chain map in the sense of Remark \ref{Rem3.16}. 
\end{Th}

The proof of the above theorem is essentially the same as the proof of Theorem \ref{Th3.8}. Similar proofs for analogous statements in the context of instanton Floer theory have been given in, e.g., \cite{MR1883043, MR1910040, daemi2020equivariantaspectssingularinstanton, daemi2022instantonsrationalhomologyspheres}.

\section{Spectral invariants and monopole Floer homology}\label{Sec5}
The majority of the discussion in this section applies to both $HM, \widetilde{HM}$. We use the notation $CM^\circ, HM^\circ$, $\circ \in \{\emptyset, \sim\}$, to refer to either version of the monopole Floer theory discussed in Section \ref{Sec4}. Specifically, if $\circ = \emptyset$, then $CM^\circ, HM^\circ$ is $CM$ and $HM$, respectively. If $\circ = \sim$, then $CM^\circ, HM^\circ$ refer to $\widetilde{CM}, \widetilde{HM}$.

\subsection{Definition and properties}\label{Sub5.1}
Let $Y$ be a closed, oriented, smooth, and connected rational homology three-sphere equipped with a metric $g$. Let $\mathfrak{s}$ be a $spin^c$ structure on $Y$, and associated with it, consider a chamber $m_\mathfrak{s} \in \mathfrak{m}(Y)$ such that for a generic $\eta \in i\Omega^2(Y)$ that is closed, we have $I(g, \eta) = m_{\mathfrak{s}}$. The Chern-Simons-Dirac functional $\mathcal{L}_\eta$ induces an $\R$-filtration $\ell_{\eta}: CM^\circ \to \R \cup \{-\infty\}$ defined by
$$\ell_\eta\left(\sum a_i \alpha_i\right) = \max \{\mathcal{L}_{\eta}(\alpha_i) : a_i \neq 0\}, \quad \alpha_i \in \mathcal{R}^*(Y,\mathfrak{s},g,\eta), \quad \quad \text{if } \circ = \emptyset,$$
$$\ell_{\eta}(\alpha \oplus \beta \oplus a) = \max \{\ell_{\eta}(\alpha), \ell_{\eta}(\beta), \ell_{\eta}(\Theta)\}, \quad \alpha, \beta \in CM, a \in \mathbf{F}\setminus \{0\}, \quad \quad \text{if } \circ = \sim.$$
In the modern language of persistent homology (cf. Definition 2.2, \cite{MR3590354}), $(CM^\circ, \ell_\eta)$ is a nonarchimedean normed vector space.

For each $t \in \R$, we define
$$CM^{\circ t}(Y, m_\mathfrak{s}, g, \eta; \mathbf{F}) = \{\mathbf{x} \in CM^{\circ}(Y, m_\mathfrak{s}, g, \eta; \mathbf{F}) : \ell_\eta(\mathbf{x}) \leq t\}.$$
Let $d^\circ$ be the differential of $CM^\circ$. Our first task is to show that $d^\circ$ respects the $\R$-filtration given by $\ell_{\eta}$.

\begin{Prop}\label{Prop5.1}
    Let $\alpha_\pm \in \mathcal{R}(Y, \mathfrak{s}, g, \eta)$ that are the asymptotic limits of the gauge equivalence class of some solution $(B, \Phi)$ of the Seiberg-Witten equations \eqref{eq:4.4} on $Y\times \R$. Then $\mathcal{L}_\eta(\alpha_+) \leq \mathcal{L}_\eta(\alpha_-)$ for $\eta$ small enough.
\end{Prop}

\begin{proof}
    For $T\in \R$, we denote $W_T = Y \times [-T,T]$, which is a manifold with boundary $-Y \cup Y$. Consider $\mu$ to be a closed purely imaginary two-form on $Y\times \R$ such that $\mu$ restricts to each $Y$-slice is $\eta$. Note that in the case of a finite cylinder, the mean curvature $H$ of the boundary vanishes. Furthermore, the metric on a cylinder can be chosen such that it has uniform positive scalar curvature \cite{gromov2022perspectives}. Let $A_0$ be a fixed referenced connection on $Y$. We have
    \begin{align}
        &\int_{W_T} (F_B - \mu) \wedge (F_B -\mu)  = \int_{W_T}F_B \wedge F_B - 2\int_{W_T} F_B \wedge \mu + \int_{W_T} \mu \wedge \mu \nonumber\\
        & = \int_{Y\times T}(B|_{Y\times T}-A_0) \wedge (F_{B|_{Y\times T}}+F_{A_0}-\eta) \nonumber \\
        &- \int_{Y\times -T} (B|_{Y\times -T} - A_0) \wedge (F_{B|_{Y\times -T}}+F_{A_0}-\eta) - 2\int_{W_T} F_B \wedge \mu + \int_{W_T} \mu \wedge \mu \nonumber \\
        & + \int_{Y \times T} (B|_{Y\times T}-A_0)\wedge \eta - \int_{Y \times -T}(B|_{Y\times -T}-A_0)\wedge \eta. \nonumber
    \end{align}
    As a result, the topological energy $\mathscr{E}^{top}_{W_T,\mu}(B,\Phi)$ can be written as
    \begin{align}\label{eq:5.1}
        &\dfrac{1}{4}\int_{W_T}(F_B - \mu) \wedge (F_B -\mu) - \dfrac{1}{2} \int_{W_T}\la \Phi, \rho_{W_T}(\mu^+) \Phi \ra   \\
        &+ \int_{Y\times -T} \la \Phi_{Y\times -T}, D_{B|_{Y\times -T}}\Phi|_{Y\times -T}\ra - \int_{Y \times T}\la \Phi_{Y\times T}, D_{B|_{Y\times T}}\Phi|_{Y\times T}\ra  \nonumber  \\
        &= \dfrac{1}{4} \int_{Y\times T}(B|_{Y\times T}-A_0) \wedge (F_{B|_{Y\times T}}+F_{A_0}-\eta) - \int_{Y \times T}\la \Phi_{Y\times T}, D_{B|_{Y\times T}}\Phi|_{Y\times T}\ra \nonumber \\
        &-\dfrac{1}{4}\int_{Y\times -T} (B|_{Y\times -T} - A_0) \wedge (F_{B|_{Y\times -T}}+F_{A_0}-\eta) + \int_{Y\times -T} \la \Phi_{Y\times -T}, D_{B|_{Y\times -T}}\Phi|_{Y\times -T}\ra \nonumber \\
        &- \dfrac{1}{2}\int_{W_T}F_B \wedge \mu + \dfrac{1}{4}\int_{W_T} \mu \wedge \mu + \dfrac{1}{4}\int_{Y} (B|_{Y\times T}-B|_{Y\times -T})\wedge \eta - \dfrac{1}{2}\int_{W_T}\la \Phi, \rho_{W_T}(\mu^+)\Phi\ra. \nonumber 
    \end{align}
    Let $\epsilon >0$. By \eqref{eq:5.1}, Lemma \ref{Lem4.3}, and Lemma \ref{Lem4.4}, if $\norm{\eta} < \epsilon$, then
    \begin{align}\label{eq:5.2}
        0 < \lim_{T\to \infty} \mathscr{E}^{an}_{W_T, \mu}(B,\Phi) = \lim_{T\to \infty} \mathscr{E}^{top}_{W_T, \mu}(B, \Phi) \leq 2(\mathcal{L}_{\eta}(\alpha_-) - \mathcal{L}_{\eta}(\alpha_+)) + O(1)\epsilon.
    \end{align}
    The estimate \eqref{eq:5.2} implies that for a generic $\eta$ small enough, we do get $\mathcal{L}_{\eta}(\alpha_+) \leq \mathcal{L}_{\eta}(\alpha_-)$ as desired.
\end{proof}

From Proposition \ref{Prop5.1}, we immediate see that for any $\mathbf{x} \in CM^{\circ t}(Y, m_\mathfrak{s}, g, \eta; \mathbf{F})_{\ast}$, then $\ell_{\eta}(d^{\circ}\mathbf{x}) \leq \ell_\eta (\mathbf{x}) \leq t$, provided that $\eta$ is small enough. Thus, $(CM^{\circ t}(Y, m_\mathfrak{s}, g, \eta; \mathbf{F}), d^{\circ t})$ is a complex in its own right. Here $d^{\circ t}$ denotes the restriction of $d^\circ$ to $CM^{\circ t}$. With this, for $\eta$ small enough, we define
$$HM^{\circ t}(Y, m_\mathfrak{s}, g, \eta; \mathbf{F}) := H(CM^{\circ t}(Y, m_\mathfrak{s}, g, \eta; \mathbf{F})).$$
The natural inclusion $\iota^t: CM^{\circ t} \to CM^{\circ}$ induces a homomorhism $\iota^{t}_{\ast} : H^{\circ t}(Y, m_\mathfrak{s}, g, \eta; \mathbf{F}) \to HM^{\circ}(Y, m_\mathfrak{s}, g, \eta; \mathbf{F})$. Recall from Corollary \ref{Cor4.11}, for $m_{\mathfrak{s}} \geq 0$, if $q \leq -1$ or $q \geq 2m_\mathfrak{s}$, $HM^\circ_q$ is an invariant of $Y$. For such an allowable range of $q$, we define the following numerical quantities induced from the $CM^{\circ}$ and the $\R$-filtration.

\begin{Def}\label{Def5.2}
    Let $Y$ be a rational homology three-sphere that is smooth, closed, connected, and oriented. Suppose $g$ is a fixed Riemannian metric of $Y$, and $\mathfrak{s}$ is a $spin^c$ structure on $Y$. Let $m_\mathfrak{s} \in \mathfrak{m}(Y)$ be a chamber such that $m_{\mathfrak{s}} \geq 0$, and $\eta \in i\Omega^2(Y)$ that is closed and generic. For $q \leq -1$ or $q \geq 2m_{\mathfrak{s}}$, we define
    $$\rho_q(Y, m_{\mathfrak{s}}, g) = \lim_{\norm{\eta} \to 0} \inf\{t : \iota^t_\ast: \widetilde{HM}^t_q(Y, m_\mathfrak{s}, g, \eta; \mathbf{F}) \to \widetilde{HM}_q(Y, m_\mathfrak{s}, g, \eta; \mathbf{F}) \text{ non-trivial}\}.$$
    Similarly, we define $\lambda_q(Y,g)$ associated with $HM_q$. 
\end{Def}

Note that $\rho_q(Y,g), \lambda_q(Y,g)$ take values in $\R \cup \{+\infty\}$. Adopting the labeling convention at the beginning of the subsection, we use $\rho^{\circ}_q(Y,m_\mathfrak{s}, g)$ to refer to $\rho_q(Y,m_\mathfrak{s}, g)$ or $\lambda_q(Y, m_\mathfrak{s}, g)$, depending on whether $\circ = \emptyset$ or $\sim$.   Just from the definition, we immediately have the following observation.

\begin{Prop}\label{Prop5.3}
    With the same set-up as in Definition \ref{Def5.2}, if $\rho^\circ_q(Y, m_\mathfrak{s}, g)$ is finite, then $HM^\circ_q(Y, m_\mathfrak{s}; \mathbf{F})$ is non-trivial. 
\end{Prop}

We end this subsection by showing that

\begin{Th}\label{Th5.4}
    With the same set-up as above with the additional assumption that $\mathfrak{s}$ is a torsion $spin^c$ structure, i.e., $c_1(\det \mathfrak{s})^2 = 0$, then  $\rho^{\circ}_q(Y,m_\mathfrak{s}, g)$ is an invariant of $(Y,g)$.
\end{Th}

\begin{proof}
    The proof is similar to the proof of Proposition \ref{Prop5.1} and works for either version $HM, \widetilde{HM}$. Let $\eta_1, \eta_2$ be two generic small enough perturbations. We have two auxiliary data $(g,\eta_1) , (g,\eta_2)$ and consider the continuity map between the chain complexes induced by the cylinder cobordism $W_T = Y \times [-T, T]$. Let $W^*$ be obtained from $W_T$ by attaching a cylinder $Y \times \R^+$ to its boundary. Pick a two-form $\mu$ on $W^* = Y \times \R$ such that $\mu|_{Y \times -T} = \eta_1$ and $\mu|_{Y\times T} = \eta_2$. Let $(B,\Phi)$ be a solution of the $\mu$-perturbed four-dimensional Seiberg-Witten equations on $Y \times \R$ such that $\lim_{T\to \pm \infty}(B,\Phi)|_{Y \times T} = (A_\pm, \psi_\pm)$, where $(A_\pm, \psi_\pm)$ is a solution of the $\eta_2, \eta_1$-perturbed Seiberg-Witten equations on $Y$, respectively. Let $A_0$ be a fixed referenced connection. Following the calculations that derive \eqref{eq:5.1}, in this setting, we also have
    \begin{align}\label{eq:5.3}
        & 2(\mathcal{L}_{\eta_1}(B|_{Y\times -T}, \Phi|_{Y\times -T}) - \mathcal{L}_{\eta_2}(B|_{T\times T}, \Phi|_{Y\times T}))  \nonumber \\
        & = \mathscr{E}^{top}_{W_T,\mu}(B, \Phi) -\int_{W_T}\dfrac{-1}{2} F_B \wedge \mu + \dfrac{1}{4}\mu \wedge \mu \nonumber \\
        &- \dfrac{1}{4}\int_Y (B|_{Y\times T} - A_0)\wedge \eta_1 - (B|_{Y\times -T}-A_0)\wedge \eta_2 + \dfrac{1}{2}\int_{W_T} \la \Phi, \rho_{W_T}(\mu^+)\Phi\ra. 
    \end{align}
    Let $T \to \infty$, from \eqref{eq:5.3}, we obtain
    \begin{align}\label{eq:5.4}
        & 2(\mathcal{L}_{\eta_1}(A_-, \psi_-) - \mathcal{L}_{\eta_2}(A_+,\psi_+))  \nonumber \\
        & = \lim_{T \to \infty} \mathscr{E}^{an}_{W_T ,\mu}(B, \Phi) - \int_{Y\times \R} \dfrac{-1}{2} F_B \wedge \mu + \dfrac{1}{4}\mu\wedge \mu \nonumber \\
        & -\dfrac{1}{4}\int_Y (A_+ -A_0) \wedge \eta_1 - (A_- - A_0)\wedge \eta_2 + \dfrac{1}{2}\int_{Y\times \R}\la \Phi, \rho_{Y\times \R}(\mu^+)\Phi\ra \nonumber \\
        & \geq \dfrac{1}{4}\int_{Y\times \R}|F_B -\mu|^2 + \int_{Y\times \R}|\nabla_B \phi|^2 + \dfrac{1}{4}\int_{Y\times \R}|\phi|^4 + \dfrac{1}{2}\int_{Y\times \R}F_B \wedge \mu - \dfrac{1}{4}\mu \wedge \mu \nonumber \\
        & -\dfrac{1}{4}\int_Y (A_+ -A_0) \wedge \eta_1 - (A_- - A_0)\wedge \eta_2 + \dfrac{1}{2}\int_{Y\times \R}\la \Phi, \rho_{Y\times \R}(\mu^+)\Phi\ra.
    \end{align}
    By Chern-Weil theory, the estimate \eqref{eq:5.4} implies that
    \begin{align*}
        \mathcal{L}_{\eta_1}(A_-,\psi_-) \geq &\mathcal{L}_{\eta_2}(A_+,\psi_+) + \pi^*_Y c_1(det\, \mathfrak{s})^2 \nonumber \\
        &-\dfrac{1}{4}\int_Y (A_+ -A_0) \wedge \eta_1 - (A_- - A_0)\wedge \eta_2 + \dfrac{1}{2}\int_{Y\times \R}\la \Phi, \rho_{Y\times \R}(\mu^+)\Phi\ra \\
        & = \mathcal{L}_{\eta_2}(A_+,\psi_+) -\dfrac{1}{8}\int_Y (A_+ -A_0) \wedge \eta_1 - (A_- - A_0)\wedge \eta_2 \\
        &  + \dfrac{1}{4}\int_{Y\times \R}\la \Phi, \rho_{Y\times \R}(\mu^+)\Phi\ra.
    \end{align*}
    Let $\epsilon > 0$. Thus, if $\mathcal{L}_{\eta_1}(A_-,\psi_-) \leq t$ and $\norm{\eta_1}, \norm{\eta_2} < \epsilon$, then by Lemma \ref{Lem4.3} and Lemma \ref{Lem4.4},
    \begin{align}\label{eq:5.5}
        \mathcal{L}_{\eta_2}(A_+,\psi_+) &\leq t + \dfrac{1}{8}\int_Y (A_+ -A_0) \wedge \eta_1 - (A_- - A_0)\wedge \eta_2 - \dfrac{1}{4}\int_{Y\times \R}\la \Phi, \rho_{Y\times \R}(\mu^+)\Phi\ra \nonumber \\
        &\leq t + O(1)\epsilon.
    \end{align}

    By \eqref{eq:5.5}, we have the following commutative diagram for $q \leq 2m^1_\mathfrak{s}-1$ or $q \geq 2 m^2_{\mathfrak{s}}$, where $m^{i}_{\mathfrak{s}} = I(g,\eta_i)$
    \[
    \begin{tikzcd}
        HM^{\circ}_q(Y, m^1_\mathfrak{s}; \mathbf{F}) \arrow[r] & HM^{\circ}_q(Y, m^2_{\mathfrak{s}};\mathbf{F}) \\
        HM^{\circ t}_q(Y, m^1_\mathfrak{s}; \mathbf{F}) \arrow[u] \arrow[r] & HM^{\circ t + O(1)\epsilon}_q(Y, m^2_{\mathfrak{s}};\mathbf{F}) \arrow[u]
    \end{tikzcd}
    \]
    Here recall that the top map is an isomorphism induced by the continuation chain map between chain complexes. If the left vertical map is non-trivial, then the right vertical map is also non-trivial. For $\eta_i$ small enough, we define $\rho^{\circ}_q(Y, m^i_\mathfrak{s}, g, \eta_i)$ to be the infimum of all $t$ such that $\iota^t_\ast : HM^{\circ t} \to HM^{\circ}$ is non-trivial. Hence, $\rho^{\circ}_q(Y, m^2_\mathfrak{s},g,\eta_2) \leq \rho^{\circ}_q(Y, m^1_\mathfrak{s}, g, \eta_1) + O(1)\epsilon$. Similarly, we also have $\rho^{\circ}_q(Y, m^1_\mathfrak{s},g,\eta_2) \leq \rho^{\circ}_q(Y, m^2_\mathfrak{s}, g, \eta_1) + O(1)\epsilon$. Therefore, $$|\rho^\circ_q(Y, m^1_\mathfrak{s}, g, \eta_1) - \rho^\circ_q(Y, m^2_\mathfrak{s}, g, \eta_2)| \leq O(1)\epsilon.$$
    The above inequality shows that $\rho^\circ_q(Y, m_\mathfrak{s}, g)$ is an invariant of $(Y,g)$ for the allowable range of $q$ as claimed.
\end{proof}

\begin{Rem}\label{Rem5.5}
From this point on-ward, we refer to $\rho^{\circ}_q(Y, m_\mathfrak{s}, g)$ as the \textit{spectral invariant} associated with $HM^\circ_q(Y, m_\mathfrak{s}; \mathbf{F})$. In general, it is not true that if we also vary $g$, then an analogous version of the spectral invariant defined above is an invariant of $Y$. Roughly, this is because the Chern-Simons-Dirac functional's definition depends on an \textit{a priori} choice of background metric, and the $\R$-degree of the induced homomorphism at the filtered level depends on the geometry of the cobordism. This is a point of contrast to various similar invariants defined on the instanton side (see, e.g., \cite{MR4158669,Nozaki2023}).
\end{Rem}

\subsection{Relationship with cobordism}\label{Sub5.2}
In this subsection, we investigate how the cobordism affects the $\R$-filtration level of $HM^{\circ}$. For now, we assume that $Y$ is a smooth, closed, oriented, not necessarily three-manifold. Suppose $W$ is a compact $4$-manifold with boundary $\partial W = Y$. Let $W^*$ be a non-compact manifold by attaching a cylinder $Y\times [0,\infty)$ to the boundary of $W$. As before, we let $\mathfrak{s}_W$ be a $spin^c$ structure on $W$, and $\mathfrak{s} = \mathfrak{s}_W|_{Y}$ is the $spin^c$ structure on $Y$. Denote the extension of $\mathfrak{s}_W$ to $W^*$ by the same label. Let $\eta \in i\Omega^2(Y)$ that is closed and generic, $\mu \in i\Omega^2(W^*)$ such that $\mu|_{Y\times [0,\infty)}$ is the pulled-back of $\eta$ via the projection map $Y\times [0,\infty)$. Additionally, let $g_{W^*}$ be a metric on $W^*$ such that on $Y \times [0,\infty)$, it is of the form $g\times 1$, where $g$ is a metric on $Y$. 

Let $W_T = W^* \setminus Y\times [T,\infty)$. Note that $W_T$ is now a compact manifold with cylindrical ends and its boundary is $Y$. Denote $\mathfrak{s}_{W_T}$ by the restriction of the $\mathfrak{s}_W$ on $W_T$. Suppose $B_0$ is a fixed connection on $W^*$ and the restriction of $B_0$ to each $Y$-slice in the cylindrical boundary is $A_0$.  We start with the following possibly well-known lemma.

\begin{Lemma}\label{Lem5.6}
     Let $(B, \Phi)$ be a configuration on $W_T$. Then there exists a constant $C(B_0, \mathfrak{s}_W)$ depending on $B_0, \mathfrak{s}_{W}$ such that
     \begin{align*}
     \mathscr{E}^{top}_{W_T, \mu}(B,\Phi) &= -2\mathcal{L}_\eta (B|_{\partial W_T}, \Phi|_{\partial W_T}) + \dfrac{1}{4}\int_Y (B|_{\partial W_T} - A_0)\wedge \eta -\dfrac{1}{2}\int_{W_T}\la \Phi, \rho_{W_T}(\mu^+)\Phi\ra \\
     & - \dfrac{1}{2}\int_{W_T}F_B \wedge \mu + \dfrac{1}{4}\int_{W_T}\mu \wedge \mu +C.
     \end{align*}
\end{Lemma}

\begin{proof}
    Without perturbation, by formula (4.21) on page 98 of \cite{MR2388043}, there exists a $C(B_0, \mathfrak{s}_W)$ such that
    $$\mathscr{E}^{top}_{W_T}(B,\Phi) = -2 \mathcal{L}(B|_{\partial W_T}, \Phi|_{\partial W_T}) + C.$$
    Expand the above formula, and we have
    \begin{align}\label{eq:5.6}
        \dfrac{1}{4}\int_{W_T} F_B \wedge F_B -& \int_Y \la \Phi|_Y, D_{B|_Y}\Phi|_Y\ra + \int_{Y} \left(\dfrac{H}{2}\right) | \Phi|_Y|^2 = \nonumber\\
        & \dfrac{1}{4}\int_Y (B|_Y - A_0)\wedge (F_{B|_Y}+F_{A_0}) - \int_Y \la \Phi|_Y, D_{B|_Y}\Phi|_Y\ra + C.
    \end{align}
    Now, with perturbation, from \eqref{eq:5.6}, we obtain
    \begin{align}\label{eq:5.7}
        &\dfrac{1}{4}\int_{W_T}(F_B - \mu) \wedge (F_B - \mu) + \dfrac{1}{2}\int_{W_T} F_B\wedge \mu - \dfrac{1}{4}\int_{W_T} \mu \wedge \mu \nonumber \\
        & - \dfrac{1}{2}\int_{W_T} \la \Phi, \rho_{W_T}(\mu^+)\Phi\ra -\int_Y \la \Phi|_Y, D_{B|_Y}\Phi|_Y\ra + \int_Y \left(\dfrac{H}{2}\right)|\Phi|_Y|^2 \nonumber \\
        & = \dfrac{1}{4}\int_Y (B|_Y - A_0) \wedge (F_{B|_Y}+F_{A_0}- \eta) - \int_Y \la \Phi|_Y, D_{B|_Y}\Phi|_Y \ra \nonumber\\
        & + \dfrac{1}{4}\int_Y (B|_Y - A_0) \wedge \eta - \dfrac{1}{2}\int_{W_T}\la \Phi, \rho_{W_T}(\mu^+)\Phi\ra + C.
    \end{align}
    As a result, \eqref{eq:5.7} can be rewritten as the formula of the lemma as desired.
\end{proof}

Now, suppose that $Y = -Y_1 \cup Y_2$ with the set-up as at the beginning of subsection \ref{Sub4.3}. A consequence of the above lemma is the following proposition.

\begin{Prop}\label{Prop5.7}
     Let $(B,\Phi)$ be a solution of the Seiberg-Witten equations \eqref{eq:4.5} on $W^*$ such that $\lim_{T\to \pm \infty}(B,\Phi)|_{Y\times T} = (A_\pm, \psi_\pm)$, where $(A_\pm, \psi_\pm)$ is a solution of the corresponding perturbed Seiberg-Witten equations on $Y_1, Y_2$, respectively. Let $s$ be the scalar curvature on $W$. For $\eta_1, \eta_2$ small enough, then there exists a $C(B_0,\mathfrak{s}_W)$ such that
     $$-\dfrac{1}{32}\int_W s^2 \leq \mathcal{L}_{\eta_1}(A_-,\psi_-) - \mathcal{L}_{\eta_2}(A_+,\psi_+)+C.$$
\end{Prop}

\begin{proof}
    By the formula in Lemma \ref{Lem5.6} and let $T \to \infty$, we have
    \begin{align*}
        \mathscr{E}^{an}_{\mu,W^*}(B, \Phi) & = -2\mathcal{L}_{\eta}(A,\psi) + \dfrac{1}{4}\int_Y (A-A_0)\wedge \eta - \dfrac{1}{2}\int_{W^*}\la \Phi, \rho_{W^*}(\mu^+)\Phi\ra \\
        & - \dfrac{1}{2}\int_{W^*}F_B \wedge \mu + \dfrac{1}{4}\int_{W^*}\mu \wedge \mu +C.
    \end{align*}
    In the case that $Y = -Y_1 \cup Y_2$ as in the hypothesis of the Proposition \ref{Prop5.7}, we let $B_0|_{Y_1}=A_0, B_0|_{Y_2}=A'_0$. Then the above formula expands to become
    \begin{align}\label{eq:5.8}
        \dfrac{1}{2}\mathscr{E}^{an}_{W^*, \mu} &= \mathcal{L}_{\eta_1}(A_-,\psi_-) - \mathcal{L}_{\eta_2}(A_+,\psi_+) + \dfrac{1}{4}\left(\int_{Y_2} (A_+ - A'_0)\wedge \eta_2 - \int_{Y_1}(A_- -A_0)\wedge \eta_1 \right) \nonumber \\
        & - \dfrac{1}{2}\int_{W^*}\la \Phi, \rho_{W^*}(\mu^+)\Phi\ra - \dfrac{1}{2}\int_{W^*}F_B \wedge \mu \nonumber \\
        & + \dfrac{1}{4}\int_{W^*} \mu \wedge \mu + C.
    \end{align}
    Recall from \eqref{eq:4.8}, we have the following obvious estimate
    \begin{align}\label{eq:5.9}
        \mathscr{E}^{an}_{W^*, \mu}(B,\Phi) &= \lim_{T \to \infty} \left(\dfrac{1}{4}\norm{F_B - \mu}^2_{L^2}+\norm{\nabla_B \Phi}^2_{L^2}+\dfrac{1}{4}\norm{|\Phi|^2+s/2}^2_{L^2}-\norm{s/4}^2_{L^2}\right) \nonumber \\
        & \geq -\dfrac{1}{16}\int_{W} s^2.
    \end{align}
    Let $\epsilon >0$, for $\norm{\eta_1}, \norm{\eta_2} < \epsilon$, combine \eqref{eq:5.8} and \eqref{eq:5.9}, we obtain
    $$-\dfrac{1}{32}\int_W s^2 \leq \mathcal{L}_{\eta_1}(A_-,\psi_-) - \mathcal{L}_{\eta_2}(A_+,\psi_+) + C + O(1)\epsilon.$$
    The above estimate tells us that for $\eta_i$ small enough, we have the estimate in the proposition as claimed.
\end{proof}

In preparation for the main theorem of this subsection, we introduce the following definition. Recall in Subsection \ref{Sub4.3}, we set $d(W) = (c_1(det\,\mathfrak{s}_W)^2-\sigma(W))/4 + b_1(W) - b^+(W)$, and $k = m_{\mathfrak{s}_1} - m_{\mathfrak{s}_2} - d(W)/2$, where $m_{\mathfrak{s}_i}$ is a chamber associated with the $spin^c$ structure $\mathfrak{s}_i$ on $Y_i$.

\begin{Def}\label{Def5.8}
    Let $W: Y_1 \to Y_2$ a cobordism between smooth, closed, connected oriented rational homology three-spheres. For $b^+(W) \leq 1$ and $k\leq -1$, we say that $W$ is an \textit{injective cobordism} if the induced map
    $$CM^{\circ}_\ast : HM^{\circ}_q(Y_1,m_{\mathfrak{s}_1}; \mathbf{F}) \to HM^{\circ}_{q-d(W)}(Y_2, m_{\mathfrak{s}_2}; \mathbf{F})$$
    is injective when $q \leq 2m_{\mathfrak{s}_1}-1$ or $q \geq 2m_{\mathfrak{s}_2}$.
\end{Def}

\begin{Th}\label{Th5.9}
    Let $W : Y_1 \to Y_2$ be a cobordism between smooth, closed, connected, oriented rational homology three-sphere. Suppose $W$ is spin, and let $\mathfrak{s}_W$ be a spin structure on $W$ that, when restricted to $Y_i$, gives a spin structure $\mathfrak{s}_i$ on $Y_i$. Let $g$ be a metric on $W$ and denote $s$ by its scalar curvature. If $W$ is an injective cobordism, then for $q \leq 2m_{\mathfrak{s}_1}-1$ or $q \geq 2m_{\mathfrak{s}_2}$ we have 
    $$\dfrac{1}{32}\int_{W} s^2 \geq \rho^{\circ}_{q - d(W)}(Y_2; m_{\mathfrak{s}_2}, g_2) - \rho^{\circ}_{q}(Y_1, m_{\mathfrak{s}_1}, g_1) - C.$$
    Here $g_i$ is the induced metric on the boundary components of $Y$, $C$ is a constant that only depends on the canonical (trivial) $spin^c$ connection $B_0$ on $W$ and the spin structure $\mathfrak{s}_W$.
\end{Th}

\begin{proof}
    Let $\alpha_- \in \mathcal{R}(Y_1,\mathfrak{s}_1,g_1,\eta_1)$ and $\alpha_+ \in \mathcal{R}(Y_2, \mathfrak{s}_2, g_2, \eta_2)$, where $\eta_i$ is small enough and $I(g_i, \eta_i) = m_{\mathfrak{s}_i}$. By Proposition \ref{Prop5.7}, if $\mathcal{L}_{\eta_1}(\alpha_-) \leq t$, then
    $$\mathcal{L}_{\eta_2}(\alpha_+) \leq t + \dfrac{1}{32}\int_{W}s^2 + C.$$
    As a result, at the filtered level, the cobordism $W$ induces a homomorpshim $CM^{\circ t}(W)_\ast : HM^{\circ t}_q(Y_1, m_{\mathfrak{s}_1}; \mathbf{F}) \to H^{\circ t+\int_W s^2/32 + C}_{q - d(W)}(Y_2, m_{\mathfrak{s}_2}; \mathbf{F})$. This map fits into the following commutative diagram
    \[
    \begin{tikzcd}
        HM^{\circ}_q(Y_1, m_{\mathfrak{s}_1}; \mathbf{F}) \arrow[r] & HM^{\circ}_{q-d(W)}(Y_2, m_{\mathfrak{s}_2};\mathbf{F}) \\
        HM^{\circ t}_q(Y, m_{\mathfrak{s}_2}; \mathbf{F}) \arrow[u] \arrow[r] & HM^{\circ t + \int_{W}s^2/32 + C}_{q-d(W)}(Y_2, m_{\mathfrak{s}_2};\mathbf{F}) \arrow[u]
    \end{tikzcd}
    \]
    Recall that by the hypothesis, the top map is injective. Thus, if the left vertical map is non-trivial, then the right vertical map is also non-trivial. As a result, 
    $$\rho^{\circ}_{q-d(W)}(Y_2, m_{\mathfrak{s}_2}, g_2) \leq \rho^{\circ}_q(Y_1, m_{\mathfrak{s}_1}, g_1) + C + \dfrac{1}{32}\int_{W} s^2.$$
    After re-arranging the above inequality, we arrive at the claim of the theorem.
\end{proof}

From Theorem \ref{Th5.9}, we see that the spectral invariants give us a lower bound for the scalar curvature functional on a certain type of $4$-manifold with boundary. Under an appropriate assumption of the geometry of the cobordism, a refinement of the estimate \eqref{eq:5.9} allows us to say more. 

\begin{Th}\label{Th5.10}
    Let $W: Y_1 \to Y_2$ be a cobordism with the same conditions as in Theorem \ref{Th5.9}. If $W$ is an injective cobordism and the metric $g$ on $W$ has positive scalar curvature, then for $q \leq 2m_{\mathfrak{s}_1}-1$ or $q \geq 2m_{\mathfrak{s}_2}$ we have
    $$\rho^{\circ}_{q-d(W)}(Y_2, m_{\mathfrak{s}_2},g_2) \leq \rho^\circ_q(Y_1, m_{\mathfrak{s}_1}, g_1) +C.$$
    Here, the constant $C$ depends on the canonical trivial $spin^c$ connection $B_0$ on $W$ and the spin structure $\mathfrak{s}_W$.
\end{Th}

\begin{proof}
    If the scalar curvature $s$ of $g$ is always positive, then \eqref{eq:5.9} can be modified in the simplest way as follows
    \begin{align*}
        \mathscr{E}^{an}_{W^*, \mu}(B,\Phi) &= \lim_{T \to \infty} \left(\dfrac{1}{4}\norm{F_B - \mu}^2_{L^2}+\norm{\nabla_B \Phi}^2_{L^2}+\dfrac{1}{4}\norm{|\Phi|^2+s/2}^2_{L^2}-\norm{s/4}^2_{L^2}\right)  \\
        & \geq \dfrac{1}{4}\norm{s/2}^2_{L^2}- \norm{s/4}^2_{L^2} = 0.
    \end{align*}
    As a result, if $\alpha_- \in \mathcal{R}(Y_1, \mathfrak{s}_1, g_1, \eta_1)$ and $\alpha_+ \in \mathcal{R}(Y_2, \mathfrak{s}_2, g_2, \eta_2)$ with $\norm{\eta_i}< \epsilon$ and there is a gauge equivalence class of a solution $(B,\Phi)$ of \eqref{eq:4.5} that is asymptotic to $\alpha_\pm$ at the ends, then 
    $$0 \leq \mathcal{L}_{\eta_1}(\alpha_-) - \mathcal{L}_{\eta_2}(\alpha_+) + C + O(1)\epsilon.$$
    Hence if $\eta_i$ is small enough, we must have $\mathcal{L}_{\eta_2}(\alpha_+) \leq \mathcal{L}_{\eta_1}(\alpha_-) + C$. From this, we see that the cobordism $W$ induces the following homomorphism $CM^{\circ t}(W)_\ast : HM^{\circ t}_{q}(Y_1, m_{\mathfrak{s}_1}; \mathbf{F}) \to H^{\circ t + C}_{q-d(W)}(Y_2, m_{\mathfrak{s}_2}; \mathbf{F})$. Now, using a similar commutative diagram as in the proof of Theorem \ref{Th5.9} and remembering that $W$ is an injective cobordism, we arrive at the result as claimed.
\end{proof}

Theorem \ref{Th5.10} gives us an obstruction of the existence of a positive scalar curvature on a cobordism in terms of the spectral invariants of the boundary components. Specifically, from the above theorem, we immediately have the following corollary.

\begin{Cor}\label{Cor5.11}
    Let $W : Y_1 \to Y_2$ be a cobordism with the same conditions as in Theorem \ref{Th5.9}. If $W$ is an injective cobordism and for $q \leq 2m_{\mathfrak{s}_1}-1$ or $q \geq 2m_{\mathfrak{s}_2}$ we have $\rho^{\circ}_{q-d(W)}(Y_2,m_{\mathfrak{s}_2}, g_2) > \rho^{\circ}_{q}(Y_1, m_{\mathfrak{s}_1}, g_1) + C$, then $W$ can never have a metric with positive scalar curvature.
\end{Cor}

\begin{Rem}
    The definition of injective cobordism seems algebraic for now, in the sense that for us to arrive at the various statements of Theorem \ref{Th5.9} and Theorem \ref{Th5.10}, we need the top map in the commutative diagram that appears in the proof of theorems to be injective. In a later section, we will show this definition is entirely geometric and natural. In particular, we prove that ribbon rational homology cobordism is an example of such an injective cobordism.
\end{Rem}

\begin{Rem}
    Note that the statements of Theorem \ref{Th5.9}, Theorem \ref{Th5.10}, and Corollary \ref{Cor5.11} are for both spectral invariants $\rho_q$ and $\lambda_q$ (cf. Definition \ref{Def5.2}). In the next subsection, we explicitly provide a relationship between these spectral invariants under some assumptions about the $u$-map of the $\mathcal{S}$-complex of monopoles.
\end{Rem}

\subsection{Comparison of the spectral invariants associated with $HM$ and $\widetilde{HM}$}\label{Sub5.3}

We recall the algebraic set-up of the category $\mathscr{S}$ of $\mathcal{S}$-complexes (cf. Subsection \ref{Sub3.4}). Let $\mathbf{F}$ denote the field of characteristic 2. Consider the following $\mathcal{S}$-complex $\widetilde{C}_\ast = C_\ast \oplus \mathbf{F} \oplus C_{\ast-1}$ over $\mathbf{F}$, where $(C,d)$ is a chain complex that is $\mathbf{Z}$-graded and the degree of the differential $d$ is 1. Let $\widetilde{d}$ be the differential of $\widetilde{C}$ of degree $1$, which in matrix form written as the following
$$\widetilde{d} = \begin{bmatrix} d & 0 & 0 \\ \delta_1 & 0 & 0 \\ u & \delta_2 & d \end{bmatrix}.$$
Here $\delta_1 : C_\ast \to \mathbf{F}$, $\delta_2: \mathbf{F} \to C_\ast$, and $u : C_\ast \to C_{\ast-2}$ are maps such that $\widetilde{d}^2 = 0$. There is a finite filtration of $\widetilde{C}$ that is given as follows
$$\mathscr{F}^0 \widetilde{C}_\ast = 0 \subset \mathscr{F}^1 \widetilde{C}_\ast = C_{\ast-1} \subset \mathscr{F}^2\widetilde{C}_\ast= \mathbf{F} \oplus C_{\ast -1} \subset \mathscr{F}^3\widetilde{C}_\ast = \widetilde{C}_\ast.$$
The differential for each sub-complex is induced by $\widetilde{d}$, which respectively is
$$\mathscr{F}^0 \widetilde{d} = 0, \quad  \mathscr{F}^1 \widetilde{d} = d, \quad \mathscr{F}^2 \widetilde{d} = \begin{bmatrix} 0 & 0 \\ \delta_2 & d \end{bmatrix}, \quad \mathscr{F}^3 \widetilde{d} = \widetilde{d}.$$

Since the above filtration is of finite length, the associated spectral sequence $\mathscr{E}^{r}_{p,q}$ has a limit page. In what follows, we explicitly write down the limit page $\mathscr{E}^\infty_{p,q}$ of $\mathscr{E}^r_{p,q}$. We start with the zeroth page $\mathscr{E}^0_{p,q} = \mathscr{F}^p \widetilde{C}_{p+q}/ \mathscr{F}^{p-1} \widetilde{C}_{p+q}$. Since the filtration only goes from $0, \cdots, 3$, we only have to consider 
$$\mathscr{E}^0_{1,q} = \mathscr{F}^{1} \widetilde{C}_{1+q} = C_{q}, \quad \mathscr{E}^0_{2,q} = \mathscr{F}^2 \widetilde{C}_{2+q}/\mathscr{F}^1 \widetilde{C}_{2+q} = \mathbf{F},$$
$$\mathscr{E}^0_{3,q} = \mathscr{F}^3 \widetilde{C}_{3+q}/\mathscr{F}^2 \widetilde{C}_{3+q} = C_{q+3}.$$
It is not difficult to see that the differential of the complexes in the first, second, and third columns of the zeroth page is $d, 0, d$, respectively. Next, consider the first page $\mathscr{E}^1_{p,q} = H(\mathscr{E}^0_{p,q})$. Similarly, we only have to write down $\mathscr{E}^1_{p,q}$ for $p = 1, 2, 3$:
$$\mathscr{E}^1_{1,q} = H_q(C), \quad \mathscr{E}^1_{2,q} = \mathbf{F}, \quad \mathscr{E}^1_{3,q} = H_{q+3}(C).$$
Each $q^{th}$ row of the first page is a complex that is given by
\[
\begin{tikzcd}
H_q(C)  & \arrow[l, "\delta_{2\ast}"] \mathbf{F}  & \arrow[l, "\delta_{1\ast}"] H_{q+3}(C),
\end{tikzcd}
\]
where $\delta_{i\ast}$ is the induced map from the chain maps $\delta_i$ above. For the second page $\mathscr{E}^2_{p,q} = H(\mathscr{E}^1_{p,q})$, once again, we only consider those with $p = 1, 2, 3$:
$$\mathscr{E}^2_{1,q} = \dfrac{H_q(C)}{im( \delta_{2\ast}: \mathbf{F} \to H_q(C))}, \quad \mathscr{E}^2_{2,q} = \dfrac{ ker( \delta_{2\ast}: \mathbf{F} \to H_q(C))}{ im (\delta_{1\ast}: H_{q+3}(C) \to \mathbf{F})},$$
$$\mathscr{E}^2_{3,q} = ker (\delta_{1\ast}: H_{q+3}(C) \to \mathbf{F}).$$
The differential $\mathscr{E}^2_{p,q} \to \mathscr{E}^2_{p-2, q+1}$ of the second page is only non-trivial for $p=3$. As a result, the complexes of the second page are given by
\[
\begin{tikzcd}
    \mathscr{E}^2_{3,q} = ker (\delta_{1\ast}: H_{q+3}(C) \to \mathbf{F}) \arrow[r, "u_\ast"] & \dfrac{H_{q+1}(C)}{im( \delta_{2\ast}: \mathbf{F} \to H_{q+1}(C))} = \mathscr{E}^2_{1, q+1},
\end{tikzcd}
\]
where $u_\ast$ is the induced map from the map $u : C_\ast \to C_{\ast - 2}$ above. Finally, we look at the third page $\mathscr{E}^3_{p,q} = H(\mathscr{E}^2_{p,q})$. The third page is supported for $p= 1, 2, 3$. With that in mind, we have
$$\mathscr{E}^3_{1,q} = \dfrac{\mathscr{E}^2_{1,q}}{im\, u_\ast|_{ker\, \delta_{1\ast}}}, \quad \mathscr{E}^3_{2,q} = \mathscr{E}^2_{2,q} = \dfrac{ker \, \delta_{2\ast}}{\im \, \delta_{1\ast}}, \quad \mathscr{E}^3_{3,q} = ker\, u_{\ast}|_{ker\, \delta_{1\ast}} .$$

\begin{Lemma}\label{HomAlgelem1}
    For every $r \geq 3$, the spectral sequence $\mathscr{E}^r_{p,q}$ abuts to $\mathscr{E}^\infty_{p,q}$. As a result, $\mathscr{E}^\infty_{p,q} = \mathscr{E}^3_{p,q}$. Thus, $\mathscr{E}^2_{p,q} \Rightarrow H(\widetilde{C})$.
\end{Lemma}

\begin{proof}
    Since a differential of a complex on the third page is given by $\mathscr{E}^3_{p,q} \to \mathscr{E}^3_{p-3,q+2}$ and any given page is supported in the first three columns, all differentials are trivial. Subsequently, for any $r \geq 3$, $\mathscr{E}^r_{p,q} = H(\mathscr{E}^3_{p,q}) = \mathscr{E}^3_{p,q}$. This proves the claim.
\end{proof}

Denote $\mathscr{F}^i H(\widetilde{C})$ by the induced finite filtration on the homology of the $\mathcal{S}$-complex $(\widetilde{C}, \widetilde{d})$. Since $\mathscr{E}^3_{p,q} = \mathscr{E}^\infty_{p,q} \cong \mathscr{F}^p H(\widetilde{C}_{p+q})/\mathscr{F}^{p-1}H(\widetilde{C}_{p+q})$ by Lemma \ref{HomAlgelem1}, we have the following (non-canonical) isomorphism
\begin{align}\label{eq:isomoofhomologyofscomplex}
    H(\widetilde{C}_{3+q}) &\cong \dfrac{\mathscr{F}^3 H(\widetilde{C}_{3+q})}{\mathscr{F}^2 H(\widetilde{C}_{3+q})} \oplus \dfrac{\mathscr{F}^2 H(\widetilde{C}_{3+q})}{\mathscr{F}^1 H(\widetilde{C}_{3+q})} \oplus \mathscr{F}^1 H(\widetilde{C}_{3+q}) \nonumber \\
    & \cong \mathscr{E}^3_{3,q} \oplus \mathscr{E}^3_{2,q+1}\oplus \mathscr{E}^3_{1,q+2} \nonumber \\
    & \cong ker\, u_{\ast}|_{ker\, \delta_{1\ast}} \oplus \dfrac{ker \, \delta_{2\ast}}{\im \, \delta_{1\ast}} \oplus \dfrac{\mathscr{E}^2_{1,q+2}}{im\, u_\ast|_{ker\, \delta_{1\ast}}}
\end{align}

Suppose we have a non-archimedean norm $\ell$ on $(C,d)$ that induces an $\R$-filtration $\widetilde{C}^t$ on $(\widetilde{C},\widetilde{d})$ where $\widetilde{d}$ respects this $\R$-filtration. The discussion above also applies to each subcomplex $\widetilde{C}^t$ of the $\mathcal{S}$-complex $\widetilde{C}$. As a result, after re-indexing, we obtain
\begin{align}\label{eq:filteredisomofhomologyofscomplex}
    H^t_q(\widetilde{C}) \cong ker\, u^t_{\ast}|_{ker\, \delta^t_{1\ast}} \oplus \dfrac{ker \, \delta^t_{2\ast}}{\im \, \delta^t_{1\ast}} \oplus \dfrac{H^t_{q-1}(C)/ im (\delta^t_{2\ast}: \mathbf{F} \to H^t_{q-1}(C))}{im\, u^t_\ast|_{ker\, \delta^t_{1\ast}}},
\end{align}
where $H^t(\widetilde{C})$ is the induced $\R$-filtration on $H(\widetilde{C})$.

With the $\R$-filtration on $C$ and $\widetilde{C}$, we can define the following spectral-like invariant of each of the respective chain complexes.

\begin{Def}[cf. Definition \ref{Def5.2}]\label{2spectralinvariants}
    For each $q \in \mathbf{Z}$, we define $\rho_q(\widetilde{C})$ to be the infimum of all $t \in \R$ such that the induced map of $\iota^t : \widetilde{C}^t_q \to \widetilde{C}_q$ is non-trivial at the homology level. Similarly, we define 
    $$\lambda_q(C) = \inf \{t \in \R : \iota^t_\ast : H^t_q(C) \to H_q(C) \text{ is non-trivial}\}.$$
\end{Def}

The following propositions describes a relationship between $\lambda_q(C)$ and $\rho_q(\widetilde{C})$.

\begin{Prop}\label{Prop1stinequality}
    Suppose $(C,d)$ is a $\mathbf{Z}$-graded chain complex equipped with a nonarchimedean norm $\ell: C \to \R \cup \{-\infty\}$. Let $(\widetilde{C},\widetilde{d})$ be the associated $\mathcal{S}$-complex of $(C,d)$, where the differential $\widetilde{d}$ is defined by the maps $\delta_1: C_\ast \to \mathbf{F}$, $\delta_2: \mathbf{F} \to C_\ast$, $u : C_\ast \to C_{\ast - 2}$, and $d$ that satisfy
    $$\delta_{1} d = 0, \quad d\delta_2 =0, \quad ud + du + \delta_2\delta_1 = 0.$$ 
    If $\delta_{2\ast} = 0$ and $u_\ast : ker(\delta_{1\ast}: H_{q+1}(C) \to \mathbf{F}) \to H_{q-1}(C)$ is trivial, then 
    $$\lambda_{q-1}(C) \geq \rho_q(\widetilde{C}).$$
\end{Prop}

\begin{proof}
    By \eqref{eq:isomoofhomologyofscomplex} and \eqref{eq:filteredisomofhomologyofscomplex} and the hypothesis of the proposition, we have the following commutative diagram
    \begin{equation*}
\begin{tikzcd}
H_q(\widetilde{C})   & \arrow[l, "\varphi"] ker\, u_{\ast}|_{ker\, \delta_{1\ast}} \oplus \dfrac{ker \, \delta_{2\ast}}{\im \, \delta_{1\ast}} \oplus H_{q-1}(C) \\
H^t_q(\widetilde{C})  \arrow[u, "\iota^t_\ast"] & \arrow[l, "\varphi^t"] ker\, u^t_{\ast}|_{ker\, \delta^t_{1\ast}} \oplus \dfrac{ker \, \delta^t_{2\ast}}{\im \, \delta^t_{1\ast}} \oplus H^t_{q-1}(C) \arrow[u, "\iota^t_\ast"]
\end{tikzcd}
    \end{equation*}
Here $\varphi$ is an isomorphism. Suppose $\lambda_{q-1}(C) < \infty$ and let $t:=\lambda_{q-1}(C)$. By definition, $\iota^t_\ast : H^t_{q-1}(C) \to H_{q-1}(C)$ is non-trivial. As a result, the right-hand side vertical map of the above diagram is also non-trivial. Therefore, by the commutative of the above diagram and the fact that $\varphi$ is an isomorphism, we must have $\iota^t_\ast : H^t_q(\widetilde{C}) \to H_q(\widetilde{C})$ non-trivial also. Hence, $\lambda_{q-1}(C) = t \geq \rho_q(\widetilde{C})$ as claimed.
\end{proof}

\begin{Prop}\label{Prop2ndinequality}
    Suppose $(C,d)$ is a $\mathbf{Z}$-graded chain complex and $(\widetilde{C},\widetilde{d})$ is the associated $\mathcal{S}$-complex defined as in Proposition \ref{Prop1stinequality}. If $\delta_{1\ast} = 0$ and $u_\ast : H_{q-3}(C) \to H_{q-5}(C)/ im(\delta_{2\ast}: \mathbf{F} \to H_{q-5}(C))$ is trivial, then
    $$\lambda_{q-3}(C) \geq \rho_q(\widetilde{C}).$$
\end{Prop}

\begin{proof}
    The proof is similar to the proof of Proposition \ref{Prop1stinequality} by considering the following commutative diagram 
    \begin{equation*}
\begin{tikzcd}
H_q(\widetilde{C})   & \arrow[l, "\varphi"] H_{q-3}(C) \oplus \dfrac{ker \, \delta_{2\ast}}{\im \, \delta_{1\ast}} \oplus \dfrac{H_{q-1}(C)}{ im (\delta_{2\ast}: \mathbf{F} \to H_{q-1}(C))} \\
H^t_q(\widetilde{C})  \arrow[u, "\iota^t_\ast"] & \arrow[l, "\varphi^t"] H^t_{q-3}(C) \oplus \dfrac{ker \, \delta^t_{2\ast}}{\im \, \delta^t_{1\ast}} \oplus \dfrac{H^t_{q-1}(C)}{ im (\delta^t_{2\ast}: \mathbf{F} \to H^t_{q-1}(C))} \arrow[u, "\iota^t_\ast"]
\end{tikzcd}
    \end{equation*}
    and the fact that if $H^t_{q-3}(C) \to H_{q-3}(C)$ is non-trivial, then the right vertical map of the diagram is also non-trivial.
\end{proof}

We apply the general formalism above to our specific setting of $\mathcal{S}$-complex of monopoles.

\begin{Prop}[Proposition 16 of Section 12, \cite{MR2738582}]\label{Prop5.18}
    Let $Y$ be a rational homology three-sphere equipped with a $spin^c$ structure $\mathfrak{s}$. Let $m_{\mathfrak{s}} \in \mathfrak{m}(Y)$ be a chamber. For $HM_\ast(Y, m_\mathfrak{s}; \mathbf{F})$, we always have either $\delta_{1\ast}$ or $\delta_{2\ast}$ vanishes. 
\end{Prop}

By Proposition \ref{Prop5.18}, the first condition in the hypothesis of either Proposition \ref{Prop1stinequality} or Proposition \ref{Prop2ndinequality} is satisfied. Thus, a way to compare $\rho_q(Y, m_{\mathfrak{s}}, g)$ and $\lambda_q(Y, m_{\mathfrak{s}}, g)$ is to assume some addition information about the induced $u$-maps. In particular, we have

\begin{Th}\label{Th5.19}
    Let $Y$ be a rational homology three-sphere equipped with a $spin^c$ structure $\mathfrak{s}$. Suppose $m_{\mathfrak{s}} \in \mathfrak{m}(Y)$ is a chamber of $Y$ and $m_{\mathfrak{s}} \geq 0$. Fix a metric $g$ on $Y$ associated with $m_{\mathfrak{s}}$. For each $q \leq -1$ or $q \geq 2m_{\mathfrak{s}}$, 
    \begin{enumerate}
        \item If $\delta_{2\ast} = 0$ and $u_\ast : ker(\delta_{1\ast} : HM_{q+1}(Y, m_\mathfrak{s}; \mathbf{F}) \to \mathbf{F}) \to HM_{q-1}(Y, m_\mathfrak{s}; \mathbf{F})$ is trivial, then $\lambda_{q-1}(Y, m_{\mathfrak{s}}, g) \geq \rho_q (Y, m_{\mathfrak{s}}, g)$.
        \item If $\delta_{1\ast} = 0$ and $u_{\ast} : HM_{q-3}(Y, m_{\mathfrak{s}}; \mathbf{F}) \to HM_{q-5}(Y, m_{\mathfrak{s}}; \mathbf{F})/im(\delta_{2\ast})$ is trivial, then $\lambda_{q-3}(Y, m_\mathfrak{s}, g) \geq \rho_{q}(Y, m_\mathfrak{s},g)$.
    \end{enumerate}
\end{Th}

\begin{Rem}\label{Rem5.20}
    At this moment, we are not sure if there is an example of $(Y,g)$ where the vanishing conditions on $u_\ast$ in Theorem \ref{Th5.19} are satisfied. However, it is possible to relax such conditions on the $u_\ast$-map and obtain an analogous relationship between the spectral invariants $\rho, \lambda$ by appealing to the general framework of barcodes that appear in, e.g., Section 7 of \cite{MR4368349}. 
\end{Rem}

\section{Ribbon homology cobordism}\label{Sec6}
Recall from Definition \ref{Def5.8} in Subsection \ref{Sub5.2}, we say that $W: Y_1 \to Y_2$ is an injective cobordism between rational homology three-spheres if the induced map $CM(W)_\ast$ is injective. In this section, we show that a ribbon rational homology cobordism is an example of an injective cobordism. In general, if $W : Y_1 \to Y_2$ is a cobordism and we wish to show that $CM(W)_\ast : HM(Y_1) \to HM(Y_2)$ is injective, a straightforward strategy is to show that $CM(W)_\ast$ has a left-inverse. Taking advantage of the TQFT formalism, we consider the $D(W) := -W \cup_{Y_2} W : Y_1 \to Y_1$, which is called the \textit{double of} $W$ (cf. Section 5 of \cite{MR4467148}). If we can show that $CM(D(W))$ is either exactly the identity or chain homotopic to the identity map at the level of chain complex, then $CM(W)_\ast$ must be injective. The point of considering a ribbon homology cobordism $W$ is that there is a nice topological description of $D(W)$ in terms of surgery that allows us to prove $CM(D(W)) \sim$ the identity map. Once the basic definition and properties of ribbon cobordism are set up, the proof of the injectivity assertion is an application of gluing theory.

\subsection{Basic definition and properties}
We first begin with the definition of ribbon cobordism. Most of what is presented here follows Section 5 of \cite{MR4467148}. In general, a $2n$-dimensional cobordism is ribbon if and only if it has a handle decomposition that contains only $k$-handles, where $k\leq n$. In particular, a $4$-dimensional ribbon cobordism can be constructed by attaching only $1$- and $2$-handles. Here is one way a ribbon cobordism between $3$-manifolds can show up naturally: Let $K_1$ and $K_2$ be ribbon concordance knots inside $S^3$, i.e., there is a surface $S$ inside $S^3 \times I$ such that $S : K_1 \to K_2$ is a cobordism with only $0$- and $1$-handles. Define $W = S^3 \times I \setminus S$. Then $W : S^3\setminus K_1 \to S^3 \setminus K_2$ is a ribbon \textit{homology} cobordism. Recall that being a homology cobordism means that the inclusion map from each boundary component into $W$ induces an isomorphism at the level of homology. Essentially, $W$ has no topology in the interior detected by homology. Note that if $W$ is a ribbon homology cobordism, then it has the same number of $1$- and $2$-handles.

As mentioned previously, if $W: Y_1 \to Y_2$ is a ribbon cobordism, then there is a nice topological description of $D(W)$.

\begin{Prop}[Proposition 5.1, \cite{MR4467148}]\label{Prop6.1}
    Suppose $Y_1, Y_2$ are compact $3$-manifolds. Let $W: Y_1 \to Y_2$ be a ribbon cobordism with $m$ $1$-handles and $\ell$ $2$-handles. Then $D(W)$ can be described by a surgery on $X \cong (Y_1 \times I) \# m(S^1\times S^3)$ along $\ell$ disjoint simple closed curves $\gamma_1, \cdots, \gamma_\ell$. 

    If $W$ is a ribbon $\mathbf{Q}$-homology cobordism and we write
    $$[\gamma_i] = \sigma_i + \sum_{j=1}^{m} c_{ij} \alpha_j,$$
    where $\sigma_i \in H_1(Y_1), \alpha_j \in H_1(X)$ is the homology class of the core of the $j^{th}$ $S^1\times S^3$ summand, then $(c_{ij})\otimes_\mathbf{Z} \mathbf{Q}$ is invertible over $\mathbf{Q}$. Furthermore, 
    $$[\gamma_1]\wedge \cdots \wedge [\gamma_m] = \det(c_{ij})\, \alpha_1\wedge \cdots \wedge \alpha_m \in (\Lambda^* (H_1(X)/Tors)/\la H_1(Y)/Tors\ra)\otimes_{\mathbf{Z}} \mathbf{Q}.$$
\end{Prop}

We give a sketch of proof for the first part of the above proposition. The idea is rather straightforward. Firstly, taking advantage of the fact that $W$ is a ribbon cobordism, we decompose $D(W)$ as follows
$$D(W) = W_1 \cup W_2 \cup -W_2 \cup -W_1,$$
where $W_1 : Y_1 \to Y_1 \# m(S^1 \times S^2)$ is a cobordism given by $1$-handles attachment and $W_2 : Y_1 \# m(S^1\times S^2) \to Y_2$ is a cobordism given by $2$-handles attachment. It can be shown that $W_2 \cup -W_2$ is the result of a surgery on $Y_1\# m(S^1\times S^2) \times [-1,1]$ along some $\ell$ attaching circles of the $2$-handles. On the other hand, $W_1 \cup -W_1$ can be shown to be diffeomorphic to $(Y_1 \times I) \# m(S^1\times S^3)$. As a result, by comparing $D(W)$ and $W_1 \cup -W_1$, one arrives at the surgery result.

\begin{Rem}\label{Rem6.2}
    The result of Proposition \ref{Prop6.1} is true for any field other than $\mathbf{Q}$. However, in this present paper, since we are mainly interested in rational homology situations, we state the result in $\mathbf{Q}$ for applications later.
\end{Rem}

\subsection{Ribbon homology cobordism and monopole Floer homology}
Given a cobordism $W : Y_1 \to Y_2$, where $Y_i$ is a closed, oriented smooth rational homology $3$-sphere and $\gamma \subset int(W)$ is a closed loop in the interior of $W$. Consider the following surgery on $W$ along $\gamma$ as follows: We remove a tubular neighborhood $\nu(\gamma) \cong S^1 \times D^3$ and glue back $D^2 \times S^2$ to $W_0 := W\setminus \nu(\gamma)$ along the boundary $S^1 \times S^2$ to obtain $W'$, which is also viewed as a cobordism from $Y_1$ to $Y_2$. Note that
$$W = W_0 \cup_{S^1\times S^2} S^1 \times D^3, \quad W' = W_0 \cup_{S^1 \times S^2} D^2 \times S^2.$$
We would like to understand how such surgery on $W$ along $\gamma$ affects the induced map at the level of monopole Floer homology. Roughly, this amounts to understanding how the moduli space of finite energy monopoles on $W$ and $W'$ are factored through a $1$-dimensional moduli space of finite energy monopoles on $W_0$ via gluing theory. 

To this end, firstly, let's consider the case of closed manifolds. Suppose for now that $W$ is closed. Let $\mathfrak{s}_W$ be a $spin^c$ structure on $W$ and fix a metric $g$ on $W$. Consider $\mathcal{B}(W, \mathfrak{s}_W)$ to be configuration space of the $\mu$-perturbed Seiberg-Witten equations on $W$, where $\mu$ is some generic purely imaginary $2$-form on $W$. Since $\mu$ is chosen to be generic, the moduli space $\mathscr{M}(W, \mathfrak{s}_W)$ of the perturbed Seiberg-Witten equations on $W$ contains no reducible and is a compact finite oriented manifold. The associated fundamental class $[\mathscr{M}(W, \mathfrak{s}_W)]$ is an element of $H_\ast (\mathcal{B}(W, \mathfrak{s}_W))$. For the sake of the exposition, we assume that $\dim \mathscr{M}(W, \mathfrak{s}_W) = 1$. Now, we let $\mathbb{L}$ be a $U(1)$-bundle over $\mathcal{B}(W,\mathfrak{s}_W) \times W$ induced by the pulled-back of the determinant line bundle of the $spin^c$ structure over $W$. Denote by $c_1(\mathbb{L}) \in H^2(\mathcal{B}(W, \mathfrak{s}_W)\times W)$ by the first chern class of $\mathbb{L}$. Via K\"unneth decomposition, given $\gamma \in H_1(W)$, we obtain $s(\gamma):= c_1(\mathbb{L})/\gamma \in H^{1}(\mathcal{B}(W,\mathfrak{s}_W))$. Then the pairing
$$\la s(\gamma), [\mathscr{M}(W, \mathfrak{s}_W)] \ra \in \mathbf{Z}.$$
Equivalently, the pairing above can be interpreted as the degree of the holonomy map around $\gamma$ from $\mathscr{M}(W, \mathfrak{s}_W) \to S^1$. More generally, if $\dim \mathscr{M}(W, \mathfrak{s}_W) = k$ so that $[\mathscr{M}(W,\mathfrak{s}_W)]\in H_k(\mathcal{B}(W,\mathfrak{s}_W))$, then one can still get a numerical number out of this situation by considering $\gamma_1 \wedge \cdots \gamma_k \in \Lambda^k H_1(W)$ which gives the pairing
$$\la s(\gamma_1) \cup \cdots \cup s(\gamma_2), [\mathscr{M}(W, \mathfrak{s}_W)]\ra \in \mathbf{Z}.$$
Note that if $b^+(W) > 1$, the pairing defined above is nothing but the classical Seiberg-Witten invariants (cf. e.g., \cite{MR2492194} for more details on this interpretation of the Seiberg-Witten invariants). 

An analogous construction as the above carries over in the case where $W$ is a compact $4$-manifold with boundary. Now we are assuming that $W$ is a compact manifold with boundary, where the boundary is not necessarily connected. As in Subsection \ref{Sub4.1}, we denote by $Y$  the boundary of $W$ and let $\mathscr{M}(W;\alpha)$ be the moduli space of gauge equivalence classes of solutions of \eqref{eq:4.5} that are asymptotic to $\alpha \in \mathcal{R}(Y, \mathfrak{s}, g, \eta)$. If $\dim \mathscr{M}(W;\alpha) = 0$, we obtain the induced chain map $CM(W) : CM(Y) \to \mathbf{F}$ of degree $-d(W)$ (cf. \eqref{eq:4.22}). If $\gamma \subset int(W)$ a closed loop, then using the holonomy map similarly defined in the closed case, we have another chain map $CM(W, \gamma) : CM(Y) \to \mathbf{F}$ given by 
$$CM(W, \gamma)(\alpha) = \la s(\gamma), [\mathscr{M}(W; \alpha)]\ra \mod 2,$$
whenever $\dim \mathscr{M}(W; \alpha) = 1$. Note that it is not difficult to see that changing $\gamma$ by another loop in the same homology class produces a chain homotopic chain map. Thus, the result of the induced map at the homology level only depends on $[\gamma] \in H_1(W)$. In other words, for each $[\gamma] \in H_1(W)$, we have a homomorphism $CM(W, [\gamma])_\ast : HM(Y) \to \mathbf{F}$. 

The following proposition should be compared with the instanton versions in, e.g., Proposition 6.5 of \cite{MR4467148} and Theorem 7.16 of \cite{MR1883043}.

\begin{Prop}\label{Prop6.3}
    Suppose $W$ is a smooth compact oriented $4$-manifold with boundary $Y$, where $Y$ is not necessarily connected. Let $\mathfrak{s}_W$ be a $spin^c$ structure on $W$ and $\mathfrak{s}$ is its restriction to $Y$. Suppose $\gamma \subset int(W)$ is a closed loop and $W'$ is the result of the surgery on $W$ along $W$. We extend $\mathfrak{s}_W$ over $W'$ trivially and denote by $\mathfrak{s}_{W'}$ the induced $spin^c$ structure. Then over $\mathbf{F}$, we must have
    $$CM(W', \mathfrak{s}_{W'}) = CM(W, \gamma, \mathfrak{s}_W).$$
    This implies that $CM(W', \mathfrak{s}_{W'})_\ast = CM(W, [\gamma], \mathfrak{s}_W)_\ast$ at the level of homology.
\end{Prop}

\begin{proof}
    Firstly, note that we may view $W$ as a cobordism from $Y \to \emptyset$ that is decomposed by $W_0 \cup_{S^1\times S^2} S^1\times D^3$, where $W_0$ is a cobordism from $Y \to S^1\times S^2$ and $S^1 \times D^2$ is a cobordism from $S^1\times S^2 \to \emptyset$. Let $\alpha$ be a monopole on $Y$. We would like to understand how $\mathscr{M}(W;\alpha)$ is factored through the respective moduli spaces on $W_0$ and $S^1\times D^3$. We attach a cylinder to $W_0$ and $S^1\times D^3$ along $S^1\times S^2$. The result manifolds are denoted by $W^*_0$ and $(S^1\times D^3)^*$, respectively.
    
    Since $S^1 \times S^2$ has positive scalar curvature, the moduli space $\mathscr{M}(S^1\times S^2)$ contains only reducible monopoles, and thus can be identified with $T^{b_1(S^1\times S^2)}= S^1$. This puts us in a Morse-Bott situation. So, we choose a Morse-Smale function on $\mathscr{M}(S^1\times S^2) = S^1$ that has two critical points $\theta_1, \theta_2$. Now, on $(S^1 \times D^3)^*$ we put a metric of positive scalar curvature. Thus, the moduli space of finite energy monopoles on $(S^1\times D^3)^*$ contains only reducible solutions and also can be identified with the Jacobian torus $T^{b_1(S^1\times D^3)} = S^1$. Let $\mathscr{M}(W_0; \alpha, S^1)$ be the moduli space of monopoles on $W^*_0$ that are asymptotic to $\alpha$ toward $-\infty$ and to an element of $\mathscr{M}(S^1\times S^1)$ toward $+\infty$. Similarly, we define $\mathscr{M}(S^1\times D^3; S^1)$. The asymptotic maps give us a map
    $$\mathscr{M}(W_0; \alpha, S^1) \times \mathscr{M}(S^1\times D^3; S^1) \to S^1 \times S^1$$
    that is transversal to the diagonal of $S^1\times S^1$. Thus, we obtain $\mathscr{M}(W; \alpha) \cong \mathscr{M}(W_0; \alpha, S^1) \times_{S^1} S^1$. As a result, $\dim \mathscr{M}(W;\alpha) = \dim \mathscr{M}(W_0; \alpha, S^1)$. We are only interested in the case where $\mathscr{M}(W;\alpha) = 1$, and since any monopoles inside $\mathscr{M}(W_0; \alpha, S^1)$ limit to either $\theta_1$ or $\theta_2$ at $+\infty$, we choose whichever that results in a dimension-one moduli space, say, $\mathscr{M}(W_0; \alpha, \theta_1)$.

    Let $\gamma'$ be a parallel circle with $\gamma$ in $W_0$. By functoriality, we have
    \begin{align*}
    CM(W, \gamma, \mathfrak{s}_W)(\alpha) &= CM(S^1\times D^3, \gamma) CM(W_0, \gamma')(\alpha)\\
    & = CM(S^1\times D^3, \gamma) \left( \la s(\gamma'), \mathscr{M}(W_0; \alpha, \theta_1) \ra \theta_1 \right) \\
    & = \la s(\gamma'), \mathscr{M}(W_0; \alpha, \theta_1) \ra CM(S^1\times D^3, \gamma) \theta_1.
    \end{align*}
    Since $\theta_1$ can only be extended to the trivial flat connection on $S^1\times D^3$, we obtain 
    $$CM(W, \gamma, \mathfrak{s}_W)(\alpha) = \la s(\gamma'), \mathscr{M}(W_0; \alpha, \theta_1) \ra.$$

    We move on to $W' = W_0 \cup_{S^1\times S^2} D^2 \times S^2$, which is another cobordism from $Y \to \emptyset$. It is not hard to see that $\chi(W') = \chi(W) + 2$. As a result, $\dim \mathscr{M}(W'; \alpha) = \dim \mathscr{M}(W; \alpha) - 1 = 0$. Note that the moduli space of finite energy monopoles on $D^2\times S^2$ is just a single point. As a result, similar to the argument above, we see that $\mathscr{M}(W',\alpha)\cong \mathscr{M}(W_0; \alpha, \theta_1) \times_{S^1} \{\text{point}\}$. Thus, $CM(W',\mathfrak{s}_{W'})(\alpha) = \la s(\gamma'), \mathscr{M}(W_0; \alpha, \theta_1)\ra$. Finally, change $\gamma$ by another representative in the homology class $[\gamma]\in H_1(W)$ produces chain homotopy maps. Therefore, we arrive at the result as claimed.
\end{proof}

We are now ready to prove one of the main theorems of this section. The following theorem is regarding $HM$.

\begin{Th}\label{Th6.4}
    Suppose $W: Y_1 \to Y_2$ is a ribbon rational homology cobordism, where $Y_i$ is a smooth, closed, oriented rational homology three-sphere. Let $\mathfrak{s}_W$ be a $spin^c$ structure on $W$ and its restriction to $Y_i$ is denoted by $\mathfrak{s}_i$. Let $m_{\mathfrak{s}_i} \in \mathfrak{m}(Y_i)$ be a chamber. Then, the induced cobordism map $CM(D(W))_{\ast} : HM(Y_1, m_{\mathfrak{s}_1}; \mathbf{F}) \to HM(Y_1, m_{\mathfrak{s}_2}; \mathbf{F})$ satisfies
    $$CM(D(W))_{\ast} = |H_1(W, Y_1)| CM(Y_1 \times I)_{\ast}.$$
    The map on the right-hand side is some non-zero scalar multiple of the identity map on $HM(Y_1, m_{\mathfrak{s}_1}; \mathbf{F})$. Consequently, $CM(W)_\ast$ is injective.
\end{Th}

\begin{proof}
    Let $m$ be the number of $1$-handles in $W$. By Proposition \ref{Prop6.3}, we know that $D(W)$ can be described as a surgery on $X\cong (Y_1 \times I) \# m (S^1\times S^3)$ along $m$ disjoint simple closed curves $\gamma_1, \cdots , \gamma_m$. Then, by Proposition \ref{Prop6.3}, we have
    $$CM(D(W))_\ast = CM(X, [\gamma_1]\wedge\cdots \wedge [\gamma_2])_\ast = \det (c_{ij}) CM(X, \alpha_1\wedge \cdots \wedge \alpha_m)_\ast,$$
    where $\alpha_j \in H_1(X)$ is the homology class of the core of the $j^{th}$ $S^1\times S^3$ summand, and $|\det(c_{ij})| = |H_1(W, Y_1)|$, which is non-zero. On the other hand, $Y_1 \times I$ can be viewed as a result of surgery on $X$ along $\alpha_1, \cdots, \alpha_m$. Thus, by Proposition \ref{Prop6.3}, we also have $CM(Y_1 \times I)_\ast = CM(X, \alpha_1 \wedge\cdots \wedge \alpha_2)_\ast$.
    Therefore, by combining the two formulae derived above, we arrive at the result as desired.
\end{proof}

This next theorem is regarding the homology of the $\mathcal{S}$-complex of monopoles and its interaction with rational homology cobordism.

\begin{Th}\label{Th6.5}
    Suppose $Y_1, Y_2$ are rational homology three-spheres and $W : Y_1 \to Y_2$ is a ribbon rational homology cobordism. Fix a $spin^c$ structure $\mathfrak{s}_W$ on $W$ such that its restriction to $Y_i$ is denoted by $\mathfrak{s}_i$. Let $m_{\mathfrak{s}_i} \in \mathfrak{m}(Y_i)$ be a chamber. Then $\widetilde{CM}(D(W)) : \widetilde{CM}(Y_1, m_{\mathfrak{s}_1}; \mathbf{F}) \to \widetilde{CM}(Y_2, m_{\mathfrak{s}_2}; \mathbf{F})$ is $\mathcal{S}$-chain homotopy equivalent to an $\mathcal{S}$-isomorphism. In particular, $\widetilde{CM}(W)_\ast$ is an injective homomorphism.
\end{Th}

\begin{proof}
    Following the same argument as in the proof of Theorem \ref{Th6.4}, we see that there is a chain homotopy $L : CM(Y_1, m_{\mathfrak{s}_1}; \mathbf{F}) \to CM(Y_1, m_{\mathfrak{s}_1}; \mathbf{F})$ that satisfies $L \, d + d \, L = CM(D(W)) - \text{id}_{CM(Y_1)}$. Now define $h : \widetilde{CM}(Y_1) \to \widetilde{CM}(Y_1)$ to be
    $$h = \begin{bmatrix} L & 0 & 0 \\ 0 & L & 0 \\ 0 & 0 & 0 \end{bmatrix}.$$
    A straightforward calculation shows that $h$ is an $\mathcal{S}$-chain homotopy between $\widetilde{CM}(D(W))$ with
    $$ \begin{bmatrix}
        \text{id}_{CM(Y_1)} & 0 & 0 \\ \star & \text{id}_{CM(Y_1)} & \star \\ \star & 0 & 1
    \end{bmatrix}.$$
    Since the above matrix is clearly invertible over $\mathbf{F}$, this proves the claim.
\end{proof}

With Theorem \ref{Th6.4} and Theorem \ref{Th6.5} combined with Theorem \ref{Th5.10} and Corollary \ref{Cor5.11}, using the spectral invariants, we obtain some information about the the geometry of ribbon homology cobordism.

\begin{Th}\label{Th6.6}
    Let $W : Y_1 \to Y_2$ be a ribbon rational cobordism between smooth, closed, connected, oriented rational homology three-sphere. Suppose $W$ is spin, and let $\mathfrak{s}_W$ be a spin structure on $W$ that, when restricted to $Y_i$, gives a spin structure $\mathfrak{s}_i$ on $Y_i$. Let $g$ be a metric on $W$ and denote $s$ by its scalar curvature. Then for allowable range of degree $q$ (cf. Corollary \ref{Cor4.11}) we have 
    $$\dfrac{1}{32}\int_{W} s^2 \geq \rho^{\circ}_{q - d(W)}(Y_2; \mathfrak{s}_2, g_2) - \rho^{\circ}_{q}(Y_1, \mathfrak{s}_1, g_1) - C.$$
    Here $g_i$ is the induced metric on the boundary components of $Y$, $d(W)$ is a topological constant of $W$, $C$ is a constant that only depends on the canonical (trivial) $spin^c$ connection $B_0$ on $W$ and the spin structure $\mathfrak{s}_W$.

    Furthermore, if we have $\rho^{\circ}_{q-d(W)}(Y_2,\mathfrak{s}_2, g_2) > \rho^{\circ}_{q}(Y_1, \mathfrak{s}_1, g_1) + C$, then $W$ can never have a metric with positive scalar curvature.
\end{Th}



\bibliography{refs}
\bibliographystyle{amsplain}
\Addresses
	
	\end{document}